\documentclass[twoside,a4paper,reqno,11pt]{amsart} 
\usepackage[top=28mm,right=28mm,bottom=28mm,left=28mm]{geometry}

\usepackage{amsfonts, amsmath, amssymb, mathrsfs, bm, latexsym, stmaryrd, array, hyperref, mathtools, bold-extra}

\usepackage[pdftex]{color,graphicx}

\renewcommand{\O}{\Omega}

\newcommand{\F}{\mathbb{F}_{q}}
\newcommand{\Nat}{\mathbb{N}}

\newcommand{\R}{\mathbb{R}}

\newcommand{\la}{\langle}
\newcommand{\ra}{\rangle}

\renewcommand{\to}{\rightarrow}

\newcommand{\leqs}{\leqslant}
\newcommand{\geqs}{\geqslant}
\newcommand{\normeq}{\trianglelefteqslant}

\newcommand{\vs}{\vspace{2mm}}
\newcommand{\fpr}{\mbox{{\rm fpr}}}

\newcommand{\Sym}{\operatorname{Sym}}

\makeatletter
\newcommand{\imod}[1]{\allowbreak\mkern4mu({\operator@font mod}\,\,#1)}
\makeatother

\newtheorem{theorem}{Theorem} 
\newtheorem*{theorem*}{Theorem}
\newtheorem*{conj*}{Conjecture}

\newtheorem{thm}{Theorem}[section] 
\newtheorem{prop}[thm]{Proposition} 
\newtheorem{lem}[thm]{Lemma}
\newtheorem{cor}[thm]{Corollary}

\theoremstyle{definition}
\newtheorem{rem}[thm]{Remark}

\newtheorem*{deff}{Definition}
\newtheorem{defn}[thm]{Definition}

\begin{document}

\title[A generalisation of Cameron's base size conjecture]{A generalisation of Cameron's base size conjecture}

\author{Marina Anagnostopoulou-Merkouri}
\address{M. Anagnostopoulou-Merkouri, School of Mathematics, University of Bristol, Bristol BS8 1UG, UK}
\email{marina.anagnostopoulou-merkouri@bristol.ac.uk}
\begin{abstract}
Let $G\leqs {\rm Sym}(\O)$ be a finite transitive permutation group with point stabiliser $H$. A
base for $G$ is a subset of $\O$ whose pointwise stabiliser is trivial, and the minimal cardinality of a base is called the base size of $G$, denoted by $b(G, \O)$. Equivalently, $b(G, \O)$ is the minimal positive integer $k$ such that $G$ has a regular orbit on the Cartesian product $\O^k$. A well-known conjecture of
Cameron from the 1990s asserts that if $G$ is an almost simple primitive group and $H$ is a so-called non-standard subgroup, then $b(G, \O) \leqs 7$, with equality if and only if $G$ is the Mathieu group ${\rm M}_{24}$ in its natural action of degree $24$. This conjecture was settled in a series of papers by Burness et al. (2007-11). 

In this paper, we complete the proof of a natural generalisation of Cameron's conjecture. Our main result states that if $G$ is an almost simple group and $H_1, \ldots, H_k$ are any non-standard maximal subgroups of $G$ with $k \geqs 7$, then $G$ has a regular orbit on $G/H_1 \times \cdots \times G/H_k$, noting that Cameron's original conjecture corresponds to the special case where the $H_i$ are pairwise conjugate subgroups. In addition, we show that the same conclusion holds with $k = 6$, unless $G = {\rm M}_{24}$ and each $H_i$ is isomorphic to ${\rm M}_{23}$. For example, this means that if $G$ is a simple exceptional group of Lie type and $H_1, \ldots, H_6$ are proper subgroups of $G$, then there exist elements $g_i \in G$ such that $\bigcap_i H_i^{g_i} = 1$. By applying recent work in a joint paper with Burness, we may assume $G$ is a group of Lie type and our proof uses  probabilistic methods based on fixed point ratio estimates.
\vs

\end{abstract}

\date{\today}

\maketitle

\setcounter{tocdepth}{1}
\tableofcontents

\section{Introduction}\label{s:intro}
Let $G \leqs \Sym(\O)$ be a transitive permutation group on a finite set $\Omega$. Recall that a \emph{base} for $G$ is a subset of $\O$ whose pointwise stabiliser is trivial. We write $b(G, \O)$ for the \emph{base size} of $G$, which is the minimal cardinality of a base for $G$. This classical invariant has been extensively studied since the early days of group theory in the nineteenth century, finding a wide range of applications to other areas of mathematics such as graph theory (see~\cite{BC}) and model theory (indicatively, see~\cite{Cherlin, kelsey-colva, lachlan}), for example. Bases are also fundamental in the computational study of finite groups and the implementation of efficient algorithms (see~\cite{Seress}). There is an extensive literature on bases for permutation groups and we refer the reader to the survey articles \cite{BC,LSh3,Maroti} and \cite[Section 5]{Bur181} for further details. 

Determining the base size of a permutation group is a fundamental problem in group theory and it has attracted significant interest in recent years. In particular, there has been a strong focus on studying base sizes of primitive groups, motivated by some highly influential conjectures due to Babai, Cameron, Kantor, and Pyber from the 1990s (see~\cite{babai, CK, pyber}). (Recall that a transitive permutation group $G\leqs {\rm Sym}(\O)$ is \emph{primitive} if a point stabiliser is a maximal subgroup of $G$). Although these conjectures have now been resolved, the extensive work and newly introduced methods required to solve those problems has created lasting momentum in the area and several new directions of research have emerged.

\vs

A finite group $G$ is \emph{almost simple} if there exists a nonabelian simple group $G_0$ such that $G_0\normeq G \leqs {\rm Aut}(G_0)$, and we call $G_0$ the \emph{socle} of $G$ (this is the unique minimal normal subgroup of $G$). The possibilities for $G_0$ are determined by the \emph{Classification of Finite Simple Groups}. In particular, $G_0$ is either alternating, or a group of Lie type, or one of the so-called sporadic simple groups. The latter include the family of Mathieu groups ${\rm M}_{11}, {\rm M}_{12}, {\rm M}_{22}, {\rm M}_{23}, {\rm M}_{24}$. 

We say that a subgroup $H$ of a finite group $G$ is \emph{core-free} if $\bigcap_{g\in G} H^g = 1$. Note that if $G$ is almost simple, then $H$ is core-free if and only if it does not contain $G_0$. Let us now recall the following key definition:

\begin{deff}
Let $G$ be a finite almost simple group with socle $G_0$ and let $H$ be a core-free maximal subgroup of $G$. We say that $H$ is \emph{standard} if one of the following holds:
\begin{enumerate}
\item [\rm (i)] $G_0 = A_n$ and the natural action of $H$ on $\{1, \ldots, n\}$ is not primitive;
\vs
\item [\rm (ii)] $G$ is a classical group and $H$ is a subspace subgroup of $G$.
\end{enumerate}
We say that $H$ is \emph{non-standard} otherwise.
\end{deff}

\noindent We will formally define a subspace subgroup of a classical group in Section~\ref{s:prelims} (see Definition~\ref{def:ns}). But roughly speaking, if $G$ is a classical group, then $H$ is standard if it is \emph{reducible}, in the sense that $H\cap G_0$ acts reducibly on the natural module for $G_0$. Note that if there exists an isomorphism of $G$ to some other group $K$ that maps $H$ to a standard subgroup of $K$, then we will also say that $H$ is standard in this case. For instance, there exists an isomorphism $\phi:G \to K$ with $G = {\rm PSL}_4(2)$ and $K = A_8$, which maps the maximal subgroup $G = {\rm GL}_2(4).2$ of $G$ to the stabiliser in $K$ of some $3$-element subset of $\{1, \ldots, 8\}$, so we will view $H$ as a standard subgroup of $G$. Those cases are handled in Table~\ref{t:ns}.

If $G$ is an almost simple group in a transitive action with point stabiliser $H$, then we will say that the action of $G$ is non-standard if $H$ is non-standard. It turns out that groups in standard actions behave vastly differently to those in non-standard actions in terms of bases, so it is natural to distinguish between standard and non-standard actions in this setting. In particular, standard actions tend to admit much larger bases. For example, the base sizes of $S_n$ and $A_n$ in standard actions have been determined due to work by several authors (see~\cite{BGL, CoenColva, Hal, James, MS, MorSp}) and we observe that in the majority of cases, $b(G, \O)$ can get arbitrarily large as $n\to \infty$. For example, $b(S_n, \{1, \ldots, n\}) = n-1$. On the contrary, as we will now explain, bases of almost simple groups in non-standard actions exhibit a very different behaviour.

It was conjectured by Cameron and Kantor in~\cite{CK} that if $G$ is a primitive almost simple permutation group in a non-standard action, then $b(G, \O)$ is bounded above by an absolute constant independent of $G$. This was proved by Liebeck and Shalev in~\cite{LS} with an undetermined constant. In response, Cameron conjectured in~\cite[p.122]{Cam} that the best possible value of this absolute constant is $7$. This conjecture was proved in a series of papers by Burness et al.~\cite{B07, BGS, BLS, BOW}. In particular, we have the following result:
\begin{theorem}\label{thm:Cameron}
Let $G\leqs {\rm Sym}(\O)$ be an almost simple primitive permutation group with non-standard point stabiliser $H$. Then $b(G, \O) \leqs 7$, with equality if and only if $(G, H) = ({\rm M}_{24}, {\rm M}_{23})$.
\end{theorem}

\noindent The aim of this paper is to generalise Theorem~\ref{thm:Cameron} in the setting we describe below.

\vs

Let $G$ be a finite group and let $H_1, \ldots, H_k$ be a collection of core-free subgroups of $G$, allowing repetitions. Consider the natural componentwise action of $G$ on the Cartesian product 
\[
X = G/H_1 \times \cdots \times G/H_k
\]
and observe that $G$ has a regular orbit on $X$ if and only if  
\[
\bigcap_{i=1}^k H_i^{g_i} = 1
\]
for some elements $g_i \in G$. Following~\cite{AB}, we say that a $k$-tuple $\tau = (H_1, \ldots, H_k)$ of core-free subgroups of $G$ is \emph{regular} if $G$ has a regular orbit on $X$, and \emph{non-regular} otherwise. Here, we are only interested in the regularity of subgroup tuples all of whose components are core-free, and so, given a subgroup tuple, we will always assume that its components are core-free. We will also say that $\tau$ is \emph{conjugate} if the $H_i$ are pairwise conjugate. Finally, we call the minimal $r$ such that every $r$-tuple of core-free subgroups of $G$ is regular the \emph{regularity number} of $G$, which we denote by $R(G)$.

Observe that a transitive group $G\leqs {\rm Sym}(\O)$ has a base of size $k$ if and only if $G$ has a regular orbit on $\O^k$. Therefore, a conjugate tuple $(H_1, \ldots, H_k)$ is regular if and only if $G$ admits a base of size $k$ in its natural action on $G/H_1$, which means that $k\geqs b(G, H_1)$. 
 
If $G$ is almost simple and all the components of a tuple $\tau = (H_1, \ldots, H_k)$ are non-standard, then we will say that $\tau$ is \emph{non-standard}. Hence, in this language, Theorem~\ref{thm:Cameron} asserts that every conjugate non-standard $7$-tuple of $G$ is regular.

\vs

The regularity of subgroup tuples was introduced and studied in~\cite{AB} for almost simple groups with alternating and sporadic socle, where Burness and the author proposed generalisations of some of the most influential base size conjectures in the regularity setting, including Cameron's conjecture. In particular, if $G$ is a finite almost simple group and $R_{\rm ns}(G)$ denotes the minimal integer $r$ such that all non-standard tuples of $G$ are regular, then they conjecture that $R_{\rm ns}(G) \leqs 7$, with equality if and only if $G = {\rm M}_{24}$. As noted above, Theorem~\ref{thm:Cameron} implies that all conjugate non-standard $7$-tuples of a finite almost simple group $G$ are regular and we are claiming that the same conclusion holds without the restrictive conjugacy condition.

This conjecture was proved in~\cite{AB} for almost simple groups with alternating and sporadic socle, which reduces the problem to the almost simple groups of Lie type. The main aim of this paper is to complete the proof of this conjecture. We now state our main theorem:


\begin{theorem}\label{thm:main}
Let $G$ be a finite almost simple group. Then $R_{\rm ns}(G) \leqs 7$, with equality if and only if $G = {\rm M}_{24}$.
\end{theorem}

It is shown in~\cite{AB}, that up to conjugacy, the only non-regular non-standard $6$-tuple when $G = {\rm M}_{24}$ is $(H, H, H, H, H, H)$, where $H = {\rm M}_{23}$. Moreover, note that if $G$ is sporadic or exceptional, then every core-free maximal subgroup is non-standard, so Theorem~\ref{thm:main} asserts that $R(G) \leqs 7$ (with equality if and only if $G = {\rm M}_{24}$) in these cases.

\vs

As we will see in Section~\ref{s:prelims}, we will prove a stronger version of Theorem~\ref{thm:main} for classical groups. In particular, if $G$ is an almost simple classical group, then we show that $R_{\rm ns}(G) \leqs 5$ with equality if and only if $G = {\rm U}_6(2).2$, in which case the only non-regular $4$-tuple is $(H, H, H, H)$ with $H = {\rm U}_4(3).2^2$. Combining this with~\cite[Theorem 1(ii) and Theorem 2(i)]{AB}, we obtain the following more refined version of Theorem~\ref{thm:main}.

\begin{theorem} Let $G$ be a finite almost simple group with socle $G_0$:
\begin{itemize}
\item [\rm (i)] If $G_0 = A_n$, then $R_{\rm ns}(G) \leqs 6$ with equality if and only if $G = A_8$. Moreover, if $n\geqs 13$, then $R_{\rm ns}(G) = 2$;
\vs
\item [\rm (ii)] If $G_0$ is sporadic, then $R_{\rm ns}(G) = R(G)\leqs 7$ with equality if and only if $G = {\rm M}_{24}$. Moreover, the exact value of $R(G)$ is recorded in \cite[Table 3]{AB};
\vs
\item [\rm (iii)] If $G_0$ is a classical group, then $R_{\rm ns}(G) \leqs 5$ with equality if and only if $G = {\rm U}_6(2).2$;
\vs
\item [\rm (iv)] If $G_0$ is an exceptional group, then $R_{\rm ns}(G) = R(G) \leqs 6$, with equality if $G_0 = E_6(q)$ or $E_7(q)$.
\end{itemize}
\end{theorem}

We note that it is still an open problem to determine all the almost simple exceptional groups with $R(G) = 6$, but it follows from~\cite[Corollary 3]{B18} that if $G_0 = E_6(q)$ (respectively, $E_7(q)$), then the conjugate $5$-tuples containing copies of the $P_1$ or $P_6$ (respectively $P_7$) parabolic subgroup of $G$ are non-regular. 

\vs

We now briefly discuss the methods we use to prove Theorem~\ref{thm:main} for groups of Lie type. These are similar to those used in the proof of Cameron's conjecture, but additional challenges arise due to the substantial weakening of our assumptions, which we will highlight below. Our main tool will be a generalised version of a powerful probabilistic method for studying base sizes, which was first introduced in~\cite{LSh} by Liebeck and Shalev. 

Let $G$ be a finite group and let $\tau = (H_1, \ldots, H_k)$ be a tuple of core-free subgroups of $G$. We will write $X_{\tau} = G/H_1\times \cdots \times G/H_k$ and $\mathbb{P}(G, \tau)$ for the probability that a uniformly random point in $X_{\tau}$ is in a regular orbit of $G$. Then $\tau$ is regular if
\[
Q(G, \tau) = 1 - \mathbb{P}(G, \tau) < 1. 
\]
In general, it is very difficult to compute $Q(G, \tau)$ precisely. However, in most cases it is sufficient to work with the following upper bound (see~\cite{AB}):
\[
Q(G, \tau) \leqs \widehat{Q}(G,\tau) := \sum_{i=1}^t|x_i^G| \cdot \left(\prod_{j=1}^k {\rm fpr}(x_i, G/H_j)\right)
\]
where $x_1, \ldots, x_t$ is a complete set of representatives of the prime order conjugacy classes in $G$ and 
\[{\rm fpr}(x, G/H) = \frac{|C_{\O}(x)|}{|\O|}  = \frac{|x^G\cap H|}{|x^G|}\]
denotes the \emph{fixed point ratio} of $x$ in the natural action of $G$ on the set $\O$ of cosets of $H$. That is, the proportion of fixed points of $x$ in this action. We refer the reader to Section~\ref{s:prob-main} for more information.

\vs

Suppose that $G$ is an almost simple classical group with socle $G_0$ with natural module $V$ and let $\tau$ be a non-standard tuple of $G$. Here, our general approach when estimating $\widehat{Q}(G, \tau)$ is to use the following `zeta type function'
\[
\eta_G(t) = \sum_{C\in \mathscr{C}}|C|^{-t}
\]
where $t\in \R$ and $\mathscr{C}$ denotes the set of $G_0$-conjugacy classes of prime order elements of $G$. Burness obtains upper bounds on ${\rm fpr}(x, G/H)$ for a prime order element $x$ and a non-standard subgroup $H$ in terms of the size of the $G$-class of $x$ in~\cite{fprI, fprII, fprIII, fprIV}. If ${\rm dim}V > 10$, then using these allows us to bound $\widehat{Q}(G, \tau)$ by $\eta_G(t)$ for some specified $t\in (0, 1)$. Our problem then reduces to showing that $\eta_G(t) < 1$. 

On the other hand, if ${\rm dim} V\leqs 10$, then this general approach is no longer applicable. In some cases, we can work with a modified version of the general method, whereas a completely different approach is required in other cases. In particular, if ${\rm dim} V \leqs 5$, using $\eta_G(t)$ will no longer be possible, so we take a more direct approach to estimate $\widehat{Q}(G, \tau)$. In these `low-dimensional' cases, our proof heavily relies on a detailed analysis of fixed point ratios of classical groups in non-standard actions. More specifically, we make use of the description of the structure of the maximal subgroups of the low-dimensional classical groups given by Bray, Holt, and Roney-Dougal in~\cite{BHR}, together with the detailed information on conjugacy classes of prime order elements presented in~\cite{BG} by Burness and Giudici to improve the bounds given in~\cite{fprI, fprII, fprIII, fprIV} sufficiently to force $\widehat{Q}(G, \tau) < 1$. More information on this is given in Section~\ref{s:fprbounds}. We believe that those results might be of independent interest, so we record them in Table~\ref{t:fprs-classical}.

\vs

Now suppose that $G$ is an exceptional group of Lie type over $\F$. Here, our main goal will be again to force $\widehat{Q}(G, \tau) < 1$ for any $6$-tuple $\tau$ of core-free subgroups of $G$, but our approach when doing so is rather different to classical groups. Our proofs will again rely on fixed point ratio estimates, and our main reference for this will be~\cite{LLS}. In some cases, the uniform bounds given in~\cite[Theorems 1 and 2]{LLS} will suffice for our purposes, but in most cases, a more detailed case-by-case analysis will be required. This can be achieved using the powerful machinery developed in~\cite{LLS} by Lawther, Liebeck and Seitz. 

If $H$ is a non-parabolic subgroup of $G$, then upper bounds on ${\rm fpr}(x, G/H)$ for prime order elements $x\in G$ that are sufficient for our purposes can be extracted from~\cite{BLS}. Here, the authors use the techniques developed in~\cite{LLS} to derive those estimates. We believe that a record of such upper bounds could be of independent interest, so we record an upper bound on
\[
g(x, q) = \max\{{\rm fpr}(x, G/H) \, : \, H\leqs G \text{ non-parabolic}\}
\]
for a prime order element $x\in G$ in Tables~\ref{t:e6:fprs}, \ref{t:e6:fprs-q2}, \ref{t:e7:fprs}, \ref{t:f4:fprs}, and \ref{t:f4:fprs-q2} in Section~\ref{s:non-parabolics} when $G_0 \in \{F_4(q), E_6^{\epsilon}(q), E_7(q)\}$, where estimating $g(x, q)$ is a more laborious task.

The case where $H$ is a parabolic subgroup of $G$ requires a lot more work, particularly when $G_0 \in \{F_4(q), E_6^{\epsilon}(q), E_7(q), E_8(q)\}$, in which case the bounds in~\cite[Theorems 1 and 2]{LLS} are not good enough for our purposes. We note that if $H$ is a core-free maximal subgroup of $G$, then
 \[
 {\rm fpr}(x, G/H) = \frac{\chi(x)}{\chi(1)},
 \]
 where $\chi(x) = 1_{H}^G$ is the corresponding permutation character. Lawther, Liebeck and Seitz give an explicit formula for $\chi(x)$ in~\cite[2.4]{LLS} and~\cite[3.2]{LLS} for unipotent and semisimple elements respectively. Using these results together with information on conjugacy classes of unipotent and semisimple elements given in~\cite{LSeitz2} and~\cite{AHL, FJ1, FJ2, Lu2} respectively, as well as data on the so-called \emph{Green functions} recorded in~\cite{Lu3}, we are able to compute the precise value of ${\rm fpr}(x, G/H)$ for unipotent and semisimple elements.
 
 For graph, field and graph-field automorphisms, we are again able to work as in~\cite{LLS} to obtain sufficient fixed point ratio estimates.
 
In the way we describe above, we are able to estimate
 \[
 f(x, q) := \max \{{\rm fpr}(x, G/H) \, : \, H \in \mathcal{M}_G\}
 \]
for every prime order element $x\in G$, where $\mathcal{M}_G$ denotes the set of core-free maximal subgroups of $G$. This in turn allows us to obtain a uniform upper bound on $\widehat{Q}(G, \tau)$ that applies to any $6$-tuple $\tau$ of maximal subgroups of $G$.

\vs

The main challenge when the conjugacy assumption on our subgroup tuples is removed is that we often need to work with a very large number of possibilities of non-standard tuples, which forces us to adopt a uniform approach that allows us to handle all cases simultaneously. More specifically, our main strategy is to obtain a uniform upper bound on $\widehat{Q}(G, \tau)$ that applies to any choice of non-standard tuple $\tau$, which often means that for every choice of prime order element in $G$, we are bound to worst-case fixed point ratio estimates over all non-standard subgroups of our group $G$. 

This means that we will often need to obtain stronger fixed point ratio estimates for our purposes. We state those in Theorem~\ref{thm:low-rank-fprs} for the classical groups. For exceptional groups, we record worst case fixed point ratio estimates over non-parabolic actions in Propositions~\ref{prop:e6-fpr}, \ref{prop:e6-fpr-q2}, \ref{prop:e7-fpr}, and \ref{prop:f4-fpr}. Our main source for obtaining those is inspection of \cite{BLS} and \cite{B18}, and in some instances we need to improve what currently is available. We give detailed information in Sections~\ref{s:prelims-main}, \ref{s:fprbounds} and \ref{s:non-parabolics}. 

Let us see an example that illustrates the necessity of stronger fixed point ratio estimates in certain cases. Consider the case $G_ 0 = {\rm Sp}_6(q)$ for even $q$ and let $\tau = (H, H, H, K)$ be a $4$-tuple of $G$ with $H$ of type $G_2(q)$ and $K$ not of type $G_2(q)$. In~\cite{fprIV} Burness shows that ${\rm fpr}(x, G/H) < |x^G|^{-0.250}$ and moreover, he shows that the bound is best possible. It turns out that this bound is not sufficient to force $b(G, G/H) = 4$, and so Burness uses a constructive approach instead in the proof of~\cite[3.5]{B07} to force $b(G, G/H) = 4$. 

Due to the fact that $K$ is not of the same type as $H$, it would be much harder to obtain a geometric argument. Moreover, whilst sufficient to show that $b(G, G/K) \leqs 4$, the bound ${\rm fpr}(x, G/K) < |x^G|^{-1/3}$ obtained by Burness in~\cite{fprI, fprII, fprIII, fprIV} is not good enough to establish the regularity of $\tau$. However, it turns out that we can improve ${\rm fpr}(x, G/K)$ enough for any choice of $K$, so that we can use a probabilistic approach. More details can be found in Propositions~\ref{prop:sp6} and~\ref{prop:a3}.

Additionally, for classical groups there are certain instances where existing information on $\eta_G(t)$ is no longer sufficient for our purposes and we will need stronger results. We obtain those in Section~\ref{s:eta}.

\vs

The probabilistic approach we adopt to prove Theorem~\ref{thm:main} also allows us to establish some interesting asymptotic results. In particular, if $G$ is an almost simple group of Lie type and $\mathcal{P}$ denotes the set of non-standard subgroups of $G$, then we define
\[
\mathbb{P}(G, c) = \min \{ \mathbb{P}(G, \tau) \, : \, \tau \in \mathcal{P}^c\}.
\]
We have the following theorem:

\begin{theorem}\label{thm:cam-asymptotic}
If $G$ is a finite almost simple group of Lie type, then $\mathbb{P}(G, 6) \to 1$ as $|G|\to \infty$.
\end{theorem}

We note that the conclusion to Theorem~\ref{thm:cam-asymptotic} also sheds light to the regularity of tuples of simple algebraic groups. We will explore this connection further in a follow-up paper.

\vs

\noindent \textbf{Organisation.} This paper is organised as follows: In Section~\ref{s:prelims-main} we record some preliminary results on probabilistic and computational methods, and in Section~\ref{s:prelims}, we present some preliminaries that only apply to classical groups. We discuss subgroup structure and conjugacy classes of classical groups, and we also give a precise definition of a subspace subgroup of a classical group (see Definition~\ref{def:ns}). The rest of the section is devoted to probabilistic methods, which is the key tool that we will be using. 

Section~\ref{s:fprbounds} is devoted to improving existing fixed point ratio estimates for low-dimensional classical groups. In Section~\ref{s:eta} we prove some results relating to the `zeta-type' function introduced in Definition~\ref{def:eta}, and Section~\ref{s:proof} is devoted to the proof of Theorem~\ref{thm:main-classical}, a refinement of Theorem~\ref{thm:main} for classical groups. 

In Section~\ref{s:exceptional} we record some preliminaries that only apply to exceptional groups. In particular, the section is divided into two subsections, where we briefly discuss subgroup structure and conjugacy classes respectively. In Section~\ref{s:non-parabolics} we record fixed point ratio estimates for non-parabolic actions of exceptional groups, and then in Section~\ref{s:parabolics}, we proceed to discuss fixed point ratios for parabolic actions. Finally, Section~\ref{s:proof-exceptional} is devoted to the proof of Theorem~\ref{thm:exceptional}, which allows us to conclude the proof of Theorem~\ref{thm:main}.

\vs

\noindent \textbf{Notation.} Our notation is fairly standard. If $n$ is a natural number, we will often write $[n]$ to denote the set $\{1, \ldots, n\}$. Moreover, we will also often write $\pi(n)$ to denote the set of prime divisors of a positive integer $n$ and $\pi_a(n)$ to denote the set of prime divisors $p$ of $n$ with $p\leqs a$. In a similar fashion, we will denote the set of prime numbers at most $s$ by $\pi_s$ and the set of odd prime numbers at most $s$ by $\pi'_s$. 

For a group $G$, we write $i_s(G)$ for the number of elements of order $s$ in $G$. For classical groups, we write ${\rm PSL}^{+}_n(q) = {\rm PSL}_n(q) = {\rm L}_n(q)$ and ${\rm PSL}^{-}_n(q) = {\rm PSU}_n(q) = {\rm U}_n(q)$. We use the same notation and terminology as~\cite{BG} for field, graph, and graph-field automorphisms, which is consistent with~\cite{CFSGIII}, and we follow the definition of~\cite{KL} for the $\mathscr{C}_i$ and $\mathscr{S}$ collections of maximal subgroups of classical groups. We will also be using the standard notation $E_6^{\epsilon}(q)$ for the finite exceptional groups of type $E_6$, where $E_6^{+}(q) = E_6(q)$ and $E_6^{-}(q) = {}^2E_6(q)$. 

Finally, following the notation in~\cite{LLS}, we will often write $x = s, u, \phi, \tau, \tau\phi$ to mean that $x$ is semisimple, unipotent, a field, graph, or graph-field automorphism of $G_0$ respectively. 

We will introduce more technical terminology and notation, when and where appropriate in the paper.

\vs

\noindent \textbf{Acknowledgements.} This paper is part of the author's doctoral thesis under the supervision of Professor Tim Burness at the University of Bristol. The author thanks Professor Burness for his encouragement, feedback, and invaluable advice. She also thanks him for introducing her to this interesting problem. The author is grateful to the anonymous referee for a careful reading on the paper and helpful comments and suggestions. She also thanks Frank Lübeck for sharing \textsf{GAP}-readable data on Foulkes functions of exceptional groups. Finally, she thanks the Heilbronn Institute for Mathematical Research for funding her PhD at the University of Bristol.

\section{Preliminaries}\label{s:prelims-main}
In this section, we record some preliminary lemmas that we will repeatedly use to prove Theorem~\ref{thm:main}. We first discuss preliminaries related to probabilistic methods, and then we briefly comment on some of the main computational methods we apply in this paper. 

\subsection{Probabilistic methods}\label{s:prob-main}
Our main tool for proving our main result is the probabilistic method. Recall that if $G\leqs {\rm Sym}(\O)$ is a finite transitive permutation group with point stabiliser $H$, then the \emph{fixed point ratio} of $x\in G$ is defined as
\[
{\rm fpr}(x, G/H) = \frac{|C_{\O}(x)|}{|\O|} = \frac{|x^G\cap H|}{|x^G|}
\]
where $C_{\O}(x) = \{\omega \in \O \, : \, \omega^x = \omega\}$ is the set of fixed points of $x$ and $x^G$ denotes the conjugacy class of $x$ in $G$. We will use a generalisation of the powerful probabilistic approach based on fixed point ratio estimates that was first introduced in~\cite{LSh} by Liebeck and Shalev for the study of bases, which is also applicable in the more general regularity setting.

We first state two lemmas that appeared in~\cite[2.1, 2.2]{AB}, which are natural generalisations of equivalent lemmas used for the study of base sizes of permutation groups, which can be found in~\cite{B07}.

\begin{lem}\label{l:fpr}
Let $G$ be a finite group and let $\tau = (H_1, \ldots, H_k)$ be a core-free tuple of subgroups of $G$. Then $\tau$ is regular if 
\[
\widehat{Q}(G,\tau) := \sum_{i=1}^t|x_i^G| \cdot \left(\prod_{j=1}^k {\rm fpr}(x_i, G/H_j)\right) < 1,
\]
where $x_1, \ldots, x_t$ is a set of representatives of the conjugacy classes in $G$ of elements of prime order.
\end{lem}

\begin{lem}\label{l:favbound}
Let $G$ be a finite group and let $\tau = (H_1, \ldots, H_k)$ be a core-free tuple of subgroups of $G$. Suppose $x_1, \ldots, x_m\in G$ are such that $|x_i^G| \geqs B$ for all $i\in [m]$ and $\sum_{i = 1}^m |x_i^G\cap H_j| \leqs A_j$ for all $j\in [k]$. Then 
\[
\sum_{i = 1}^m |x_i^G| \cdot \left( \prod_{j = 1}^k {\rm fpr}(x_i, G/H_j)\right) \leqs B^{1 - k}\cdot \prod_{j = 1}^kA_j
\]
\end{lem}

Our aim will be to derive fixed point ratio bounds that apply to all almost simple groups with a fixed socle $G_0$. We now prove a small lemma that will allow us to deal with all such cases simultaneously.

\begin{lem}\label{lem:inndiag-fpr}
Let $G$ be a primitive permuatation group with point stabiliser $H$, let $x\in G$, and let $1\neq N\normeq G$. If $G_1 = \la N, x\ra$ and $H_1 = H\cap G_1$, then ${\rm fpr}(x, G/H) = {\rm fpr}(x, G_1/H_1)$.
\end{lem}

\begin{proof}
First note that the maximality of $H$ implies that $G = NH$. In particular, $G = G_1H$ and therefore $G/H = G_1H/H$. This implies that $G_1$ also acts transitively on the cosets of $H$. 

We will now prove that $G/H = \O$ and $G_1/H_1 = \O_1$ are isomorphic transitive $G_1$-sets. Define $\phi: \O \to \O_1$ by $\phi(Hg) = H_1g$ for $g\in G_1$. We claim that $\phi$ is a bijection. Suppose that $H_1y = H_1z$ for some $y, z \in G_1$. Then $yz^{-1} \in H_1$, and so $yz^{-1}\in H$, since $H_1\leqs H$. Therefore, $\phi$ is injective. Moreover, $|G:H| = |G_1:H_1|$, and so $\phi$ is bijective. 

Now note that if $y, z\in G_1$, then 
\[
\phi(Hyz) = H_1yz = \phi(Hy)z 
\]
so the actions of $G_1$ on $G/H$ and $G_1/H_1$ are indeed equivalent. This implies that $|C_{\O}(x)| = |C_{\O_1}(x)|$, and therefore
\[
\fpr(x, G/H) = \frac{|C_{\O}(x)|}{|\O|} = \frac{|C_{\O_1}(x)|}{|\O_1|} = {\rm fpr}(x, G_1/H_1),
\]
as required.
\end{proof}

\begin{rem}\label{rem:inndiag-fpr}
Note that if $G$ is a finite almost simple group with socle $G_0$, $H$ is a maximal core-free subgroup of $G$ and $x\in G_I = {\rm Inndiag}(G_0)$, then $G_0 \normeq \la G_0, x\ra \leqs G_I$, and Lemma~\ref{lem:inndiag-fpr} implies that 
\[
{\rm fpr}(x, G/H) \leqs \frac{|x^{G_{I}}\cap H|}{|x^{G_0}|}
\]
We will repeatedly use this observation when computing bounds for ${\rm fpr}(x, G/H)$. 
\end{rem}

We finally record a small number theoretic lemma that we will repeatedly use to derive fixed point ratio estimates:

\begin{lem}\label{lem:number-theory}
The following are true:
\begin{itemize}
\item If $\{a_1, \ldots, a_m\}$ and $\{b_1, \ldots, b_n\}$ be two sets of integers, all at least $2$, then
\[
\frac{\prod_{i = 1}^m (q^{a_i}-1)}{\prod_{i = 1}^n(q^{b_i} - 1)} > 2q^{\sum a_i - \sum b_i}
\]

\item If $a, b\in \mathbb{Z}^{+}$ with $b\leqs a$, then 
\[
\frac{(q^a+1)}{(q^b+1)} < q^{a-b}
\]
\end{itemize}
\end{lem}

\begin{proof}
This lemma is~\cite[1.2 (i) and (iii)]{LLS}.
\end{proof}

\subsection{Computational methods}\label{sss:comp-classical}
Here we briefly discuss how we use \textsf{GAP}~\cite{GAP} and  {\sc Magma}~\cite{magma} in our proofs. Note that we only discuss some of the methods here. In particular, for exceptional groups, we heavily use the polynomial functionality in \textsf{GAP} and {\sc Magma} to compute fixed point ratios, but we discuss this in detail later in Section~\ref{s:parabolics}.

\vs 

Let $G$ be a finite almost simple group of Lie type over $\F$ with socle $G_0$. There are two natural instances where we use {\sc Magma} in our proofs. The first is in Section~\ref{s:fprbounds} where $G$ is classical and we aim to improve the existing fixed point ratio bound given in Theorem~\ref{thm:timsbound}. There are some cases where we can only derive rather crude estimates, which prove the result when $q$ is sufficiently large, but do not suffice when $q$ is relatively small. In the latter situation, the dimension of the natural module for $G_0$ is small (in particular no more than $8$), which means that we can construct $G$ using the {\sc Magma} function \texttt{AutomorphismGroupSimpleGroup()} and call \texttt{MaximalSubgroups(G)} to obtain a set of representatives of the $G$-classes of maximal subgroups of $G$. We then filter out the non-standard subgroups of $G$ and one can verify the relevant fixed point ratio bounds using the following functions:

{\small
\begin{verbatim}
fpr := function( G, H, g )
  cl:=Classes(H);
  A:=[cl[i][2] : i in [1..#cl] | IsConjugate(G,cl[i][3],g)];
  Append(~A,0);
  return &+A*#Centraliser(G,g)/#G;
end function;
 
MaxFpr := function(G,H)
  cl:=Classes(G);
  a:=[i : i in [1..#cl] | IsPrime(cl[i][1])];
  b:=[i : i in a | fpr(G,H,cl[i][3]) gt 0];
  c,d:=Maximum([-Log(fpr(G,H,cl[i][3]))/Log(cl[i][2]): i in b]);
  return c;
end function;
\end{verbatim}}

\vs

\noindent The function \texttt{fpr} takes as input a permutation group $G$, a subgroup $H$ of $G$ and an element $g\in G$ and computes ${\rm fpr}(g, G/H)$. The function \texttt{MaxFpr} takes as input a group $G$ and a core-free subgroup $H$ of $G$ and returns the largest $t\in \R$ so that ${\rm fpr}(x, G/H) < |x^G|^{-t}$ for all $x\in G$ of prime order. The case where the computation using the code above is the most time consuming is when $G_0 = {\rm PSp}_8(3)$, in which case we can get output in less than an hour.

\vs

In some cases, we will also be using {\sc Magma} to confirm the regularity of a non-standard tuple $\tau$ of $G$. As mentioned above, our standard method is to bound $\widehat{Q}(G, \tau)$ above by $1$, but again, there are instances where the bounds we obtain are not good enough for small values of $q$, so in these cases we use the function \texttt{RegTuplesPlus} (see~\cite[p.4]{ABcomp}), to obtain $R_{\rm ns}(G)$. This takes as input a finite group $G$, an ordered tuple $M = [H_1, \ldots, H_t]$ of subgroups of $G$ and two positive integers $l$ and $N$ with $l\geqs 2$. The output is the minimal integer $k\geqs l$ such that all $k$-tuples of $G$ containing subgroups in $M$ are regular, together with all the non-regular $(k-1)$-tuples of $G$. We refer the reader to~\cite{ABcomp} for the code and more information.

To be able to use \texttt{RegTuplesPlus}, $|G|$ needs to be relatively small. In particular, we only make use of this function if $G_0\in \{ {}^2F_4(2)', {}^3D_4(2)\}$ for exceptional groups and when $G$ is at most $8$-dimensional. In the classical case, the maximal value of $q$ we use  \texttt{RegTuplesPlus} for, highly depends on the dimension of the natural module and the Lie type of the group we are considering, but it is always relatively small. For example, if $G_0 = {\rm PSp}_8(q)$, then we only use {\sc Magma} for the cases $q = 2, 3$, whilst if $G_0 = {\rm L}_2(q)$, then we use {\sc Magma} for $q\leqs 31$.

\vs

Finally, if $G_0 = F_4(2)$ or $E_6^{\epsilon}(2)$, then we will use the Character Table Library in \textsf{GAP}~\cite{GAPchar} to compute fixed point ratios for certain non-parabolic maximal subgroups.

\section{Classical group preliminaries}\label{s:prelims}
Throughout this section, as well as Sections~\ref{s:fprbounds}, \ref{s:eta} and~\ref{s:proof}, we will take $G$ to be a finite almost simple classical group over a finite field $\F$, where $q = p^f$ for some prime $p$ and some positive integer $f$. We will write $G_0$ for the socle of $G$ and $V$ will denote the natural module for $G_0$. Moreover, $G_I$ will denote the group ${\rm Inndiag}(G_0)$ of inner-diagonal automorphisms of $G_0$. 

Our goal here is to prove Theorem~\ref{thm:main} for classical groups. In particular, we will prove the following stronger statement:

\begin{thm}\label{thm:main-classical}
Let $G$ be a finite almost simple classical group. Then $R_{\rm ns}(G) \leqs 5$ with equality if and only if $G = {\rm U}_6(2).2$. Moreover $\mathbb{P}(G, 4) \to 1$ as $|G| \to \infty$, or $G = {\rm \Omega}_7(q)$ and $\mathbb{P}(G, 6)\to 1$ as $|G| \to \infty$.
\end{thm}

\begin{rem}
Let us record some comments on the statement of Theorem~\ref{thm:main-classical}.
\begin{itemize}
\item [\rm (a)] We will prove that $(H, H, H, H)$ is the only non-regular $4$-tuple of ${\rm U}_6(2).2$ up to conjugacy, where $H = {\rm U}_4(3).2^2$.
\vs
\item [\rm (b)] If $G = {\rm \Omega}_7(q)$ and $\tau$ is a $4$-tuple with at least one entry not of type $G_2(q)$, then $\mathbb{P}(G, \tau) \to 1$ as $|G|\to \infty$.
\vs
\item [\rm (c)] Theorem~\ref{thm:main-classical} is a direct regularity analogue of~\cite[Theorem 1]{B07}. In particular, Burness shows that $b(G, H) \leqs 5$ for every almost simple primitive classical group with non-standard point stabiliser $H$, with equality if and only if $(G, H) = ({\rm U}_6(2).2, {\rm U}_4(3).2^2)$.
\vs
\item [\rm (d)] Our methods will be similar to the ones used in~\cite{B07}. However, recall from Section~\ref{s:intro} that removing the conjugacy assumption on our tuples requires us to work with uniform worst case fixed point ratio bounds. This means that whilst we will heavily rely on the main theorem of~\cite{fprI, fprII, fprIII, fprIV} to obtain sufficient fixed point ratio estimates in the majority of the cases, there will be some cases, which we outline in Section~\ref{s:fprbounds}, where we will require stronger bounds. We obtain those bounds in Theorem~\ref{thm:low-rank-fprs}.
\vs
\item [\rm (e)] As we will explain in Section~\ref{s:prob-classical},we will be heavily using a function $\eta_G(t)$, which is a variant of a `zeta type function' introduced by Liebeck and Shalev in~\cite{LSh1} in the proof of Theorem~\ref{thm:main-classical}. We will rely on results in~\cite{B07} for this function, but as is the case with fixed point ratios, those results will not always be applicable in this more general setting, so we will need to refine those. We do this in Section~\ref{s:eta}.
\end{itemize}
\end{rem}

\vs

We now discuss some preliminary results around classical groups. In particular, we will briefly discuss Aschbacher's theorem that describes the subgroup structure of classical groups and record some information on conjugacy classes of prime order elements. We will then proceed by giving a formal definition of a subspace subgroup, and we will finally provide some more information on probabilistic and computational methods that only apply to the classical groups.

\subsection{Subgroup structure}\label{s:subgp}

The subgroup structure of finite classical groups is described by Aschbacher in~\cite{Asch}, where he defines nine separate collections and proves that every maximal subgroup of a finite classical group is contained in one of those collections. Note that the situation is slightly different when ${\rm P\O}^{+}_8(q)$ or ${\rm PSp}_4(2^f)$ and we discuss this below. The first eight collections, denoted by $\mathscr{C}_1, \ldots, \mathscr{C}_8$ are often referred to as the \emph{geometric collections}, and the final collection, which we denote by $\mathscr{S}$ is known as the \emph{non-geometric collection}. The members of the geometric collections can be described in terms of the underlying geometry of $V$. For example, the members of $\mathscr{C}_2$ are stabilisers of appropriate direct sum decompositions, and similarly, the members of $\mathscr{C}_4$ and $\mathscr{C}_7$ are stabilisers of appropriate tensor product decompositions. The $\mathscr{S}$-collection consists of subgroups that correspond to some absolutely irreducible representation on $V$ of a covering group of $H\cap G_0$ (see~\cite[p.3]{KL} for a precise definition). Note that the groups in $\mathscr{S}$ are almost simple. Throughout this paper, we will adopt the definition of~\cite{KL} for the collections of maximal subgroups, which differs slightly from the definition in the original paper by Aschbacher~\cite{Asch}. We give a brief description of each of the collections in Table~\ref{t:aschbacher}. 

Let $H$ be a maximal subgroup of $G$ such that $H\not\in \mathscr{S}$ and set $H_0 = G_0\cap H$. If $H_0$ is maximal, then $H_0\in \mathscr{C}$, where $\mathscr{C} = \mathscr{C}_1\cup \cdots \cup \mathscr{C}_8$.  If $G_0 = {\rm PSp}_4(2^f)$, or ${\rm P\O}_8^{+}(q)$, then the existence of exceptional automorphisms gives rise to an extra collection of so-called \emph{novelty subgroups}. Following the notation in~\cite{BG}, we will denote this class by $\mathscr{N}$.  We give a description of all those subgroups in Table~\ref{t:n-collection}, although the only cases of groups in $\mathscr{N}$ that will be considered in this paper are the ones where $G_0 = {\rm PSp}_4(q)$ for even $q$ and $H$ is of type ${\rm O}^{\epsilon}_2(q)\wr S_2$ or ${\rm O}_2^{-}(q).2$. We note that Aschbacher's original paper~\cite{Asch} only provides partial information for the case where $G_0 = {\rm P\O}^{+}_8(q)$ and $G$ contains triality graph automorphisms. In this case, the maximal subgroups were determined in~\cite{Kl3} by Kleidman.

We will repeatedly refer to the \emph{type} of a subgroup, following the terminology in~\cite{KL}. If $H\in \mathscr{S}$, then the type of $H$ will be the socle of the almost simple group $H$. On the other hand, if $H\in \mathscr{C}$, then the type of $H$ will give an approximate group-theoretic description of $H$ and will often indicate the underlying geometry associated to it. For example, if $G_0 = {\rm PSL}_n(q)$ and $H$ is of type ${\rm GL}_k(q)\wr S_t$, then $H$ is a stabiliser of a direct sum decomposition $V = V_1\oplus \cdots \oplus V_t$, where ${\rm dim}V_i = k$ for all $i$.

Our standard reference for the existence, structure and maximality of subgroups in $\mathscr{C}$ will be~\cite{KL}, where Kleidman and Liebeck provide a complete description of the structure of geometric subgroups and determine maximality (up to conjugacy) for groups whose natural module has dimension $n\geqs 13$. For classical groups whose natural module has dimension at most $12$ our standard reference becomes~\cite{BHR}, where Bray, Holt and Roney-Dougal determine all the maximal subgroups (up to conjugacy) of the low-dimensional classical groups.

In view of the above discussion, we now provide a version of Aschbacher's theorem on subgroup structure of classical groups:

\begin{thm}
Let $G$ be an almost simple classical group with socle $G_0$, and let $H$ be a core-free maximal subgroup of $G$. Then $H\in \mathscr{C}\cup \mathscr{S} \cup \mathscr{N}$.
\end{thm}

{\small
\begin{table}
\[
\begin{array}{ll} \hline
\mathscr{C}_1& \text{Stabilisers of subspaces or pairs of subspaces of } V\\
\mathscr{C}_2& \text{Stabilisers of direct sum decompositions } V = \bigoplus_{i = 1}^t V_i  \text{ with } {\rm dim}V_i = k\\
\mathscr{C}_3& \text{Stabilisers of prime degree extension fields of } \F\\
\mathscr{C}_4& \text{Stabilisers of tensor product decompositions } V = V_1\otimes V_2\\
\mathscr{C}_5& \text{Stabilisers of prime index subfields of } \F\\
\mathscr{C}_6& \text{Normalisers of symplectic type } r\text{-groups, where } r\neq p\\
\mathscr{C}_7& \text{Stabilisers of tensor product decompositions } V = \bigotimes_{i = 1}^tV_i \text{ with } {\rm dim}V_i = k\\
\mathscr{C}_8& \text{Stabilisers of nondegenerate forms on } V\\ 
\mathscr{S}& \text{Almost simple absolutely irreducible subgroups}\\
\hline
\end{array}
\]
\caption{The Aschbacher collections.}
\label{t:aschbacher}
\end{table}
}

{\small
\begin{table}
\[
\begin{array}{lll} \hline
G_0 & \text{Type of } H & \text{Conditions}\\\hline
{\rm PSp}_4(q)& {\rm O}_2^{\epsilon}(q)\wr S_2& q \text{ even}\\
& {\rm O}_2^{-}(q^2).2& q\text{ even}\\
& [q^4].{\rm GL}_1(q)^2& q\text{ even}\\
{\rm P\O}^{+}_8(q)& {\rm GL}_1^{\epsilon}(q)\times {\rm GL}_3^{\epsilon}(q)\\
&{\rm O}^{-}_2(q^2)\times {\rm O}^{-}_2(q^2)\\
& G_2(q)\\
& [2^9].{\rm SL}_3(2) & q = p >2\\
& [q^{11}]{\rm GL}_2(q){\rm GL}_1(q)^2\\
\hline
\end{array}
\]
\caption{The $\mathscr{N}$-collection.}
\label{t:n-collection}
\end{table}
}

\vs

In this paper we are interested in non-standard subgroup tuples of $G$. As we saw in the introduction, in the classical setting, a core-free subgroup of $G$ is non-standard if and only if it is not a subspace subgroup. We begin by giving a precise definition of a subspace subgroup.

\begin{defn}\label{def:ns}
A core-free maximal subgroup $H$ of $G$ is a \emph{subspace subgroup} if one of the following holds:
\begin{enumerate}
\item [\rm (i)] $H\in \mathscr{C}_1$;
\vs
\item [\rm (ii)] $(G_0, p) = ({\rm Sp}_{n}(q)', 2)$ and $H \in \mathscr{C}_8$;
\vs
\item [\rm (iii)] $(G, H)$ is displayed in Table~\ref{t:ns}.
\end{enumerate}
\end{defn}

In Table~\ref{t:ns} we list all the cases not covered by (i) and (ii) in Definition~\ref{def:ns}. Those are all instances where there exists some isomorphism $\phi : G\to K$ where ${\rm soc}(K) = T\cong G_0$ and one of the following holds for $\phi(H) = L$:
\begin{itemize}
\item [\rm (a)]  $L$ is a subspace subgroup of $K$ of type (i) or (ii) from Definition~\ref{def:ns};
\vs
\item [\rm (b)]  $T = A_m$ and $L$ does not act primitively on $\{1, \ldots, m\}$.
\end{itemize}
In the first column of Table~\ref{t:ns} we list $G_0$, in the second column the type of $H$, in the third column we list the pairs with entries $T$ and the type of $L$, and finally, in the fourth column we mention the conditions under which such an isomorphism occurs, if any. 

We claim that Table~\ref{t:ns} indeed covers all relevant cases. First suppose that $\phi$ is an automorphism. Noting that inner, diagonal, and field automorphisms do not permute the Aschbacher families, we deduce that $\phi$ must be a graph or graph-field automorphism. Moreover, we must have $G_0 \in \{{\rm P\O}_8^{+}(q), {\rm PSp}_4(q)'\}$, as graph and graph-field automorphisms of linear and unitary groups again preserve Aschbacher families. Hence, either $G_0 = {\rm P\O}_8^{+}(q)$ and $\phi$ is a graph automorphism or $G_0 = {\rm PSp}_4(q)'$ and $\phi$ is a graph-field automorphism. In those two cases, we verify that $\phi(H)$ is as claimed in Table~\ref{def:ns} by inspecting the first two columns~\cite[Tables 8.14, 8.50]{BHR}. 

If $\phi$ is not an automorphism, then the possible isomorphisms of classical groups are listed in~\cite[2.9.1]{KL}, and by inspecting the Tables in~\cite[Chapter 8]{BHR} and comparing orders of maximal subgroups, we verify that Table~\ref{t:ns} is complete.

\vs

{\small
\begin{table}
\[
\begin{array}{lllll} \hline
G_0 & \text{Type of } H & T& \text{Type of } L & \text{Conditions} \\ \hline
 {\rm P\O}_8^{+}(q)& \O_7(q)  & {\rm P\O}_8^{+}(q)& {\rm O}_7(q)\perp {\rm O}_1(q) & H\, \text{irreducible}, \, q\, \text{odd} \\
& {\rm Sp}_6(q) & {\rm P\O}_8^{+}(q)& {\rm Sp}_6(q) & H\, \text{irreducible},\, q\, \text{even}\\
 & {\rm GL}_4(q) & {\rm P\O}_8^{+}(q)& {\rm O}_2^{+}(q) \perp {\rm O}_6^{+}(q)\\
& {\rm GU}_4(q) & {\rm P\O}_8^{+}(q)& {\rm O}_2^{-}(q)\perp {\rm O}_6^{-}(q)\\
 & {\rm Sp}_2(q)\otimes {\rm Sp}_4(q) & {\rm P\O}_8^{+}(q)& {\rm O}_3(q)\perp {\rm O}_5(q)& q\, \text{odd}\\
{\rm L}^{\epsilon}_4(q)& {\rm Sp}_4(q)  & {\rm P\O}_6^{\epsilon}(q)& {\rm O}_5(q) \perp {\rm O}_1(q) & q\, \text{odd} \\
& {\rm Sp}_4(q) & {\rm P\O}_6^{\epsilon}(q) & {\rm Sp}_4(q) & q\, \text{even}\\
& {\rm GL}_2^{\epsilon}(q)\wr S_2 & {\rm P\O}_6^{\epsilon}(q)& {\rm O}_4^{+}(q) \perp {\rm O}^{\epsilon}_2(q) &q\geqs 3 \text{ and } q\geqs 4 \text{ if } \epsilon = +\\
& {\rm GL}_2^{\epsilon}(q^2) & {\rm P\O}_6^{\epsilon}(q)& {\rm O}_4^{-}(q)\perp {\rm O}^{-\epsilon}_2(q) & q\geqs 4 \text{ if } \epsilon  = -\\
& A_7 & A_8& S_7 & (\epsilon, q) = (+, 2)\\
{\rm PSp}_4(q)' & {\rm Sp}_2(q)\wr S_2& {\rm Sp}_4(q)& {\rm O}_4^{+}(q)&q\, \text{even}\\
& {\rm Sp}_2(q^2)& {\rm Sp}_4(q)& {\rm O}_4^{-}(q)&q\, \text{even}\\
& {\rm Sp}_2(q)\wr S_2& \O_5(q)& {\rm O}_4^{+}(q)\perp {\rm O}_1(q)&q\, \text{odd}\\
& {\rm Sp}_2(q^2)& \O_5(q)& {\rm O}_4^{-}(q)\perp {\rm O}_1(q)&q\, \text{odd}\\
& {\rm GL}^{\epsilon}_2(q)& \O_5(q)& {\rm O}_3(q) \perp {\rm O}_2^{\epsilon}(q)&q\, \text{odd}\\
& 2^{5}.{\rm O}_4^{-}(2)& {\rm PSU}_4(2)& P_2& q = p = 3\\
{\rm L}_2(q)& {\rm GL}_1(q)\wr S_2& {\rm Sp}_2(q)& {\rm O}^+_2(q)& q \, \text{even}\\
& {\rm GL}_1(q^2)& {\rm Sp}_2(q)& {\rm O}^-_2(q)& q \, \text{even}\\
& {\rm GL}_2(q_0)& \O^{-}_4(q_0)& {\rm O}_3(q_0) \perp {\rm O}_1(q_0)& q = q_0^2, q_0\neq 2\\
& 2^2.{\rm O}_2^{-}(2) & A_5& S_4& q = 5\\
& A_5& A_6& S_5& q=9\\
& 2^2.{\rm Sp}_2(2)& {\rm L}_3(2)& P_1& q = 7\\
\hline
\end{array}
\]
\caption{The standard subgroups of classical groups arising in part (iii) of Definition~\ref{def:ns}.}
\label{t:ns}
\end{table}
}

\subsection{Conjugacy classes}\label{s:cclasses}

In this section we will give a brief description of the conjugacy classes of prime order elements in classical groups, and we will establish the notation and terminology that we will be using throughout. We refer the reader to~\cite[Chapter 3]{BG} for a more detailed description of those classes and to~\cite[Chapter 2]{BG} for a more detailed analysis of the basic properties of the classical groups, as well as their diagonal, field, graph, and graph-field automorphisms. We note that due to a theorem of Steinberg~\cite[Theorem 30]{Stein}, every element in an almost simple classical group is a product of an inner, a diagonal, a field, and a graph automorphism. A statement of this result consistent with our definitions and notation can be found in~\cite[2.1.3]{BG}

First let $x\in G_I$ be an element of prime order $r$. We will say that $x$ is \emph{unipotent} if $r = p$ and that $x$ is \emph{semisimple} if $r\neq p$. The description of $x$ will differ depending on whether it is unipotent or semisimple, so we divide the analysis into two separate sections. 

\vs 

\paragraph{\textbf{Semisimple elements}}

Here we will assume that $x$ is semisimple. We start by stating a fundamental result (\cite[4.2.2.(j)]{CFSGIII}) that we will repeatedly use throughout this paper:

\begin{prop}\label{thm:GLS-ss}
If $T$ is a finite simple group of Lie type and let $x\in {\rm Inndiag}(T)$ be a semisimple element. Then $x^{T} = x^{{\rm Inndiag}(T)}$.
\end{prop}

Note that this theorem applies to all finite simple groups, and we will in fact make use of it when dealing with exceptional groups.

The analysis slightly differs when $x$ is an involution, so we will first discuss the structure of semisimple elements of odd order and then proceed to briefly comment on involutions. More information on all this can be found in~\cite[Chapter 3]{BG}. 

First suppose that $x$ has odd order $r$ and let $i\geqs 1$ be minimal such that $r$ divides $q^i-1$. As in~\cite{BG}, we will write
\begin{equation}\label{eq:c}
c =
\begin{cases}
2i \text{ if } i \text{ is odd and }G_0\neq {\rm L}_n(q)\\
i/2 \text{ if } i\equiv 2 \imod{4} \text{ and } G_0 = {\rm U}_n(q)\\
i \text{ otherwise}
\end{cases}
\end{equation}

First assume that $c\geqs 2$ and let $\widehat{G}$ be defined according to Table~\ref{t:matrix-gp} for each type of classical group. Then~\cite[3.1.3]{BG} implies that $x$ lifts to a unique element $\hat{x}$ of order $r$ in $\widehat{G}$. That is, $x = \hat{x}Z$, where $Z$ is the centre of $\widehat{G}$. 

We will now focus on the linear case, that is $G_0 = {\rm L}_n(q)$, to describe the structure of prime order classes. We will provide information on representatives, as well as the structure of centralisers. One can get similar descriptions with minor modifications for the unitary, symplectic, and orthogonal groups.

{\small
\begin{table}
\[
\begin{array}{llll} 
G_0 & {\rm PSL}_n^{\epsilon}(q) & {\rm PSp}_n(q)& {\rm P\O}_n^{\epsilon}(q)\\
\hline
\widehat{G}& {\rm GL}_n^{\epsilon}(q)& {\rm GSp}_n(q) & {\rm GO}_n^{\epsilon}(q)\\
\end{array}
\]
\caption{Matrix groups of classical groups.}
\label{t:matrix-gp}
\end{table}
}

Note that $\hat{x}$ is diagonalisable in $\mathbb{F}_{q^i}$, but not any proper subfield, so Maschke's Theorem implies that $\hat{x}$ fixes a direct sum decomposition
\[
V = U_1\oplus \cdots \oplus U_s\oplus C_V(\hat{x})
\]
of the natural ${\rm GL}_n(q)$-module $V$, where
\[
C_V(\hat{x}) = \{v\in V \, : \, v\hat{x} = v\}
\]
is the $1$-eigenspace of $\hat{x}$ and $U_j$ is an $i$-dimensional subspace of $V$, on which $\hat{x}$ acts irreducibly. 

Now let $\mathscr{S}_r$ denote the set of non-trivial $r$-th roots of unity in $\mathbb{F}_{q^i}$, and observe that the standard Frobenius automorphism $\sigma : \mathbb{F}_{q^i} \to \mathbb{F}_{q^i}$ defined by $\lambda \mapsto \lambda^q$ induces a permutation on $\mathscr{S}_r$. In particular, all $\sigma$-orbits have length $i$, so there are precisely $t = (r-1)/i$ distinct $\sigma$-orbits. We will write $\Lambda_1, \ldots, \Lambda_t$ to denote those $\sigma$-orbits, where
\[
\Lambda_j = \{\lambda_j, \lambda_j^q, \ldots, \lambda_j^{q^{i-1}}\}
\]
for some $\lambda_j \in \mathscr{S}_r$. We now note that the set of eigenvalues of $\hat{x}$ on $U_j \otimes \mathbb{F}_{q^i}$ coincides with some $\sigma$-orbit, so we may abuse notation and write
\[
\hat{x} = [\Lambda_1^{a_1}, \Lambda_2^{a_2}, \ldots, \Lambda_t^{a_t}, I_e],
\]
where $a_j$ denotes the multiplicity of $\Lambda_j$ on the multiset of eigenvalues of $\hat{x}$ on $V\otimes \mathbb{F}_{q^i}$ and $e = {\rm dim}C_V(\hat{x})$. It now follows from~\cite[3.1.7]{BG} that two semisimple elements of odd order $r$ are ${\rm GL}_n(q)$-conjugate if and only if they have the same multiset of eigenvalues on $\mathbb{F}_{q^i}$, and in particular, since every $x\in {\rm PGL}_n(q)$ of order $r$ has a unique lift $\hat{x}$ of order $r$ in ${\rm GL}_n(q)$, we have a complete description of the prime order semisimple classes with $c\geqs 2$ in ${\rm PGL}_n(q)$. In addition,~\cite[3.1.9]{BG} gives
\[
C_{{\rm GL}_n(q)}(\hat{x}) = {\rm GL}_e(q)\times \prod_{j = 1}^t {\rm GL}_{a_j}(q^i)
\]

With minor modifications we can get similar descriptions for odd order semisimple elements in all classical groups. For example if $G_0 = {\rm PSp}_n(q)$, then the main difference is that when $i$ is odd, then the $\sigma$-orbits come in inverse pairs $(\Lambda, \Lambda^{-1})$, and so every element of order $r$ is of the form $[(\Lambda_1, \Lambda_1^{-1})^{a_1}, \ldots, (\Lambda_t, \Lambda_t^{-1})^{a_t}, I_e]$. And again, two elements of order $r$ are ${\rm Sp}_n(q)$-conjugate if and only if they have the same eigenvalues in $\mathbb{F}_{q^i}$, so we get a similar description for odd order semisimple classes. We also get a similar description of the centraliser structure of $\hat{x}$. In particular, if $\widehat{G}$ is as defined in Table~\ref{t:matrix-gp} for each type of classical group, then~\cite[Remark 3.3.4, Remark 3.4.4, Remark 3.5.6]{BG} give
\[
C_{\widehat{G}}(\hat{x}) = \widehat{G}\cap C_{{\rm GL}_n(q^{\delta})}(\hat{x})
\]
where $\delta = 2$ if $G$ is unitary and $\delta = 1$ otherwise. This is all covered in detail in~\cite[Chapter 3]{BG}. Note that we will often refer to the \emph{type} of $C_G(x)$, which will always mean $C_{\widehat{G}}(\hat{x})$.

Finally, if $c = 1$, then we still write $x = \hat{x}Z$, but in this case the lift in $\hat{x}$ might not be unique, so additional conjugacy classes arise when $r$ divides $n$. A precise analysis of conjugacy in this case is given in~\cite[3.2.2, 3.3.3]{BG} and the structure of $C_{G_I}(x)$ is recorded in~\cite[Tables B.3 and B.4]{BG}.

We now briefly comment on the case where $G_0 = {\rm PSL}_n(q)$ and $H$ is a $\mathscr{C}_3$-subgroup of $G$, as we will often come across this situation in this paper. Set $L = {\rm GL}_{n/k}(q^k)$, let $W$ be the natural $L$-module over $\mathbb{F}_{p^k}$ and consider the natural embedding of $L$ in ${\rm GL}_n(q)$ with respect to a fixed $\mathbb{F}_{q^k}$-basis $\{w_1, \ldots, w_{n/k}\}$ for $W$. We refer the reader to~\cite[Construction 2.1.4]{BG} for more information. Now set $i_0$ to be the smallest positive integer such that $r$ divides $q^{ki_0}-1$, and note that
\[
i_0 = 
\begin{cases}
i/k \text{ if } k \text{ divides }i\\
i \text{ otherwise}
\end{cases}
\]
If $\Gamma_1, \ldots, \Gamma_s$ denote the distinct $\sigma_0$-orbits on $\mathscr{S}_r$, where $\sigma_0$ is the permutation defined by $\lambda \mapsto \lambda^{q^k}$ for all $\lambda\in \mathscr{S}_r$, then either each $\sigma$-orbit is a union of $k$ distinct $\sigma_0$-orbits, or $i = i_0$ and $\sigma$ and $\sigma_0$-orbits coincide. In particular,~\cite[5.3.2]{BG} gives the following description of the embedding of $\hat{x}\in L$ inside ${\rm GL}_n(q)$:
\begin{lem}\label{lem:c3}
If $\hat{x} = [\Gamma_1, \ldots \Gamma_s, I_e]\in L$, then, up to conjugacy, 
\[
\hat{x} = [\Gamma_1^{a_1}, (\Gamma_1^q)^{a_1}, \ldots, (\Gamma_1^{q^{k-1}})^{a_1}, \ldots, \Gamma_s^{a_s}, (\Gamma_s^q)^{a_s}, \ldots, (\Gamma_s^{q^{k-1}})^{a_s}, I_{ke}]
\]
on $V$. In particular, if $i = i_0$, then $\hat{x}$ is conjugate to $[\Gamma_1^{ka_1}, \ldots, \Gamma_s^{ka_s}, I_{ke}]$.
\end{lem}

\vs 

We finish this section by very briefly commenting on the semisimple involutions in ${\rm PGL}_n(q)$ and we refer the reader to~\cite{BG} for a detailed analysis of semisimple involutions in all classical almost simple groups. So now let $G_0 = {\rm PGL}_n(q)$ and take $x$ to be a semisimple involution. For every $i$ such that $1 \leqs i \leqs n/2$ we get a unique class of involutions with representative $t_i$, where $t_i = \hat{t}Z$ and $\hat{t}$ is an involution in ${\rm GL}_n(q)$ whose $(-1)$-eigenspace has dimension $i$. If $x \in t_i^G$, then we will say that it is a $t_i$-\emph{involution}. Finally, if $n$ is even, then we get an additional class of involutions with representative $t'_{n/2}$. This class of involutions arises from the quadratic field extension $\mathbb{F}_{q^2}/\F$ and we direct the reader to~\cite[3.2.2.2]{BG} for more details. The precise structure of $C_{G_I}(x)$, as well as conditions for which $x\in G_0$ are recorded in~\cite[Table B.1]{BG}.

\vs

\paragraph{\textbf{Unipotent elements}}

We now shift our attention to unipotent elements in finite classical groups. We refer the reader to~\cite{LSeitz2} and~\cite{BG} for more information. Let $x\in G$ be an element of order $p$. Then~\cite[3.1.3]{BG} implies that $x$ again lifts to a unique element $\hat{x}$ of order $p$ in the corresponding matrix group, so that $x = \hat{x}Z$, where $Z$ again denotes the centre of this matrix group. We may now identify $\hat{x}$ with its Jordan decomposition on $V$ and write
\[
\hat{x} = [J_p^{a_p}, \ldots, J_1^{a_1}],
\]
where each $J_i$ is a unipotent Jordan block of size $i$ and $a_i$ denotes the multiplicity of $J_i$ in $\hat{x}$. Note that if $G$ is symplectic or orthogonal, then there are restrictions on the multiplicities of odd and even Jordan blocks respectively. In particular, combining~\cite[3.4.1, 3.5.1]{BG} we get:
\begin{lem}
Let $\hat{x}\in {\rm Sp}_n(q)$ (respectively ${\rm O}^{\epsilon}_n(q)$) be an element of odd order $p$ with Jordan decomposition $[J_p^{a_p}, \ldots, J_1^{a_1}]$ on $V$. Then $a_i$ is even for all odd (respectively even) $i$.
\end{lem}

Now~\cite[3.1.14]{BG} implies that two elements in ${\rm GL}_n(q)$ are conjugate if and only if they have the same Jordan decomposition on $V$, and the uniqueness of the lift $\hat{x}$ implies that we again have a complete description of unipotent classes in ${\rm PGL}_n(q)$. 

There are similar descriptions in the unitary, symplectic, and orthogonal cases and we direct the reader to~\cite{BG} for detailed information. It is worth noting that in the symplectic and orthogonal cases the Jordan decomposition of $x$ does not determine the $G_I$-class of $x$ in general. Let us briefly look at the symplectic case:

If $p$ is odd, then the $G_0$-classes of unipotent elements are in one-to-one correspondence with the \emph{signed partitions} of the form
\[
(p^{a_p}, \epsilon_{p-1}(p-1)^{a_{p-1}}, (p-2)^{a_{p-2}}, \ldots, \epsilon_22^{a_2}, 1^{a_1}),
\]
where $\sum_i ia_i = n$, $\epsilon_{2i} = \pm$ and $a_i$ is even when $i$ is odd.

If $p=2$,  all unipotent elements have Jordan decomposition $[J_2^s, J_1^{n - 2s}]$. For each odd $s \in \{1, \ldots, n/2\}$ we get a unique $G_I$-class with Jordan decomposition $[J_2^s, J_1^{n - 2s}]$, which we denote by $b_s$, and for each even $s \in \{1, \ldots, n/2\}$, we get two distinct $G_I$-classes with Jordan decomposition $[J_2^s, J_1^{n - 2s}]$, which we label by $a_s$ and $c_s$. The orthogonal case is similar.

Finally, we will often be interested in the number and size of $G_0$-classes of unipotent elements, so we will often appeal to lemmas in~\cite{BG} that determine in how many $G_0$-classes a given $G_I$-class splits. For example, if $G_0 = {\rm PGL}_n^{\epsilon}(q)$, then~\cite[3.2.7, 3.3.10]{BG} tell us that the $G_I$-class of an element $x\in G_I$ of order $p$ with Jordan decomposition $[J_p^{a_p}, \ldots, J_1^{a_1}]$ splits into $(\nu, q-\epsilon)$ distinct $G_0$-classes, where $\nu = \gcd \{j \, : \, a_j > 0\}$. Similar results are given in~\cite[3.4.12, 3.5.14]{BG} for symplectic and orthogonal groups.

\vs

Finally, if $x$ is a field, graph, or graph-field automorphism, then detailed information can be found in~\cite[Chapter 3]{BG} and the structure of $C_{G_I}(x)$ is recorded in the relevant tables of~\cite[Appendix B]{BG}. We will again be particularly interested in the number of $G_0$-classes each $G_0$-class splits into and we will often refer to the relevant lemmas in~\cite{BG}. For example, in the linear case, we will repeatedly refer to~\cite[3.2.9, 3.2.14, 3.2.15]{BG}.

We record~\cite[3.50]{fprII} that will be useful to us when dealing with field, graph and graph-field automorphisms.

\begin{lem}\label{lem:outer-cosets}
Suppose that $G_0 \neq {\rm P\O}^{+}_8(q)$, let $H$ be a subgroup of $G$ and suppose that $x\in H\setminus {\rm PGL}(V)$ has prime order. Then $x^{G}\cap H \subseteq H_Ix$, where $H_I = H\cap G_I$.
\end{lem}

\subsection{Probabilistic methods}\label{s:prob-classical}
The proof of Theorem~\ref{thm:main-classical} will heavily rely on fixed point ratio bounds established in a series of papers by Burness~\cite{fprI, fprII, fprIII, fprIV}. We now state the main result of those papers, which is Theorem 1 in~\cite{fprI}.

We will write $n$ for the dimension of the natural module, unless $G_0 = {\rm Sp}_4(2)'$ or ${\rm L}_3(2)$, in which case $n = 2$ (since ${\rm Sp}_4(2)' \cong {\rm L}_2(2)$ and ${\rm L}_3(2)\cong {\rm L}_2(7)$). We will refer to $n$ as the \emph{dimension} of $G$.

\begin{thm}\label{thm:timsbound}
Let $G$ be a finite almost simple classical group over $\F$ and let $H$ be a non-standard subgroup of $G$. Then
\[
\fpr(x, G/H) < |x^G|^{-\frac{1}{2} + \frac{1}{n} + \iota(G, H)}
\]
for all $x\in G$ of prime order, where $\iota(G, H) \geqs 0$ is a known constant.
\end{thm}

Note that in most cases $\iota(G, H) = 0$. The precise value of $\iota(G, H)$ when $\iota(G, H) > 0$ is listed in~\cite[Table 1]{fprI}. 

Following~\cite{LSh1}, we now define an `zeta-type function' associated to a finite group $G$, which will be a key tool in the proof of our main theorem for high-dimensional classical groups.

\begin{defn}\label{def:eta}
Let $G$ be a finite almost simple classical group with socle $G_0$, and let $\mathcal{C}$ be the set of $G_0$-conjugacy classes of prime order elements in $G$. For $t\in \R$, we define:
\[
\eta_G(t) = \sum_{C\in \mathcal{C}}|C|^{-t}
\]
Evidently, there exists a real number $T_G\in (0, 1)$ such that $\eta_G(T_G) = 1$.
\end{defn}

Results on the asymptotic behaviour of $\eta_G(t)$ can be extracted from~\cite[1.10]{LSh1}. Note that the function considered in~\cite{LSh1} is rather different to $\eta_G(t)$. In particular, Liebeck and Shalev sum over all $G$-conjugacy classes, whilst we are summing over all prime order $G_0$-classes. However, as we explain below, it is still possible to extract asymptotic results for our purposes:

\begin{lem}\label{lem:eta-g0}
We have $\sum_{i = 1}^r |x_i^{G}|^{-t} \leqs \eta_G(t)$ for any almost simple classical group $G$ with socle $G_0$ and $t\in (0, 1)$, where $x_1, \ldots, x_r$ is a set of representatives of $G$-conjugacy classes of prime order elements in $G$.
\end{lem}

\begin{proof}
Let $x_1, \ldots, x_s\in G$ be a set of representatives of the prime order $G_0$-classes, and let $c_i$ be the number of $G_0$-classes in which $x_i^G$ splits into. Then
\[
 \sum_{i = 1}^r|x_i^G|^{-t} \leqs \sum_{i = 1}^r c_i\left(\frac{|x_i^G|}{c_i}\right)^{-t} = \sum_{j = 1}^{s}|x_j^{G_0}|^{-t} = \eta_G(t)
\]
as required.
\end{proof}

\begin{rem}
Note that the proof of Lemma~\ref{lem:eta-g0} does not rely at all on the assumption that the elements we are summing over have prime order. In fact, by carefully inspecting~\cite[1.10]{LSh1}, we note that Liebeck and Shalev use the exact same observation in their proof and work in terms of $G_0$-classes. This allows us to extract the following asymptotic result:
\end{rem}

\begin{prop}\label{prop:eta-asymptotic}
Let $G = G(q)$ be an almost simple classical group with socle $G_0$ over $\F$. Then the following hold:
\begin{itemize}
\item [\rm (i)] $\eta_G(t) \to 0$ as $q\to \infty$ for any $t > T_G$. 
\vs
\item [\rm (ii)] For a fixed real number $t > 0$, there exists some integer $r(t)$, such that for all almost simple classical groups $G = G(q)$ with dimension $r \geqs r(t)$ we have
\[
\eta_G(t) \to 0 \text{ as } |G| \to \infty.
\]
\end{itemize}
\end{prop}

\begin{proof}
We will prove (i). The argument for (ii) is identical. If $\mathcal{S}$ denotes the set of $G_0$-conjugacy classes in $G$, then inspecting~\cite[1.10]{LSh1} gives
\[
\zeta_G(t) := \sum_{C\in \mathcal{S}} |C|^{-t} \to 1 \,\,\text{ as } \,\,  |G| \to \infty
\]
and noting that $\eta_G(t) < \zeta_G(t) - 1$, the result follows.
\end{proof}

We now prove a variant of~\cite[2.2]{B07}, which will allow us to use $\eta_G$ to establish bounds on $R_{\rm ns}(G)$ for a finite classical almost simple group $G$.

\begin{prop}\label{prop:regular}
Let $G$ be a finite almost simple classical group and let $\tau = (H_1, \ldots, H_k)$ be a core-free tuple. Suppose that there exists $m > 0$ such that the following hold:
\begin{itemize}
\item [\rm (i)]  ${\rm fpr}(x, G/H_i) < |x^{G_0}|^{-m}$ for all $x\in G$ of prime order;
\vs
\item [\rm (ii)] $T_G < km - 1$.
\end{itemize}
Then $\tau$ is regular.

\end{prop}

\begin{proof}
In view of Lemma~\ref{l:fpr}, it suffices to show that $\widehat{Q}(G, \tau) < 1$. If $x_1, \ldots, x_k$ are representatives of the $G_0$-conjugacy classes of prime order elements in $G$, then Lemma~\ref{l:fpr} implies that
\[
\widehat{Q}(G, \tau) \leqs \sum_{i = 1}^k |x_i^{G_0}|^{-km + 1} = \eta_G(km - 1).
\]
Finally, noting that $\eta_G(t) < 1$ for all $t > T_G$, the claim follows.
\end{proof}

We now record some bounds on $T_G$ established in~\cite[Proposition 2.2, Remark 2.3]{B07} that will be key to proving the main result for almost simple classical groups of large dimension.

\begin{prop}\label{prop:TG}
Let $G$ be a finite almost simple classical group of dimension $n$ with socle $G_0$. Then the following hold:
\begin{itemize}
\item [\rm (i)] If $n\geqs 6$, then $T_G < 1/3$;
\vs
\item [\rm (ii)] If $G_0 = {\rm PSL}_{10}^{\epsilon}(q)$, then $T_G < 1/5$;
\vs
\item [\rm (iii)] If $G_0 = {\rm P\O}_n^{\epsilon}(q)$ and $n\geqs 12$, then $T_G < 4/15$.
 \end{itemize}
\end{prop}

\begin{rem}
Burness again works with a slightly different function than $\eta_G(t)$. In particular, he sums over $G$-classes of prime order elements instead of $G_0$-classes. However, by closely inspecting his proofs, we note that he also makes use of the observation in Lemma~\ref{lem:eta-g0}, and so the bounds in~\cite[Proposition 2.2, Remark 2.3]{B07} still apply in our setting.
\end{rem}

We now record a result on fixed point ratios of almost simple classical groups, which is an immediate consequence of~\cite[Theorem 1]{LS}. Note that the original result concerns all the almost simple classical groups, but we will only need the result for groups with socle ${\rm PSL}_2(q)$. 

\begin{prop}\label{prop:liebeck-saxl}
Let $G$ be an almost simple group with socle ${\rm PSL}_2(q)$ for $q\geqs 23$, $H$ be a non-standard subgroup of $G$, and $x\in G$ an element of prime order $r$. Then 
\[
{\rm fpr}(x, G/H) < \frac{4}{3q}
\]
or ${\rm fpr}(x, G/H) < \alpha$ and $(H, x, \alpha)$ is recorded in Table~\ref{t:liebeck-saxl}.
\end{prop}

{\small
\begin{table}
\[
\begin{array}{llll} \hline
\text{Type of } H & x & \alpha & \text{Conditions}\\
\hline
{\rm GL}_1(q)\wr S_2 & \phi & 2/(q+1) & r = 2\\
{\rm GL}_2(q^{1/2}) & u & 2/(q+1) & p \text{ odd}\\
& s & 2/(q+1) & r = 2, \, p \text{ odd}\\
& \phi & (2 + q^{1/2}(q^{1/2} + 1))/(q^{1/2}(q+1))& r = 2, \, p \text{ odd}\\
\hline
\end{array}
\]
\caption{The exceptional cases in Proposition~\ref{prop:liebeck-saxl}.}
\label{t:liebeck-saxl}
\end{table}
}

In view of Lemma~\ref{lem:inndiag-fpr}, Lemma~\ref{lem:eta-g0} and Remark~\ref{rem:inndiag-fpr}, we will often need to refer to results in the literature on lower bounds on $|x^G_{0}|$. In particular, if $G_0$ is in the following collection:
\[
\mathcal{A} = \{{\rm PSL}^{\epsilon}_3(q), {\rm PSL}_4^{\epsilon}(q), {\rm PSL}_5^{\epsilon}(q), {\rm PSp}_4(q)\}
\]
then we will repeatedly appeal to results in~\cite{fprII} on conjugacy class bounds. In the following lemma we record the bounds that are of interest to us.

\begin{lem}\label{lem:class-bounds}
Suppose that $G_0\in \mathcal{A}$ and $x\in G$ is an element of prime order $r$. Then $|x^{G_0}| > \alpha$, where $\alpha$ is recorded in Table~\ref{t:class-bounds} and $\beta = \frac{q}{q+1}$.
\end{lem}

\begin{proof}
Lower bounds on $|x^{G_I}|$ (or $|x^{G_0}|$ if $x$ is a field, graph, or graph-field automorphism) are established in~\cite[Sections 3.3-3.5]{fprIII}. Combining these with results in~\cite{BG}, which describe when prime order $G_I$-classes split into $G_0$-classes, we verify the entries in Table~\ref{t:class-bounds}. The argument is similar in most cases, so we only treat a few cases that are slightly more involved here. In particular, we will discuss the cases where $G_0 = {\rm PSL}_3^{\epsilon}(q)$ and $x\in G_I$, or $x$ is a field or graph-field automorphism and $G_0 = {\rm PSp}_4(q)$ and $x\in G_I$.

First assume that $G_0 = {\rm PSL}_3^{\epsilon}(q)$ and $x\in G_I$. If $x$ is unipotent, then either $x = [J_3]$ and~\cite[3.22]{fprIII} gives $|x^{G_I}| > \frac{1}{2}\left(\frac{q}{q+1}\right)q^6$, or $x = [J_2, J_1]$ and ~\cite[3.22]{fprIII} gives $|x^{G_I}| > \frac{1}{2}\left(\frac{q}{q+1}\right)q^4$. Moreover,~\cite[3.2.8, 3.3.10]{BG} imply that $x^{G_I} = x^{G_0}$ if $x = [J_2, J_1]$, whilst $x^{G_I}$ splits into at most $3$ $G_0$-classes if $x = [J_3]$, so the bounds in Table~\ref{t:class-bounds} are satisfied.

Next suppose that $x$ is semisimple and let $c$ be as defined in~\eqref{eq:c}. If $c\geqs 2$, or $c = 1$ and $r\neq 3$, then the bounds in Table~\ref{t:class-bounds} are verified via~\cite[3.36]{fprII}. On the other hand, if $c = 1$ and $r = 3$, then~\cite[3.35]{fprII} gives
\[
|x^{G_0}| > \frac{1}{6}\left(\frac{q}{q+1}\right)^2q^6,
\]
as claimed.

Now assume that $x$ is a field or graph-field automorphism. If $r = 2$, then
\[
|x^{G_0}| \geqs \frac{|{\rm PSL}_3^{\epsilon}(q)|}{|{\rm PGU}_3(q^{1/2})|} > \frac{1}{2(3, q-\epsilon)}\left(\frac{q}{q+1}\right)q^{(3^2-1)/2} > \frac{1}{6}\left(\frac{q}{q+1}\right)q^4.
\] 
Similarly, if $r\geqs 3$, then we obtain $|x^{G_0}| > \frac{1}{6}\left(\frac{q}{q+1}\right)q^{16/3}$.

Next assume that $G_0 = {\rm PSp}_4(q)$ and $x\in G_I$. Suppose that $x$ is unipotent. If $p$ is odd, then appealing to~\cite[3.21]{fprII} we find that $|x^{G_0}| > \frac{1}{4}\left(\frac{q}{q+1}\right)q^4$, and in particular $|x^{G_0}| > \frac{1}{4}\left(\frac{q}{q+1}\right)q^6$ if $x$ does not have Jordan decomposition $[J_2, J_1^2]$ on $V$. On the other hand, if $p = 2$, then inspecting~\cite[Table 3.5]{fprII} tells us that $|x^{G_0}| > \frac{1}{2}q^4$ if $x$ is an $a_2$ or $b_1$-involution and $|x^{G_0}| > \frac{1}{2}q^6$ if $x$ is a $c_2$-involution.

Similarly, if $x$ is semisimple and $s$ denotes the codimension of the largest eigenspace of the lift $\hat{x}$ of $x$ on ${\rm Sp}_4(q)$, then we note that $s > 1$, and so~\cite[3.36-3.37]{fprII} immediately give 
\[
|x^{G_0}| = |x^{G_I}| > \frac{1}{4}q^4.
\]
Moreover, unless $x$ is a $t_1$ or $t_1'$ involution, then we obtain the better bound $|x^{G_0}| > \frac{1}{2}\left(\frac{q}{q+1}\right)q^6$ via~\cite[3.30, 3.37]{fprII}.

We obtain all other bounds in Table~\ref{t:class-bounds} in a similar fashion.
\end{proof}

{\small
\begin{table}
\[
\begin{array}{llll} \hline
G_0 & x & \alpha & \text{Conditions}\\ \hline
{\rm PSL}_3^{\epsilon}(q)& s& \frac{1}{2}\beta q^4& C_G(x) \text{ of type } {\rm GL}_2^{\epsilon}(q)\times {\rm GL}_1^{\epsilon}(q)\\
& & \frac{1}{6}\beta^2q^6& C_G(x) \text{ not of type } {\rm GL}_2^{\epsilon}(q)\times {\rm GL}_1^{\epsilon}(q)\\
& u& \frac{1}{2}\beta q^4& x = [J_2, J_1]\\
& & \frac{1}{6}\beta q^6& x = [J_3]\\
& \phi& \frac{1}{6}\beta q^{16/3}& r\geqs 3\\
&\phi, \tau, \phi\tau & \frac{1}{6}\beta q^{4}& r = 2\\
{\rm PSL}_4^{\epsilon}(q)& s& \frac{1}{2}\beta q^6& C_G(x) \text{ of type } {\rm GL}_3^{\epsilon}(q)\times {\rm GL}_1^{\epsilon}(q)\\
& & \frac{1}{4}\beta q^8& C_{G}(x) \text{ of type } {\rm GL}_2^{\epsilon}(q)^2 \text{ or } {\rm GL}_2^{\epsilon}(q^2)\\
& & \frac{1}{2}\beta^2q^{10} & C_G(x) \text{ not of type } {\rm GL}_{2}^{\epsilon}(q^2) \text{ or } {\rm GL}^{\epsilon}_s(q)\times {\rm GL}_{4-s}^{\epsilon}(q) \text{ for } s\in \{2, 3\}\\
& u&  \frac{1}{2}\beta q^6& x = [J_2, J_1^2]\\
& & \frac{1}{4}\beta q^8& x = [J_2^2]\\
&& \frac{1}{2}\beta q^{10}& x \neq [J_2, J_1^2], [J_2^2]\\
& \phi, \phi \tau & \frac{1}{8}\beta q^{15/2}&\\
& \tau & \frac{1}{4} \beta q^4\\
{\rm PSL}_5^{\epsilon}(q) & s & \frac{1}{2}\beta q^8& C_G(x) \text{ of type } {\rm GL}_4^{\epsilon}(q)\times {\rm GL}_1^{\epsilon}(q)\\
&& \frac{1}{2}\beta q^{12} & C_G(x) \text{ of type } {\rm GL}_3^{\epsilon}(q)\times {\rm GL}_2^{\epsilon}(q)\\
&& \frac{1}{2}\beta^2q^{14} & C_G(x) \text{ not of type } {\rm GL}_s^{\epsilon}(q)\times {\rm GL}_{5-s}^{\epsilon}(q) \text{ for } s\in \{3, 4\}\\
& u& \frac{1}{2}\beta q^8 & x = [J_2, J_1^3]\\
& & \frac{1}{2}\beta q^{12} & x = [J_2^2, J_1]\\
& & \frac{1}{2}\beta q^{14} & x \neq [J_2, J_1^3], [J_2^2, J_1]\\
& \phi, \phi\tau & \frac{1}{10}\beta q^{12} & \\
& \tau& \frac{1}{10}\beta q^{14}\\
{\rm PSp}_4(q)& s & \frac{1}{4}q^4 & x \text{ is a } t_1 \text{ or } t'_1 \text{ involution}\\
& & \frac{1}{2}\beta q^6& x \text{ is not a } t_1 \text{ or } t'_1 \text{involution}\\
& u& \frac{1}{4}\beta q^4 & p\neq 2  \text{ and } x = [J_2, J_1^2]\\
& & \frac{1}{4}\beta q^6& p\neq 2  \text{ and }x \neq [J_2, J_1^2]\\
&&\frac{1}{2}q^4& p = 2 \text{ and } x \text{ is of type } a_2 \text{ or } b_1\\
&& \frac{1}{2}q^6& p=2 \text{ and } x \text{ is of type } c_2\\
& \phi& \frac{1}{4}\beta q^{20/3} & r\geqs 3\\
& \phi, \phi\tau & \frac{1}{4}q^{5} & r = 2\\
\hline
\end{array}
\]
\caption{The bounds from Lemma~\ref{lem:class-bounds}.}
\label{t:class-bounds}
\end{table}
}

We are now in a position to prove Theorem~\ref{thm:main-classical} for high-dimensional classical groups. Let $\mathcal{C}$ denote the set of simple classical groups. We will divide the possibilities for $G_0\in \mathcal{C}$ into the following six collections, according to the method we use to prove Theorem~\ref{thm:main-classical}, and we will handle each one in turn:
\begin{align*}
\mathcal{A}_1 &= \{{\rm PSp}_8(q), {\rm PSL}_8^{\epsilon}(q), {\rm P\O}^{\epsilon}_{10}(q)\}\\
\mathcal{A}_2 &= \{{\rm PSL}^{\epsilon}_6(q)\}\\
\mathcal{A}_3 &= \{{\rm PSp}_6(q) \, : \, q \text{ even}\} \cup \{{\rm \O}_7(q) \, : \, q \text{ odd}\}\\
\mathcal{A}_4 &= \{{\rm PSL}_3^{\epsilon}(q), {\rm PSL}_4^{\epsilon}(q), {\rm PSL}_5^{\epsilon}(q), {\rm PSp}_4(q)\}\\
\mathcal{A}_5 &= \{{\rm PSL}_2(q)\}\\
\mathcal{A}_6 &= \mathcal{C} \setminus (\mathcal{A}_1\cup \mathcal{A}_2\cup \mathcal{A}_3\cup \mathcal{A}_4\cup \mathcal{A}_5)
\end{align*}

\begin{prop}\label{prop:a6}
The conclusion to Theorem~\ref{thm:main-classical} holds for the groups with $G_0\in \mathcal{A}_6$.
\end{prop}

\begin{proof}
Let $\tau = (H_1, H_2, H_3, H_4)$ be a non-standard tuple. First suppose that $G_0 = {\rm PSp}_n(q)$, in which case $n\geqs 10$. If $\iota(G, H_i) = 0$ for all $i\in \{1, \ldots, 4\}$, then it follows from Theorem~\ref{thm:timsbound} that ${\rm fpr}(x, G/H_i) < |x^G|^{-2/5}$ for all $i$ and prime order elements $x\in G$. Now Proposition~\ref{prop:TG} implies that 
\[
T_G < \frac{1}{3} < 4\cdot \frac{2}{5} - 1 = \frac{3}{5},
\]
and so setting $m = 2/5$ and using Proposition~\ref{prop:regular} proves the claim.

Now suppose that $\iota(G, H_i) > 0$ for some $i\in \{1, \ldots, 4\}$. Inspecting~\cite[Table 1]{fprI}, we deduce that $n\equiv 0 \imod{4}$ and $\iota(G, H_i) \leqs 1/n$. Noting that $n\geqs 12$, we deduce that ${\rm fpr}(x, G/H_j) < |x^G|^{-1/3}$ for all $j\in \{1, \ldots, 4\}$ and prime order elements $x\in G$, and Proposition~\ref{prop:TG} implies that $\tau$ is regular.

\vs

Next suppose that $G_0 = {\rm PSL}_n(q)$, in which case we again have $n \geqs 10$. If $\iota(G, H_i) = 0$, then as in the symplectic case, Theorem~\ref{thm:timsbound} implies that ${\rm fpr}(x, G/H_i) < |x^G|^{-2/5}$ for all $i\in \{1, \ldots, 4\}$ and prime order elements $x\in G$. Then the argument above carries through unchanged. 

Now suppose that $\iota(G, H_i) > 0$ for some $i\in \{1, \ldots, 4\}$. Inspecting~\cite[Table 1]{fprI} tells us that $n\equiv 0 \imod{2}$ and $\iota  = 1/n$. If $n = 10$, then Theorem~\ref{thm:timsbound} gives ${\rm fpr}(x, G/H_j) < |x^G|^{-3/10}$ for all $j\in \{1, \ldots, 4\}$ and Proposition~\ref{prop:TG} tells us that 
\[
T_G < \frac{1}{5} = 4\cdot \frac{3}{10}- 1.
\]
Therefore, setting $m = 3/10$ and applying Proposition~\ref{prop:regular} shows that $\tau$ is regular. If $n\geqs 12$, then Theorem~\ref{thm:timsbound} gives ${\rm fpr}(x, G/H_j) < |x^G|^{-1/3}$ for all $j\in \{1, \ldots, 4\}$, and we compute $T_G < 1/3 = 4\cdot 1/3 - 1$, so setting $m = 1/3$ and applying Proposition~\ref{prop:regular} proves the claim.

\vs

Finally suppose that $G_0 = {\rm P\O}^{\epsilon}_n(q)$, in which case $n\geqs 12$. Here inspecting~\cite[Table 1]{fprI} tells us that $\iota(G, H_i) < 1/(n-2)$, and so if follows from Theorem~\ref{thm:timsbound} that ${\rm fpr}(x, G/H_i) < |x^G|^{-19/60}$. Finally, we can infer from Proposition~\ref{prop:TG} that 
\[
T_G < \frac{4}{15} = 4\cdot \frac{19}{60} - 1,
\]
and so setting $m = 19/60$ and using Proposition~\ref{prop:regular} proves the claim.

In addition, Proposition~\ref{prop:eta-asymptotic} implies that $\mathbb{P}(G, 4)\to 1$ as $|G|\to \infty$.
\end{proof}

The remainder of the part of the paper on classical groups is devoted to proving Theorem~\ref{thm:main-classical} for groups with socle in collections $\mathcal{A}_1, \ldots, \mathcal{A}_5$.

\section{Fixed point ratios for classical groups}\label{s:fprbounds}

In this section we improve the fixed point ratio bounds obtained in~\cite{fprI, fprII, fprIII, fprIV} for some low-dimensional classical groups. The goal here is to obtain sharper bounds that will allow us to use the eta function machinery discussed in Definition~\ref{def:eta} and Propositions~\ref{prop:eta-asymptotic} and~\ref{prop:regular} in some low-dimensional cases, where a more direct approach was taken in~\cite{B07} to derive results on base sizes.

In this section, we will write $G$ for a finite almost simple classical group with socle $G_0$ over $\F$, where $q = p^f$ for some prime $p$. Moreover, $V$ will denote the natural module of $G_0$ and we will write $H$ for a non-standard subgroup of $G$ and $G_I$ for ${\rm Inndiag}(G_0)$. Finally, we will denote the standard Frobenius morphism defined by $\lambda \mapsto \lambda^q$ by $\sigma$. We will improve the fixed point ratio bounds in Theorem~\ref{thm:timsbound} only for certain families of classical groups. The following list contains the possible socle types of the groups we consider and we will treat each possibility in turn:
\[
\mathcal{B} = \{{\rm P\O}_{10}^{\epsilon}(q),  {\rm PSL}^{\epsilon}_{8}(q), {\rm PSp}_8(q), {\rm \O}_7(q), {\rm PSL}_6^{\epsilon}(q), {\rm PSp}_6(q),  {\rm PSL}_{5}^{\epsilon}(q), {\rm PSL}_4^{\epsilon}(q), {\rm PSL}_3^{\epsilon}(q)
\}
\]

We will prove the following theorem:

\begin{thm}\label{thm:low-rank-fprs}
If $G_0\in \mathcal{B}$, and $x\in G$ has prime order $r$, then ${\rm fpr}(x, G/H) < |x^{G_0}|^{-\beta}$, where $\beta$ is recorded in Table~\ref{t:fprs-classical}. 
\end{thm}

\begin{rem}
We note that our bounds have a slightly weaker form than the ones given by Burness in~\cite{fprI, fprII, fprIII, fprIV}. In particular, he bounds ${\rm fpr}(x, G/H)$ in terms of the $G$-conjugacy class of $x$, whilst we obtain an estimate in terms of the $G_0$-class of $x$. However, note that our different definition of $\eta_G$ allows us to work with this slightly looser bound.
\end{rem}

We will divide the proof of Theorem~\ref{thm:low-rank-fprs} in a series of propositions, according to the socle type of $G$, where we will consider the non-standard subgroups of $G$, for which the claim does not follow immediately from Theorem~\ref{thm:timsbound}. We briefly outline our general approach:
\begin{itemize}
\item [\rm (i)] In the majority of cases we first verify that the claim holds using {\sc Magma} for small values of $q$, where our general approach is not applicable. To do this, we use the methods described in Section~\ref{sss:comp-classical}.
\vs
\item [\rm (ii)] We next consider elements $x\in G_I$. To derive an upper bound on ${\rm fpr}(x, G/H)$ we appeal to Lemma~\ref{lem:inndiag-fpr} and Remark~\ref{rem:inndiag-fpr}, which allow us to work with the estimate
\[ 
{\rm fpr}(x, G/H) \leqs \frac{|x^{G_I}\cap H|}{|x^{G_0}|}.
\]
Therefore, our first step is to derive an upper bound on $|x^{G_I}\cap H|$ and a lower bound on $|x^{G_0}|$, which we either do directly, or we inspect the proofs of~\cite{fprI, fprII, fprIII, fprIV}, and we use the bounds derived there when those are sufficient.
\vs
\item [\rm (ii)] We then note that $\frac{|x^{G_I}\cap H|}{|x^{G_0}|} < |x^{G_0}|^{-\beta}$ if and only if $\log |x^{G_I}\cap H|/\log |x^{G_0}| < 1-\beta$, so we work with this bound and verify that the claim in Theorem~\ref{thm:low-rank-fprs} holds.
\vs
\item [\rm (iii)] We then consider graph, field and graph-field automorphisms. In this case, the estimates in the proofs of~\cite{fprI, fprII, fprIII, fprIV} are sufficient in most cases, and when they are not, then we work in a similar fashion as Burness to derive tighter bounds.
\end{itemize}

{\small
\begin{table}
\[
\begin{array}{lll} \hline
G_0 & \beta & \text{Conditions}\\
\hline
{\rm P\O}_{10}^{\epsilon}(q) & 1/3\\
{\rm PSL}_8(q) & 1/3\\
{\rm PSp}_8(q) & 1/3\\
{\rm \O}_7(q)& 0.4& H \text{ not of type } G_2(q), p\neq 2\\
& 0.225& H \text{ of type } G_2(q), p\neq 2\\
{\rm PSL}_6(q)& 1/3 & C_G(x) \text{ not of type } {\rm GL}_3(q)^2\\
& 0.32& C_G(x) \text{ of type } {\rm GL}_3(q)^2\\
{\rm PSU}_6(q)& 1/3 & C_G(x) \text{ not of type } {\rm GL}_3(q^2)\\
& 0.32& C_G(x) \text{ of type } {\rm GL}_3(q^2)\\
{\rm PSp}_6(q) & 0.4 & H \text{ not of type } G_2(q), p = 2\\
 & 0.249 & H \text{ of type } G_2(q), p = 2\\
{\rm PSL}_{n}^{\epsilon}(q)& 9/20 &n\in \{3, 4, 5\} \text{ and } x\in {\rm Inndiag}(G_0)\\
& 3/10 &n\in \{3, 4, 5\} \text{ and } x = \phi, \tau, \text{ or } \phi\tau\\
{\rm PSp}_{4}(q)& 9/20 &x\in {\rm Inndiag}(G_0)\\
& 3/10 &x = \phi \text{ or } \phi\tau\\

\hline
\end{array}
\]
\caption{The value of $\beta$ in Theorem~\ref{thm:low-rank-fprs}.}
\label{t:fprs-classical}
\end{table}
}

\begin{prop}\label{prop:sp8}
Suppose that $G_0 = {\rm PSp}_8(q)$, $q \geqs 3$ and $x\in G$ is an element of prime order $r$. Then $\fpr(x, G/H) < |x^{G_0}|^{-1/3}$. 
\end{prop}

\begin{proof}
If $H$ is not of type ${\rm Sp}_4(q^2)$ or ${\rm Sp}_4(q)\wr S_2$, then the claim follows by Theorem~\ref{thm:timsbound}, so it remains to consider those two cases. The claim can easily be checked using {\sc Magma} for $q\leqs 3$ using the function \texttt{MaxFpr()} provided in Section~\ref{sss:comp-classical}, so for the remainder of the proof we may assume that $q\geqs 4$. As mentioned in Section~\ref{sss:comp-classical}, the computation for $q = 3$ is more time-consuming than other cases we handle using {\sc Magma}, but one can get output in less than an hour.

\vs

\noindent \textbf{Case 1.} \emph{$x$ is semisimple and $r>2$:} First suppose that $H$ is of type ${\rm Sp}_4(q^2)$. Let $i$ be minimal such that $r$ divides $q^i - 1$ and $i_0$ be minimal such that $r$ divides $q^{2i_0}-1$, and note that $i_0 = i/2$ if $i$ is even and $i_0 = i$ otherwise. Note that we may assume that $x\in H$, since otherwise ${\rm fpr}(x, G/H) = 0$. In this case $r$ must divide $|H|$, and so we have $i\in \{1, 2, 4, 8\}$. We need to consider each possibility of $i$ in turn. The calculations for each value of $i$ are straightforward and similar, and so we will only consider the cases where $i = 4$ and $i=1$ here and leave the rest for the reader to check.

\vs 

First suppose that $i = 4$. In this case $i_0 = 2$ and $C_G(x)$ is of type ${\rm GU}_2(q^2)$, ${\rm GU}_1(q^2)^2$, or ${\rm GU}_1(q^2) \times {\rm Sp}_4(q)$. If $C_G(x)$ is of type ${\rm GU}_2(q^2)$, then we may write $x = [\Lambda^2]$ for some $\sigma$-orbit $\Lambda$ and using Lemma~\ref{lem:c3} we compute
\[
|x^{G_I} \cap H| = 2\frac{|{\rm Sp}_4(q^2)|}{|{\rm GU}_2(q^2)|} + \frac{|{\rm Sp}_4(q^2)|}{|{\rm GU}_1(q^2)|}  = 2q^6(q^4 + 1)(q^2 - 1) + q^8(q^2-1)^2(q^4+1)
\]
and
\[
|x^{G_0}| = \frac{|{\rm Sp}_8(q)|}{|{\rm GU}_2(q^2)|}  = q^{14}(q^2 - 1)^2(q^4 + 1)(q^6 - 1),
\]
and so the claim follows.

If $C_G(x)$ is of type ${\rm GU}_1(q^2)^2$, then $x = [\Lambda_1, \Lambda_2]$ for some $\sigma$-orbits $\Lambda_1, 
\Lambda_2$, and we have
\[
|x^{G_I}\cap H| \leqs 4\frac{|{\rm Sp}_4(q^2)|}{|{\rm GU}_1(q^2)|^2} = 4q^8(q^2 - 1)^2(q^4+1)
\]
by Lemma~\ref{lem:c3} and
\[
|x^{G_0}| = \frac{|{\rm Sp}_8(q)|}{|{\rm GU}_1(q^2)|^2}  = q^{16}(q^2 - 1)^3(q^4 + 1)(q^6 - 1),
\]
and so ${\rm fpr}(x, G/H) < |x^{G_0}|^{-1/3}$ here as well.

Finally, if $C_G(x)$ is of type ${\rm GU}_1(q^2)\times {\rm Sp}_4(q)$, then $x = [\Lambda, I_4]$ for some $\sigma$-orbit $\Lambda$, and again using Lemma~\ref{lem:c3} we compute
\[
|x^{G_I}\cap H| \leqs 2\frac{|{\rm Sp}_4(q^2)|}{|{\rm GU}_1(q^2)||{\rm Sp}_2(q^2)|} = 2q^6(q^4 + 1)(q^2 - 1)
\]
and 
\[
|x^{G_0}| = \frac{|{\rm Sp}_8(q)|}{|{\rm GU_1}(q^2)||{\rm Sp}_4(q)|} = q^{12}(q^2 - 1)(q^6 - 1)(q^4 + 1)
\]
and we observe that the desired bound is satisfied.

\vs

We finally consider the case $i = i_0=1$. Note that we have assumed that $x\in H$ and under this assumption we claim that $C_G(x)$ must be of type ${\rm GL}_4(q)$, ${\rm GL}_2(q)\times{\rm Sp}_4(q)$, or ${\rm GL}_2(q)^2$. Let $\hat{x}$ denote the lift of $x$ in ${\rm Sp}_4(q^2)$. Recall from Section~\ref{s:cclasses} that two elements of order $r$ are ${\rm Sp}_4(q^2)$-conjugate if and only if they have the same eigenvalues in $\mathbb{F}_{q^{2i_0}} = \mathbb{F}_{q^2}$. Moreover, we have seen again in Section~\ref{s:cclasses} that since $i_0 = 1$ the eigenvalues of $\hat{x}$ come in inverse pairs $(\lambda, \lambda^{-1})$. Thus, if $\hat{x}$ denotes the lift of $x$ in ${\rm Sp}_4(q^2)$, then we may write $\hat{x} = [\lambda, \lambda^{-1}, \mu, \mu^{-1}]$ with $\lambda, \mu\in \mathbb{F}_{q^2}$ not necessarily distinct.


As explained in Section~\ref{s:cclasses}, the equality $i = i_0$ implies that $\sigma$ and $\sigma_0$-orbits coincide, and in particular, Lemma~\ref{lem:c3} implies that $\hat{x}$ embeds as $[(\lambda, \lambda^{-1})^2, (\mu, \mu^{-1})^2]$ in $G$, so all eigenvalues occur with even multiplicity. It follows that $C_G(x)$ cannot be of type ${\rm GL}_1(q)^4$ or ${\rm GL}_3(q) \times {\rm Sp}_2(q)$, and the claim follows.

We continue by considering each possible centraliser type for $x$ in turn. All cases are entirely similar so we only consider the case where $C_G(x)$ is of type ${\rm GL}_2(q)\times{\rm Sp}_4(q)$ and leave the remaining cases to the reader.

In this case we compute
\[
|x^{G_I}\cap H| = \frac{|{\rm Sp}_4(q^2)|}{|{\rm GL}_1(q^2)||{\rm Sp}_2(q^2)|} =q^6(q^2+1)^2(q+1)(q^4+1)
\]
and 
\[
|x^{G_0}| = \frac{|{\rm Sp}_8(q)|}{|{\rm GL}_2(q)||{\rm Sp}_4(q)|} =  q^{11}(q^4+q^2+1)(q^4+1)(q^2+1)(q+1),
\]
and so the claim holds.

\vs

Now suppose that $H$ is of type ${\rm Sp}_4(q)\wr S_2$ and again let $i$ be minimal such that $r\vert q^i-1$. We may again assume that $x\in H$, and in particular that $r$ divides $|H|$. It follows that $i\in \{1, 2, 4\}$. The computations are straightforward here as well, and so, to avoid unnecessary repetition, we only consider the case $i = 2$ and leave the remaining cases to the reader.

If $i = 2$, then there are at most $11$ possibilities for $C_G(x)$ up to isomorphism. More precisely, $C_G(x)$ is one of the following types:
\begin{align*}
&{\rm GU}_4(q), {\rm GU}_3(q)\times {\rm GU}_1(q), {\rm GU}_3(q)\times {\rm Sp}_2(q), {\rm GU}_2(q)^2, {\rm GU_2}(q)\times {\rm GU}_1(q)^2,\\ &{\rm GU}_2(q)\times {\rm GU}_1(q)\times {\rm Sp}_2(q), {\rm GU}_2(q)\times {\rm Sp}_4(q),{\rm GU}_1(q)^4, {\rm GU}_1(q)^3\times {\rm Sp}_2(q),\\ &{\rm GU}_1(q)^2\times {\rm Sp}_4(q), {\rm GU}_1(q)\times {\rm Sp}_6(q). 
\end{align*}
All cases can be handled in a similar way. For example, let us consider the case where $C_G(x)$ is of type ${\rm GU}_2(q)^2$. As described in Section~\ref{s:cclasses}, if $\hat{x}$ denotes the lift of $x$ in ${\rm Sp}_8(q)$, then $\hat{x}$ is ${\rm Sp}_8(q)$-conjugate to an element of the form $[\Lambda_1^2, \Lambda_2^2]$ for two distinct $\sigma$-orbits $\Lambda_1, \Lambda_2$.

Let $\widehat{H} = {\rm Sp}_4(q) \wr S_2$ denote the lift of $H$ in ${\rm Sp}_8(q)$ and $V_1\oplus V_2$ denote the decomposition of $V$ into $4$-dimensional $\F$-spaces $V_1, V_2$ preserved by $\hat{H}$. Note that since $r$ is odd, we must have $\hat{x} \in \widehat{B} = {\rm Sp}_4(q)^2$. Now there are two $\widehat{H}$-classes of elements that fuse to the ${\rm Sp}_8(q)$-class of $\hat{x}$, namely one with representative $[Y^2]$, where $Y\in {\rm Sp}_4(q)$ is ${\rm Sp}_4(q)$-conjugate to $[\Lambda_1, \Lambda_2]$ and the one with representative $[Z, W]$, where $Z\in {\rm Sp}_4(q)$ is ${\rm Sp}_4(q)$-conjugate to $[\Lambda_1^2]$ and $W\in {\rm Sp}_4(q)$ is ${\rm Sp}_4(q)$-conjugate to $[\Lambda_2^2]$. It now remains to compute the sizes of the $\widehat{H}$-classes of $[Y^2]$ and $[Z, W]$ to obtain a bound on $|x^{G_I}\cap H|$. It is easy to see that the $\widehat{H}$-class of $[Y^2]$ consists of all elements of the form $[A, B]$, with $A, B$ $\widehat{H}$-conjugate to $Y$, and so
\[
|[Y^2]^{\widehat{H}}| = \frac{|{\rm Sp}_4(q)|^2}{|{\rm GU}_1(q)|^4}.
\]
On the other hand, if $[A, B]$ is $\widehat{H}$-conjugate to $[Z, W]$, then since $\widehat{H}$ can permute $V_1$ and $V_2$ either $A \in Z^{\widehat{H}}$ and $B \in W^{\widehat{H}}$, or  $A \in W^{\widehat{H}}$ and $B \in Z^{\widehat{H}}$, and so 
\[
|[Z, W]^{\widehat{H}}| = 2\left(\frac{|{\rm Sp}_4(q)|}{|{\rm GU}_2(q)|}\right)^2.
\]
It follows that
\begin{align*}
|x^{G_I}\cap H| = 2\left(\frac{|{\rm Sp}_4(q)|}{|{\rm GU}_2(q)|}\right)^2 + \frac{|{\rm Sp}_4(q)|^2}{|{\rm GU}_1(q)|^4}
= 2q^6(q-1)^2(q^2+1)^2 + q^8(q-1)^4(q^2+1)^2
\end{align*}
and moreover,
\[
|x^{G_0}| = \frac{|{\rm Sp}_8(q)|}{|{\rm GU}_2(q)|^2} = q^{14}(q-1)^2(q^2+1)^2(q^4+q^2+1)(q^4+1),
\]
so ${\rm fpr}(x, G/H) < |x^{G_0}|^{-1/3}$.

\vs

\noindent \textbf{Case 2.} \emph{$x$ is unipotent and $p$ is odd:} If $H$ of type ${\rm Sp}_4(q^2)$, then~\cite[5.3.2(i)]{BG} implies that the possible Jordan forms on $V$ for an element $x\in H$ are $[J_4^2], [J_2^4]$, and $[J_2^2, J_1^4]$. We can easily compute $|x^{G_I}\cap H|$ and $|x^{G_0}|$ using the information in~\cite[Table B.7]{BG} in every case and deduce that the desired bound holds. For example, suppose that $x$ has Jordan form $[J_2^2, J_1^4]$, and let $W$ denote the natural module of ${\rm Sp}_4(q^2)$. Appealing to~\cite[5.3.2(i)]{BG}, we deduce that $x$ has Jordan form $[J_2, J_1^2]$ on $W$. Moreover, we note that there are two signed partitions corresponding to this Jordan decomposition, so we compute
\[
|x^{G_I}\cap H| \leqs \frac{2|{\rm Sp}_4(q^2)|}{q^6|{\rm Sp}_2(q^2)|} = 2(q^8-1)
\]
and
\[
|x^{G_0}| \geqs \frac{|{\rm Sp}_8(q)|}{q^{11}|{\rm O}^-_2(q)||{\rm Sp}_4(q)|} = \frac{1}{2}q(q-1)(q^2+1)(q^4+1)(q^6-1)
\]
and these bounds are sufficient.

If $H$ is of type ${\rm Sp}_4(q)\wr S_2$, then the possibilities and bounds for $|x^G\cap H|$ and $|x^G|$ are listed in the proof of~\cite[3.2]{B07}, and one can check that those bounds are sufficient, so we omit the details here. 
\vs

\noindent \textbf{Case 3.} \emph{$x$ is an involution:} If $H$ is of type ${\rm Sp}_4(q)\wr S_2$, then one can easily check that the desired bound holds by inspecting the proof of~\cite[2.8, 2.9]{fprIII}, so we now assume that $H$ is of type ${\rm Sp}_4(q^2)$. 

If $x$ is not semisimple with $C_G(x)$ of type ${\rm Sp}_4(q)^2$, then the bounds found in the proof of~\cite[3.2]{fprIII} are sufficient, so we may now assume that $q$ is odd and $C_G(x)$ is of type ${\rm Sp}_4(q)^2$. By~\cite[Section 3.4.2.2]{BG} we have $x\in G_0$. Now~\cite[4.3.10]{KL} tells us that
\[
H\cap G_0 = {\rm PSp}_4(q^2).\la\phi\ra = B.\la \phi\ra,
\]
where $\phi$ is a field automorphism of order 2. We first consider involutions in $x^{G_I}\cap B$. Note that, as explained in~\cite[3.4.2]{BG}, $B$ contains two classes of involutions, namely those with centraliser of type ${\rm Sp}_2(q^2)^2$ and those with centraliser of type ${\rm GL}_2(q^2)$. If $y$ is of type ${\rm Sp}_2(q^2)^2$, then it clearly embeds in $G$ as an involution of type ${\rm Sp}_4(q)^2$, so $y\in x^{G_I}\cap B$. On the other hand, if $y$ is of type ${\rm GL}_2(q^2)$, then the lift $\hat{y}$ of $y$ in ${\rm Sp}_4(q^2)$ has order $4$. Now if $\hat{z}$ denotes the embedding of $\hat{y}$ inside ${\rm Sp}_8(q)$, then inspecting~\cite[Construction 2.4.2]{BG} reveals that $\hat{z}$ also has order $4$, so $z = \hat{z}Z({\rm Sp}_8(q))$ is of type ${\rm GL}_4(q)$. In particular, $y$ embeds in $G_I$ as an involution of type ${\rm GL}_4(q)$, and so $y\not\in x^{G_I}\cap B$. Therefore, we compute
\[
|x^{G_I}\cap B| \leqs \frac{|{\rm PSp}_4(q^2)|}{|{\rm Sp}_2(q^2)|^2} \,\,\,\,\, \text{and} \,\,\,\,\, |x^{G_0}| \geqs \frac{|{\rm PSp}_8(q)|}{2|{\rm Sp}_4(q)|^2}.
\]
Moreover, (68) in~\cite{fprIII} gives us
\[
|x^{G_I}\cap (H\setminus B)| \leqs \frac{|{\rm Sp}_4(q^2)|}{|{\rm Sp}_4(q)|}
\]
and so the claim holds.

\vs

\noindent \textbf{Case 4.} \emph{$x$ is a field automorphism:} In this case, the bounds presented at the start of~\cite[Section 3.2]{fprIII} and~\cite[2.3]{fprIII} are sufficient.
\end{proof}

\begin{prop}\label{prop:gl8}
Suppose that $G_0 = {\rm PSL}^{\epsilon}_{8}(q)$ and $x\in G$ is an element of prime order $r$. Then $\fpr(x, G/H) < |x^{G_0}|^{-1/3}$. 
\end{prop}

\begin{proof}
We only need to consider the case where $H$ is of type ${\rm Sp}_8(q)$, since otherwise the claim follows straight from Theorem~\ref{thm:timsbound}. In this remaining case, the claim can be verified using {\sc Magma} for $q \leqs 2$, and for $q\geqs 3$ one can argue in a similar fashion to the proof of Proposition~\ref{prop:sp8}, so we leave the details for the reader to check. 

If $x$ is a field, graph, or graph-field automorphism, then the bounds in~\cite[8.1]{fprII} suffice, so it is left to check the claim for semisimple and unipotent elements. In those cases we essentially follow the same strategy as in~\cite[8.1]{fprII}. In particular, we note that $H_I = H\cap {\rm PGL}(V) = {\rm GSp}_8(q)$ and the conjugacy class structure of prime order elements in is briefly discussed in~\ref{s:cclasses} and in~\cite[Chapter 3]{BG} in more detail. By using this information and reading off the centraliser structure of prime order elements in ${\rm PGL}_8(q)$ and ${\rm PGSp}_8(q)$ from~\cite[Tables B.3 and B.7]{BG}, we can verify that the claim holds. 

For example, suppose that $x$ is a unipotent involution with Jordan decomposition $[J_2^2, J_1^4]$. Then we observe that $x^{G_I} \cap H_I$ is the union of two $H_I$-classes, namely the classes of type $a_2$ and $c_2$, and so by inspecting~\cite[Table B.7]{BG} we find
\[
|x^{G_I}\cap H| \leqs \frac{|{\rm Sp}_8(q)|}{q^{11}|{\rm Sp}_2(q)||{\rm Sp}_4(q)|} + \frac{|{\rm Sp}_8(q)|}{q^{12}|{\rm Sp}_4(q)|} = (q^8-1)(q^6+q^4+q^2)
\]
and 
\[
|x^{G_0}| \geqs \frac{|{\rm GL}_8(q)|}{q^{9}|{\rm GL}_2(q)||{\rm GL}_4(q)|} \geqs q^{23}(q^5-1)(q^7-1)
\]
so the claim holds. We deal with all other cases in a similar fashion.
\end{proof}

\begin{prop}\label{prop:o10}
Suppose that $G_0 = {\rm P\O}^{\epsilon}_{10}(q)$ and $x\in G$ is an element of prime order $r$. Then $\fpr(x, G/H) < |x^{G_0}|^{-1/3}$. 
\end{prop}

\begin{proof}
In this case we may assume that $H$ is of type ${\rm GL}_5^{\epsilon}(q)$, as the claim follows by Theorem~\ref{thm:timsbound} in all other cases. The claim can be verified using {\sc Magma} for $q \in \{2, 3\}$, and for $q\geqs 4$ the bounds in the proof of~\cite[3.3]{fprIII} are sufficient.
\end{proof}

\begin{prop}\label{prop:gl6}
Suppose that $G_0 = {\rm PSL}^{\epsilon}_{6}(q)$ and $x\in G$ is an element of prime order $r$. Then either  ${\rm fpr}(x, G/H) < |x^{G_0}|^{-1/3}$, or one of the following holds:
\begin{enumerate}
\item[\rm (i)] $\epsilon = +$, $x$ is semisimple with $C_G(x)$ of type ${\rm GL}_3(q)^2$, and $\fpr(x, G/H) < |x^{G_0}|^{-0.32}$
\vs
\item[\rm (ii)] $\epsilon = -$, $x$ is semisimple with $C_G(x)$ of type ${\rm GL}_3(q^2)$, and $\fpr(x, G/H) < |x^{G_0}|^{-0.32}$.
\end{enumerate}
\end{prop}

\begin{proof}
Here we may assume that $H$ is of type ${\rm Sp}_6(q)$ since the claim follows directly by Theorem~\ref{thm:timsbound} in all other cases. 

\vs

\noindent \textbf{Case 1.} \emph{$x$ is semisimple and $r>2$:} Let $i$ be minimal such that $r\vert q^i-1$. Again, we may assume that $x\in H$ and in particular that $r$ divides $|H|$, since otherwise we have ${\rm fpr}(x, G/H) = 0$. Thus, we have $i\in \{1, 2, 3, 4, 6\}$. The embedding here is straightforward so we will only give details for the cases where $i \in \{1, 4\}$ and leave the remaining cases for the reader.

First suppose $i = 4$. In this case $C_G(x)$ is of type ${\rm GL}_1(q^4) \times {\rm GL}^{\epsilon}_2(q)$ and we have
\[
|x^{G_I}\cap H| \leqs \frac{|{\rm Sp}_6(q)|}{|{\rm GU}_1(q^2)||{\rm Sp}_2(q)|} = q^8(q^2-1)(q^6-1)
\]
and 
\begin{align*}
|x^{G_0}| \geqs \frac{|{\rm GL}_6(q)|}{|{\rm GL}_1(q^4)||{\rm GU}_2(q)|} = q^{14}(q-1)^2(q^3-1)(q^5-1)(q^4 + q^2+1).
\end{align*}
The claim follows.

Now assume $i=1$. If $\epsilon = +$, then there are six possibilities for $C_G(x)$ for an element $x\in H$. More precisely, $C_G(x)$ can be of type ${\rm GL}_3(q)^2$, ${\rm GL}_2(q)^2\times {\rm GL}_1(q)^2$, ${\rm GL}_2(q)^3$, ${\rm GL}_2(q)\times {\rm GL}_1(q)^4$, ${\rm GL}_1(q)^6$, or ${\rm GL}_4(q)\times {\rm GL}_1(q)^2$. All cases can be handled in a similar way. For example, if $C_G(x)$ is of type ${\rm GL}_3(q)^2$, then we find
\[
|x^{G_I}\cap H| = \frac{|{\rm Sp}_6(q)|}{|{\rm GL}_3(q)|} = q^6(q+1)(q^2+1)(q^3+1)
\]
and
\[
|x^{G_0}| = \frac{|{\rm GL}_6(q)|}{|{\rm GL}_3(q)|^2} = q^9(q^4+q^3+q^2+q+1)(q^2+1)(q^3+1)
\]
and one can check that $\log|x^{G_I}\cap H|/\log|x^{G_0}| < 0.68$.

Similarly, if $C_G(x)$ is of type ${\rm GL}_2(q)^2\times {\rm GL}_1(q)^2$, then we have
\[
|x^{G_I}\cap H| \leqs \frac{|{\rm Sp}_6(q)|}{|{\rm GL}_2(q)||{\rm GL}_1(q)|} = q^8(q+1)^2(q^2-1)(q^4+q^2+1) 
\]
and
\[
|x^{G_0}| \geqs \frac{|{\rm GL}_6(q)|}{|{\rm GL}_2(q)|^2|{\rm GL}_1(q)|}
 = q^{13}(q+1)(q^2+q+1)(q^2+1)(q^4+q^3+q^2+q+1)(q^4+q^2+1)
\]
and so the claim is true in this case as well. We leave the remaining cases for the reader to check.

Now suppose that $\epsilon = -$. In this case the possible centraliser types for a prime order element $x\in H$ are ${\rm GL}_3(q^2)$, ${\rm GL}_2(q^2)\times {\rm GL}_1(q^2)$, ${\rm GL}_2(q^2)\times {\rm GU}_1(q)$, ${\rm GL}_1(q^2)^2\times {\rm GU}_2(q)$, ${\rm GL}_1(q^2)^3$, and ${\rm GL}_1(q^2)\times {\rm GU}_4(q)$, and again we can verify that the desired bound holds by treating each possibility in turn. For example, if $C_G(x)$ is of type ${\rm GL}_3(q^2)$, then 
\[
|x^{G_I}\cap H| = \frac{|{\rm Sp}_6(q)|}{|{\rm GL}_3(q)|} = q^6(q+1)(q^2+1)(q^3+1)
\]
and
\[
|x^{G_0}| = \frac{|{\rm GU}_6(q)|}{|{\rm GL}_3(q^2)|} = q^9(q+1)(q^3+1)(q^5+1) 
\]
and one can check that $\log|x^{G_I}\cap H|/\log|x^{G_0}| < 0.68$.

\vs

\noindent \textbf{Case 2.} \emph{$x$ is unipotent:} By appealing to Section~\ref{s:cclasses}, we find that the Jordan decompositions corresponding to an element $x\in H$ are 
\[
[J_6], [J_4, J_2], [J_4, J_1^2], [J_3^2], [J_2^3], [J_2^2, J_1^2], [J_2, J_1^4], 
\]
so we check each possibility in turn to confirm that the required bound holds. For example, if $x$ is of type $[J_2, J_1^4]$, then we get
\[
|x^{G_I}\cap H| \leqs \frac{2|{\rm Sp}_6(q)|}{q^5|{\rm Sp}_4(q)|} \leqs 2(q^6-1)
\]
and 
\[
|x^{G_0}| \geqs \frac{|{\rm GU}_6(q)|}{q^9|{\rm GU}_4(q)||{\rm GU}_1(q)|} = (q^4 - q^3+q^2-q+1)(q^6-1)
\]
and hence ${\rm fpr}(x, G/H) < |x^{G_0}|^{-1/3}$. The other cases can be dealt with in a similar fashion by inspecting ~\cite[Tables B.3 and B.7]{BG}, so we omit the details here.

\vs

\noindent \textbf{Case 3.} \emph{$x$ is a graph, field or graph-field automorphism or a semisimple involution:} Here the bounds in the proof of~\cite[8.1]{fprII} suffice.
\end{proof}

\begin{rem} The cases in (i) and (ii) of Proposition~\ref{prop:gl6} are genuine exceptions. In particular, the bounds we have derived are precise and one can check that the functions
\[
f(x) = \frac{\log (x^6(x+1)(x^2+1)(x^3+1))}{\log (x^9(x^4 + x^3 + x^2+x+1)(x^2+1)(x^3+1))}
\]
and
\[
g(x) = \frac{\log(x^6(x+1)(x^2+1)(x^3+1))}{\log(x^9(x+1)(x^3+1)(x^5+1))}
\]
both are decreasing on $(0, \infty)$ with a horizontal asymptote at $2/3$, so we cannot bound ${\rm fpr}(x, G/H)$ by $|x^{G_0}|^{-1/3}$.
\end{rem}

\begin{prop}\label{prop:sp6}
Suppose that $G_0 = {\rm PSp}_{6}(q)$ $q$ is even and $x\in G$ is an element of prime order $r$. Then either $\fpr(x, G/H) < |x^{G_0}|^{-0.4}$, or $H$ is of type $G_2(q)$. 
\end{prop}

\begin{proof}
The claim can be easily verified using {\sc Magma} for $q \leqs 4$, and so we may assume that $q\geqs 8$ for the remainder of the proof. By inspecting~\cite[Tables 8.28 and 8.29]{BHR}, we find that if $H$ is not of type $G_2(q)$, then $H_0 = H\cap G_0 \in \{{\rm Sp}_2(q)\wr S_3, {\rm Sp}_2(q^3) {:} 3, {\rm Sp}_6(q_0)\}$, where $q = q_0^k$ for some prime $k$.

First suppose that $H_0 = {\rm Sp}_2(q)\wr S_3= B.S_3$. If $x$ is a field automorphism of order $r$, then the bounds in~\cite[Section 2.3]{fprIII} give
\[
|x^G\cap H| < 48q^{9(1-1/r)} \,\,\,\,\, \text{and} \,\,\,\,\, |x^G| > \frac{1}{4}q^{21(1-1/r)}
\]
and so the claim holds. 

Now suppose that $x\in {\rm PGL}(V)$. There are two cases to consider depending whether or not $x^G \cap (H\setminus B) \neq \emptyset$. First suppose that $x^G \cap (H\setminus B) \neq \emptyset$. In this case, we note that either $r = 2$ or $r = 3$. If $r = 2$, then by inspecting the proof of~\cite[2.8]{fprIII}, we deduce that $\log|x^{G_I}\cap H|/\log|x^{G_0}| < 0.543$, so we may now assume that $r = 3$. Here, we may assume $x = b\pi = (x_1, x_2, x_3)\pi$, where $b\in B$ and $\pi\in S_3$ is a $3$-cycle. It then follows from the proof of~\cite[4.5]{LSh} that $x$ is $B$-conjugate to $\pi$. Now note that if $y = (X, Y, Z)\in C_B(x)$, then since $\pi$ permutes $X, Y$, and $Z$, we must have $X = Y = Z$, and so
\[
|x^{G_I}\cap H| = \frac{6|{\rm Sp}_2(q)|^3}{3|{\rm Sp}_2(q)|} = 2q^2(q^2-1)^2.
\]

We now note that with respect to a symplectic basis $\{e_1, e_2, e_3, f_1, f_2, f_3\}$, $\pi$ embeds in ${\rm Sp}_6(q)$ as the block-diagonal matrix $[A, A]$, where
\begin{equation}\label{eq:a}
A = 
\begin{pmatrix}
0 & 0 & 1 \\
1 & 0 & 0 \\
0&1&0 
\end{pmatrix}.
\end{equation}
Since $r = 3$ and $q$ is even, it follows that $x$ is semisimple. Moreover, since $q$ is even, $3$ must either divide $q-1$ or $q+1$, and so if $i$ is minimal such that $3$ divides $q^i-1$, we must have $i=1$ or $2$. 

Now note that $x$ has eigenvalues $1^2, (\omega)^2, (\omega^2)^2$, where $\omega\in \mathbb{F}_{q^i}$ is a primitive cube root of unity, and so $C_G(x)$ is of type ${\rm GL}_2^{\epsilon}(q)\times {\rm Sp}_2(q)$, where $\epsilon = (-1)^{i+1}$. Therefore,
\[
|x^{G_0}| \geqs \frac{|{\rm Sp}_6(q)|}{|{\rm Sp}_2(q)||{\rm GU}_2(q)|} = q^7(q-1)(q^2+1)(q^4+q^2+1)
\]
and the claim holds.

The case $x^G\cap H \subseteq B$ is straightforward. First suppose that $x$ is semisimple and $i$ is minimal such that $r$ divides $q^i-1$. We may again assume that $x\in H$, in which case $r$ divides $|H|$, and we must have $i\in \{1, 2\}$. If $i = 2$, then $C_G(x)$ must be of type 
\[
{\rm GU}_3(q), {\rm GU}_2(q)\times {\rm Sp}_2(q), {\rm GU}_2(q)\times {\rm GU}_1(q), {\rm GU}_1(q)\times {\rm Sp}_4(q), {\rm GU}_1(q)^2\times {\rm Sp}_2(q), \text{ or } {\rm GU}_1(q)^3. 
\]
We can now compute bounds for $|x^{G_I}\cap H|$ and $|x^{G_0}|$ using~\cite[Table B.7]{BG}. For example, if $C_G(x)$ is of type ${\rm GU}_2(q)\times {\rm GU}_1(q)$, then as discussed in Section~\ref{s:cclasses}, we may assume that the lift $\hat{x}$ of $x$ in ${\rm Sp}_6(q)$ is $[(\Lambda_1)^2, \Lambda_2]$ for some $\sigma$-orbits $\Lambda_1, \Lambda_2$. Now we find that $\hat{y} = [Y_1, Y_2, Y_3] \in \hat{x}^{{\rm Sp}_2(q)^3}$ if and only if $Y_i$ is ${\rm Sp}_2(q)$-conjugate to $[\Lambda_2]$ for some $i$ and $Y_j$ is ${\rm Sp}_2(q)$-conjugate to $\Lambda_1$ for $j\neq i$. Therefore,
\[
|x^{G_I}\cap H| \leqs \frac{3|{\rm Sp}_2(q)^3|}{|{\rm GU}_1(q)|^3} = 3q^3(q-1)^3
\]
and
\[
|x^{G_0}| \geqs \frac{|{\rm Sp}_6(q)|}{|{\rm GU}_2(q)||{\rm GU}_1(q)|} = q^8(q-1)^2(q^2+1)(q^4+q^2+1)
\]
and the claim follows. We can compute in the same way in all other cases. If $i=1$, then the argument is almost identical, so we omit the details to avoid unnecessary repetition.

Now suppose that $x\in H$ is unipotent. Since $x^G\cap H \subseteq B$, we deduce that $x$ must be a $b_1, c_2$ or $b_3$ involution by consulting~\cite[3.4.14]{BG}. We consider each possibility in turn by inspecting~\cite[Table B.7]{BG} and we obtain the bounds in Table~\ref{t:sp6invols}, which allow us to verify that the desired claim holds.
{\small
\begin{table}
\[
\begin{array}{lll} \hline
x & |x^G\cap H| \leqs & |x^G| \geqs \\ \hline
b_1 & 3(q^2-1) & q^6-1\\
c_2 & 2(q^2-1)^2 & (q^4-1)(q^6-1)\\
b_3 & (q^2-1)^3 & q^2(q^4-1)(q^6-1)\\
\hline
\end{array}
\]
\caption{Fixed point ratio bounds for unipotent involutions arising in Proposition~\ref{prop:sp6}.}
\label{t:sp6invols}
\end{table}
}

\vs 

Now suppose that $H_0 = {\rm Sp}_2(q^3).3 = B.3$. If $x$ is a field automorphism of order $r$, then the proof of~\cite[3.2]{fprIII} gives
\[
|x^G\cap H| \leqs 2q^{9(1-1/r)} \,\,\,\,\, \text{and} \,\,\,\,\, |x^G| > \frac{1}{4}q^{21(1-1/r)}
\]
and the claim follows.

Now suppose that $x\in {\rm PGL}(V)$. If $x^G \cap (H\setminus B) \neq \emptyset$, then $x$ acts on ${\rm soc}(H_0) = {\rm Sp}_2(q^3)$ as a field automorphism of order $3$, and so we get
\[
|x^G\cap H| \leqs 2\frac{|{\rm Sp}_2(q^3)|}{|{\rm Sp}_2(q)|} = 2q^2(q^4+q^2+1).
\]
Moreover, with respect to the symplectic basis $\{e_1, e_1^q, e_1^{q^2}, f_1, f_1^q, f_1^{q^2}\}$, $x$ embeds in ${\rm Sp}_6(q)$ as $[A, A]$, where $A$ is the matrix defined in~\eqref{eq:a}. By arguing as in the case $H_0 = {\rm Sp}_2(q)\wr S_3$, we get
\[
|x^G| \geqs \frac{|{\rm Sp}_6(q)|}{|{\rm Sp}_2(q)||{\rm GU}_2(q)|} =q^7(q-1)(q^2+1)(q^4+q^2+1).
\]
Therefore, the required bound is satisfied and it now remains to deal with the case where $x^G\cap H\subseteq B$.

First suppose that $x$ is semisimple (so $r$ is odd), and let $i$ be minimal such that $r\vert q^i-1$ and $i_0$ be minimal such that $r\vert q^{3i_0} - 1$. Note that $i_0 = i/3$ if $3$ divides $i$, and $i = i_0$ otherwise. We may again assume that $x\in H$, so $r$ divides $|H|$, and in particular $i\in \{1, 2, 3, 6\}$. All cases are similar, so we only consider the cases where $i = 6$ and $i = 2$ here.

First suppose that $i = 6$, so $i_0 = 2$. In this case, $x$ is $G_I$-conjugate to $[\Lambda]Z$ for some $\sigma$-orbit $\Lambda$, which as discussed in~\cite[Section 5.3.1]{BG} is a union of $3$ distinct $\sigma_0$-orbits, where $\sigma_0: \alpha \mapsto \alpha^{q^3}$. Therefore, $C_G(x)$ is of type ${\rm GU}_1(q^3)$, and we compute
\[
|x^{G_I}\cap H| \leqs \frac{|{\rm Sp}_2(q^3)|}{|{\rm GU}_1(q^3)|} = q^3(q^3-1)
\]
and 
\[
|x^{G_0}| \geqs \frac{|{\rm Sp}_6(q)|}{|{\rm GU}_1(q^3)|} = q^9(q^2-1)(q^3-1)(q^4-1)
\]
and so the claim holds.

Similarly, if $i = i_0= 2$, then $C_G(x)$ is of type ${\rm GU}_1(q)^3$, and so
\[
|x^{G_I}\cap H| \leqs \frac{|{\rm Sp}_2(q^3)|}{|{\rm GU}_1(q^3)|}  = q^3(q^3-1)
\]
and 
\[
|x^{G_0}| \geqs \frac{|{\rm Sp}_6(q)|}{|{\rm GU}_1(q)|^3} =  q^9(q-1)^3(q^2+1)(q^4+q^2+1)
\]
and we deduce that ${\rm fpr}(x, G/H) < |x^{G_0}|^{-0.4}$ as desired.

Finally, if $x\in H$ is unipotent, then it embeds in $G$ as a $b_3$ involution and we compute
\[
|x^{G_I}\cap H| = q^6-1
\]
and 
\[
|x^{G_0}| \geqs q^2(q^4-1)(q^6-1)
\]
Therefore, the claim is true for all elements $x\in H$.

\vs

We are now left with the case where $H$ is of type ${\rm Sp}_6(q_0)$, and $q = q_0^k$ for some prime $k$. If $x\in {\rm PGL}(V)$, then we obtain the required bound by inspecting \cite[Table B.7]{BG} in a similar fashion to the previous two cases, so we leave the details to the reader. 

Now assume that $x$ is a field automorphism of order $r$. If $r\neq k$, then~\cite[4.9.1]{CFSGIII} implies that
\[
|x^G\cap H| \leqs |x^{{\rm Sp}_6(q_0)}| < 2q_0^{21(1-1/r)}
\]
and~\cite[3.48]{fprII} gives
\[
|x^G| > \frac{1}{4}q_0^{21(1-1/r)}
\]
and it can be checked that those bounds are sufficient.

Suppose now that $r = k$. If $r > 2$, then one can check that the trivial bound
\[
|x^G\cap H| \leqs |{\rm Sp}_6(q_0)|
\]
and
\begin{equation}\label{eq:xgC5}
|x^G| > \frac{|{\rm Sp}_6(q)|}{|{\rm Sp}_6(q_0)|}
\end{equation}
suffice. Otherwise, if $r = k = 2$, then~\cite[4.3]{FGS} gives us
\begin{align*}
|x^G\cap H| \leqs i_2({\rm Sp}_6(q_0)) + 1 &\leqs 1 + \frac{|{\rm Sp}_6(q_0)|}{q_0^7|{\rm Sp}_2(q_0)|} + \frac{|{\rm Sp}_6(q_0)|}{q_0^8|{\rm Sp}_2(q_0)|} + \frac{|{\rm Sp}_6(q_0)|}{q_0^5|{\rm Sp}_4(q_0)|} + \frac{|{\rm Sp}_6(q_0)|}{q_0^6|{\rm Sp}_2(q_0)|}\\
& = (q_0^2+q_0+1)(q_0^4-1)(q_0^6-1) + q_0^6,
\end{align*}
where $i_2({\rm Sp}_6(q_0))$ denotes the number of involutions in ${\rm Sp}_6(q_0)$. The result follows via~\eqref{eq:xgC5}, so the proof is complete.
\end{proof}

\begin{prop}\label{prop:o7}
Suppose that $G_0 = {\O}_{7}(q)$, $q \geqs 5$ and $x\in G$ is an element of prime order $r$. Then either $\fpr(x, G/H) < |x^{G_0}|^{-0.4}$, or $G$ is of type $G_2(q)$. 
\end{prop}

\begin{proof}
The claim can be verified using {\sc Magma} for $q = 5$, so we may assume that $q\geqs 7$ for the remainder of the proof. By inspecting~\cite[Tables 8.39-8.40]{BHR} we see that if $H$ is not of type $G_2(q)$, then $H$ must be of type ${\rm O}_1(q)\wr S_7, {\rm O}_7(q_0)$ with $q = q_0^k$ for some prime $k$, or ${\rm Sp}_6(2)$. We treat each possibility in turn. 

If $H$ is of type  ${\rm Sp}_6(2)$, then it follows directly from~\cite[2.15]{fprIV} that ${\rm fpr}(x) < |x^G|^{-1/2}$, and so the claim holds.

If $H$ is of type ${\rm O}_1(q)\wr S_7$, then the bounds in~\cite[2.10]{fprIII} suffice.

Finally, if $H$ is of type ${\rm O}_7(q)$, then we obtain the result by inspecting the proof of~\cite[5.3]{fprII} and~\cite[Table B.12]{BG}, in a similar fashion to the proof of Proposition~\ref{prop:sp6}. We omit the details.
\end{proof}

\begin{prop}\label{prop:gl5}
Suppose that $G_0 = {\rm PSL}^{\epsilon}_{5}(q)$ and $x\in G_I$ is an element of prime order $r$. Then $\fpr(x, G/H) < |x^{G_0}|^{-9/20}$. 
\end{prop}

\begin{proof}
The claim can easily be verified using {\sc Magma} when $q\leqs 5$, so for the remainder we will assume that $q\geqs 7$. By inspecting~\cite[Tables 8.18-8.21]{BHR}, we deduce that  the possible types of $H$ are the ones listed in Table~\ref{t:gl5}. Moreover, it follows from Lemma~\ref{lem:class-bounds} that $|x^{G_0}| >  \frac{1}{2}\left(\frac{q}{q+1}\right)q^{8}$.

{\small
\begin{table}
\[
\begin{array}{lll} \hline
G_0& \text{type of } H & \text{Conditions} \\ \hline
{\rm L}_5^{\epsilon}(q)& {\rm GL}_1^{\epsilon}(q)\wr S_5&\\
& {\rm GL}_1^{\epsilon}(q^5)&\\
& 5^2.{\rm Sp}_2(5)& p\equiv q\equiv \epsilon\imod{5}\text{ or } \epsilon = -, q = p^2 \text{ and } p\equiv 2, 3 \imod{5}\\
& {\rm O}_5(q)& q\, \text{odd}\\
& {\rm GU}_5(q_0)& q = q_0^k \text{ with } k=2 \text{ if } \epsilon = + \text{ and }k \text{ odd prime if } \epsilon = -\\
& {\rm L}_2(11)&q = p \equiv \epsilon, 3\epsilon, 4\epsilon, 5\epsilon, 9\epsilon \imod{11}\\
& {\rm U}_4(q)& q = p\equiv \epsilon \imod{6}\\
{\rm L}_5(q)& {\rm GL}_5(q_0)& q = q_0^k, k\,\text{prime}
\\\hline
\end{array}
\]
\caption{Non-standard subgroups of almost simple groups with socle ${\rm L}_5^{\epsilon}(q)\, (q\geqs 7)$.}
\label{t:gl5}
\end{table}
}

First note that if $H$ is of type $5^{2}.{\rm Sp}_2(5)$, then since $q\geqs 7$ and $q\equiv 1\imod{5}$, the required bound holds, since $|x^G\cap H| \leqs |H|$. 

Similarly, if $H$ is of type ${\rm L}_2(11)$ or ${\rm U}_4(2)$, then we can construct ${\rm Aut}(H)$ using the {\sc Magma} command \texttt{AutomorphismGroupSimpleGroup()}, and we observe that 
\[
|x^G\cap H| \leqs \max{\{i_s({\rm Aut}(H)) \, : \, s \in \pi(|{\rm Aut}(H)|)\}} \leqs 5184,
\]
where we recall that $i_s({\rm Aut}(H))$ denotes the number of elements of prime order $s$ in ${\rm Aut}(H)$. This is sufficient unless $q = 7$ and $\epsilon = +$ and $H$ is of type $U_4(2)$. In this case, we observe that if $r \neq 5$, then $i_r({\rm Aut}(H)) \leqs 891$, which is sufficient. 
 
 On the other hand, if $r = 5$, then $x$ is semisimple and the minimal $i$ such that $r\vert q^i-1$ is $4$. This implies that $C_G(x)$ is not of type ${\rm GL}_4^{\epsilon}(q)\times {\rm GL}_1^{\epsilon}(q)$, so $|x^{G_0}| > \frac{1}{2}\left( \frac{q}{q+1}\right)q^{12}$, and the claim is again satisfied.

Now assume that $H$ is of type ${\rm GL}_1^{\epsilon}(q)\wr S_5$. Then $H$ preserves a decomposition $V_1\oplus \cdots \oplus V_5$ of the natural module $V$. Recall from Lemma~\ref{lem:class-bounds} that $|x^{G_0}| > \frac{1}{2}\left(\frac{q}{q+1}\right)^2q^{14}$, unless $x$ is semisimple and $C_G(x)$ is of type ${\rm GL}^{\epsilon}_s(q)\times {\rm GL}^{\epsilon}_{n-s}(q)$ for $s\in \{3, 4\}$, or $x$ is unipotent with Jordan form $[J_2^s, J_1^{5-2s}]$ for $s\in \{1, 2\}$. Therefore, if $x$ is not in those classes of prime order elements, we may use the trivial bound 
\begin{equation}\label{eq:gl5}
|x^G\cap H| \leqs |H| < \left(\frac{1}{2}\left(\frac{q}{q+1}\right)^2q^{14}\right)^{11/20} < |x^{G_0}|^{11/20}.
\end{equation}

We may now assume that $|x^{G_0}| \leqs \frac{1}{2}\left(\frac{q}{q+1}\right)^2q^{14}$. Recall that~\cite[4.2.9]{KL} implies that ${\rm PGL}(V) \cap H \leqs (q-\epsilon)^4.S_5 = B.S_5$. First assume that $x$ is unipotent. Then $x^{G_I}\cap B = \emptyset$, since $r$ does not divide $|B|$, so we may write $x = b\pi$ for some $b\in B$ and $\pi = (r^h, 1^{5-hr}) \in S_5$. Note that~\cite[3.11]{fprII} implies that $x$ lifts to a unique element $\hat{x} = (\hat{x}_1, \ldots, \hat{x}_5)\pi\in \hat{B}\pi$ of order $r = p$, where $\widehat{B} = {\rm GL}^{\epsilon}_1(q)^5$. Moreover, it is shown in the proof of~\cite[4.5]{LSh} that $\hat{x}$ is $\widehat{B}$-conjugate to $(1, \ldots, 1, \hat{x}_{hr+1}, \ldots, \hat{x}_5)\pi$, and as demonstrated in the proof of~\cite[2.6]{fprIII}, $\hat{x}$ has Jordan form $[J_p^{h + b_p}, J_{p-1}^{b_{p-1}}, \ldots, J_1^{b_1}]$, where $[J_p^{b_p}, \ldots, J_1^{b_1}]$ is the Jordan form of the restriction of $x$ on $V_{hr+1}\oplus \cdots \oplus V_5$. 

From this we deduce that if $\hat{x}$ has Jordan form $[J_2, J_1^3]$, then $\pi$ must be a $2$-cycle, and so we compute
\[
|x^{G_I}\cap H| = \binom{5}{2}(q-\epsilon)
\]
and since $|x^{G_0}| > \frac{1}{2}\left(\frac{q}{q+1}\right)q^8$, we clearly have ${\rm fpr}(x, G/H) < |x^{G_0}|^{-9/20}$, as claimed. 

Similarly, if $\hat{x}$ has Jordan form $[J_2^2, J_1]$, then $\pi$ must be a product of two transpositions, and so
\[
|x^{G_I}\cap H| = \frac{1}{2}\binom{5}{2}\binom{3}{2}(q-1)^3
\]
which is again sufficient.

Now assume that $x$ is semisimple. Recall that under our assumption $C_G(x)$ is of type ${\rm GL}^{\epsilon}_4(q)\times {\rm GL}^{\epsilon}_1(q)$ or ${\rm GL}^{\epsilon}_3(q)\times {\rm GL}^{\epsilon}_2(q)$. First assume that $C_G(x)$ is of type ${\rm GL}^{\epsilon}_4(q)\times {\rm GL}^{\epsilon}_1(q)$, so $\hat{x}$ has eigenvalues $(\lambda)^4, \mu\in \F$. If $x^{G_I}\cap H\subseteq B$, then $|x^G\cap H|$ is bounded above by the number of ways of arranging the eigenvalues of $\hat{x}$ on the main diagonal, so we get $|x^{G_I}\cap H| \leqs 5$, and the bound quickly follows. We may now assume that $x^{G_I} \cap (H\setminus B) \neq \emptyset$. We may gain write $x = b\pi$ for some $b\in B$ and $\pi = (r^h, 1^{5-hr}) \in S_5$. Note that $\pi$ embeds in ${\rm GL}(V)$ as a permutation matrix, and so we quickly deduce that if $\hat{x} = [\lambda I_4, \mu]$, then $\pi$ must be a transposition, and consequently, $x$ must embed in $G$ as a $t_1$-involution. Therefore, we get
\[
|x^{G_I}\cap H| = |x^{G_I}\cap B| + |x^{G_I} \cap (H\setminus B)| \leqs 5 + 10(q-\epsilon)
\]
and so the claim is true.

If $C_G(x)$ is of type ${\rm GL}^{\epsilon}_3(q)\times {\rm GL}^{\epsilon}_2(q)$, then using the same arguments we find that 
\[
|x^{G_I}\cap B| \leqs \binom{5}{2} \,\,\, \text{and} \,\,\, |x^{G_I}\cap (H\setminus B)| = \frac{1}{2}\binom{5}{2}\binom{3}{2}(q-1)^3
\]
and since $|x^{G_0}| > \frac{1}{2}\left(\frac{q}{q+1}\right)q^8$ the result follows.

If $H$ is of type ${\rm GL}_1^{\epsilon}(q^5).5$, then it follows directly from~\cite[5.3.2]{BG} that there are no elements in $H$ with $|x^{G_0}| \leqs  \frac{1}{2}\left(\frac{q}{q+1}\right)q^{12}$, and so we have
\[
|x^G\cap H| \leqs |H| < \left( \frac{1}{2}\left(\frac{q}{q+1}\right)q^{12}\right)^{11/20} < |x^{G_0}|^{11/20}
\]
for all $x\in H$, as required.

Finally, if $H$ is of type ${\rm O}_5(q)$, ${\rm GL}_5^{\epsilon'}(q_0)$, then it is straightforward to verify the required bound using the relevant tables in~\cite[Appendix B]{BG}, as in the proof of Proposition~\ref{prop:sp6}. 
\end{proof}

\begin{prop}\label{prop:gl4}
Suppose that $G_0 = {\rm PSL}^{\epsilon}_{4}(q)$ and $x\in G_I$ is an element of prime order $r$. Then $\fpr(x, G/H) < |x^{G_0}|^{-9/20}$.
\end{prop}

\begin{proof}
If $q \leqs 11$, then we verify that the bound holds using {\sc Magma}, so we may assume for the remainder of the proof that $q \geqs 13$. By inspecting~\cite[Tables 8.8-8.11]{BHR} we deduce that the possible types of $H$ are the ones listed in Table~\ref{t:gl4}. Moreover, Lemma~\ref{lem:class-bounds} implies that $|x^{G_0}| > \frac{1}{2}\left( \frac{q}{q+1}\right)q^6$ for any prime order element $x\in G_I$.

{\small
\begin{table}
\[
\begin{array}{lll} \hline
G_0& \text{type of } H & \text{Conditions} \\ \hline
{\rm L}_4^{\epsilon}(q)& {\rm GL}_1^{\epsilon}(q)\wr S_4&\\
& {\rm GU}_4(q_0)& q = q_0^k \text{ with } k=2 \text{ if } \epsilon = + \text{ and }k \text{ odd prime if } \epsilon = - \\
&2^4.{\rm Sp}_4(2)& p = q\equiv \epsilon, 5\epsilon\imod{8}\\
& {\rm O}^{+}_4(q)& q\, \text{odd}\\
& {\rm O}^{-}_4(q)& q\, \text{odd}\\
& {\rm U}_4(2)& q = p\equiv \epsilon \imod{6}\\
& {\rm L}_2(7)& q = p\equiv \epsilon, 2\epsilon, 4\epsilon \imod{7}\\
& A_7& q = p\equiv \epsilon, 2\epsilon, 4\epsilon \imod{7}\\
{\rm L}_4(q)& {\rm GL}_4(q_0)& q = q_0^k, k\text{ prime}\\
 \hline
\end{array}
\]
\caption{Non-standard subgroups of almost simple groups with socle ${\rm L}_4^{\epsilon}(q)\, (q\geqs 8)$.}
\label{t:gl4}
\end{table}
}

If $H\in \mathscr{S}$, then using the {\sc Magma} command \texttt{AutomorphismGroupSimpleGroup()}, we can construct ${\rm Aut}(H)$ and obtain bounds for $m = \max\{i_s({\rm Aut}(H)) \, : \, s \in \pi(|{\rm Aut}(H)|)\}$. For $q > 13$ we find
\begin{equation}\label{eq:5184}
|x^G\cap H| \leqs m \leqs 5184 < |x^{G_0}|^{11/20}
\end{equation}
and this bound is sufficient, so we may now assume that $q = 13$. If $H$ is not of type ${\rm U}_4(2)$, or $r\neq 5$, then again by computing directly using {\sc Magma}, we find that $m\leqs 891$, and so the claim holds. 

Now let us assume that $q = 13$, $r = 5$ and $H$ is of type ${\rm U}_4(2)$. Then $\epsilon = +$ and we note that $r$ does not divide $q-1$, so Lemma~\ref{lem:class-bounds} implies that that $|x^{G_0}| > \frac{1}{4}\left(\frac{q}{q+1}\right)q^8$. The claim is then verified via~\eqref{eq:5184}.

Now assume $H$ is of type ${\rm GL}_1^{\epsilon}(q) \wr S_4$. If $|x^{G_0}| > \frac{1}{2}\left(\frac{q}{q+1}\right)^2q^{10}$, then the trivial bound $|x^G\cap H| \leqs |H|$
is sufficient. We may now assume that $|x^{G_0}| \leqs \frac{1}{2}\left(\frac{q}{q+1}\right)^2q^{10}$. Lemma~\ref{lem:class-bounds} implies that if this holds, then either $x$ is unipotent with Jordan form $[J_2^s, J_1^{4-2s}]$ for $s\in \{1, 2\}$, or $x$ is semisimple and $C_G(x)$ is of type ${\rm GL}^{\epsilon}_s(q) \times {\rm GL}^{\epsilon}_{4-s}(q)$ for $s\in \{2, 3\}$ or ${\rm GL}_2(q^2)$. In those cases we can argue in the same way as in the proof of Proposition~\ref{prop:gl5} and obtain the required bound, so we leave the details to the reader. 

If $H$ is of type $2^4.{\rm Sp}_4(2)$, then the claim follows directly from the bounds established in the proof of~\cite[6.5]{fprII}.

Finally, if $H$ is of type ${\rm GL}_4(q_0)$, ${\rm GU}_4(q_0)$, or ${\rm O}^{\epsilon}_4(q)$ then we can compute directly using the relevant tables in~\cite[Appendix B]{BG}.
\end{proof}

\begin{prop}\label{prop:gl4-outer}
Suppose that $G_0 = {\rm PSL}^{\epsilon}_{4}(q)$ and $x\in G$ is a field, graph, or graph-field automorphism of prime order $r$. Then $\fpr(x, G/H) < |x^{G_0}|^{-3/10}$.
\end{prop}

\begin{proof}
If $q\leqs 11$, then we can verify the claim using {\sc Magma}, so we may assume that $q\geqs 13$ for the remainder of the proof. First suppose that $r$ is odd. Note that we must have $q\neq p$, and so $G$ is of type ${\rm GL}_1^{\epsilon}(q)\wr S_4$, ${\rm GL}_4^{\epsilon'}(q_0)$ or ${\rm O}^{\epsilon''}_4(q)$. In those cases the bounds in the proofs of  the bounds in the proofs of~\cite[2.7]{fprIII},~\cite[8.2]{fprII} and~\cite[5.1]{fprII} respectively are sufficient.

Now suppose that $r = 2$. If $H\in \mathscr{S}$, then again  $q = p$, and so $x$ must be an involutory graph automorphism. The same argument we used in Proposition~\ref{prop:gl4} shows that $|x^{G}\cap H| \leqs 891$. Moreover, Lemma~\ref{lem:class-bounds} implies that $|x^{G_0}| > \frac{1}{4}\left(\frac{q}{q+1}\right) q^4$ and those bounds are sufficient unless $q = 13$. Now suppose that $q = 13$. If $\epsilon = +$, then we find
\[
|x^{G_0}| = \frac{1}{2}\frac{|{\rm PGL}_4(13)|}{|{\rm PGSp}_4(13)|} = 371124
\]  
and one can check that the claim holds. On the other hand, if $\epsilon = -$, then $H$ is not of type ${\rm U}_4(2)$, so we get
\[
|x^{G}\cap H| \leqs i_2({\rm P\Gamma L}_2(7)) \leqs 241
\]
and this bound is sufficient. 

If $H$ is of type ${\rm GL}_1^{\epsilon}(q)\wr S_4$, then the bounds in~\cite[2.7]{fprIII} are sufficient if $x$ is a field or graph automorphism. If $\epsilon = +$ and $x$ is a graph-field automorphism then Lemma~\ref{lem:class-bounds} gives $|x^{G_0}| > \frac{1}{8}\left(\frac{q}{q+1}\right)q^{15/2}$ and this together with the bound $|x^{G}\cap H| < |H|$ is sufficient.

If $H$ is of type $2^4.{\rm Sp}_4(2)$, then $q = p$, and so again, $x$ must be an involutory graph automorphism. As demonstrated in the proof of~\cite[6.6]{fprII}, we have $|x^G\cap H| \leqs 340$, and one can check that this bound suffices. 

Finally, if $H$ is of type ${\rm O}_4^{\epsilon''}(q)$ or ${\rm GL}_4(q_0)$, then the claim can be verified via the bounds in the proofs of~\cite[8.2]{fprII} and~\cite[5.1]{fprII} respectively.
\end{proof}

\begin{prop}\label{prop:gl3}
Suppose that $G_0 = {\rm PSL}^{\epsilon}_{3}(q)$ and $x\in G_I$ is an element of prime order $r$. Then $\fpr(x, G/H) < |x^{G_0}|^{-9/20}$.
\end{prop}

\begin{proof}
The claim can easily be verified using {\sc Magma} if $q\leqs 19$, and so we may assume for the remainder of the proof that $q\geqs 23$. By inspecting~\cite[Tables 8.3-8.6]{BHR} we deduce that  the possible types of $H$ are the ones listed in Table~\ref{t:gl3}.

{\small
\begin{table}
\[
\begin{array}{l l l} \hline
G_0& \text{type of } H & \text{Conditions} \\ \hline
{\rm L}_3^{\epsilon}(q)& {\rm GL}_1(q)\wr S_3&\\
& {\rm GU}_3(q_0)&q = q_0^k \text{ with } k=2 \text{ if } \epsilon = + \text{ and }k \text{ odd prime if } \epsilon = - \\
& 3^2.{\rm Sp}_2(3)& p = q \equiv \epsilon \imod{3}\\
& {\rm O}_3(q)& q\, \text{odd}\\
& {\rm L}_2(7)& q = p\equiv \epsilon, 2\epsilon, 4\epsilon \imod{7}\\
& A_6& q = p\equiv \epsilon, 4\epsilon \imod{15} \text{ or } \epsilon = +, q = p^2 \text{ and } p\equiv 2, 3 \imod{5}\\
{\rm L}_3^{\epsilon}(q)& {\rm GL}_3(q_0)& q = q_0^k, k\,\text{prime}
\\ \hline
\end{array}
\]
\caption{Non-standard subgroups of almost simple groups with socle ${\rm L}_3^{\epsilon}(q)\, (q\geqs 16)$.}
\label{t:gl3}
\end{table}
}

If $H\in \mathscr{S}$, then as in the proofs of Propositions~\ref{prop:gl5}, and~\ref{prop:gl4}, we use the {\sc Magma} command $\texttt{AutomorphismGroupSimpleGroup()}$ to construct ${\rm Aut}(H)$ and compute $m = \max \{i_s(H) \,:\, s \, \text{prime}\}$. If $H$ is of type ${\rm L}_2(7)$, then
\[
|x^{G_I}\cap H| \leqs m = 56
\]
and this bound is sufficient for all $x\in H$, as $|x^{G_0}| > \frac{1}{2}\left(\frac{q}{q+1}\right)q^4$ by Lemma~\ref{lem:class-bounds}. On the other hand, if $H$ is of type $A_6$, then $m = 144$, and since $q\geqs 23$, we deduce that 
\[
|x^{G_I}\cap H| \leqs m = 144 < \left(\frac{1}{2}\left(\frac{q}{q+1}\right)q^4\right)^{11/20} < |x^{G_0}|^{11/20}.
\]

Now suppose that $H$ is of type $3^{2}.{\rm Sp}_2(3)$. Since $q\geqs 23$, the trivial bound
\[
|x^{G_I}\cap H| \leqs |H| < \left(\frac{1}{2}\left(\frac{q}{q+1}\right)q^4\right)^{11/20} < |x^{G_0}|^{11/20}
\]
is sufficient.

Next let us assume that $H$ is of type ${\rm GL}_1^{\epsilon}(q) \wr S_3$. If $|x^{G_0}| > \frac{1}{6}\left(\frac{q}{q+1}\right)^2q^6$, then we obtain the result via the trivial bound $|x^{G_I}\cap H| \leqs |H\cap G_{I}|$, so we may assume that $|x^{G_0}| \leqs \frac{1}{6}\left(\frac{q}{q+1}\right)^2q^6$. In this case, Lemma~\ref{lem:class-bounds} implies that if $x$ is unipotent, then it has Jordan form $[J_2, J_1]$ and if $x$ is semisimple, then $C_G(x)$ is of type ${\rm GL}^{\epsilon}_2(q)\times {\rm GL}^{\epsilon}_1(q)$. If $x$ is unipotent, then we argue in exactly the same way as in the proof of Proposition~\ref{prop:gl5} and we find $|x^{G_I}\cap H| \leqs 3(q+1)$. Moreover, we note that $|x^{G_0}| > \frac{1}{2}\left(\frac{q}{q+1}\right)q^4$, so the claim holds. If $x$ is semisimple, then the argument is also identical with the one in the proof of Proposition~\ref{prop:gl5}, and we get $|x^G\cap H| \leqs 3(q+1) + 3$, which is again sufficient.

If $H$ is of type ${\rm GL}_1^{\epsilon}(q^3)$, then it follows immediately from~\cite[5.3.2]{BG} that $|x^{G_0}| > \frac{1}{6}\left(\frac{q}{q+1}\right)^2q^6$, and so the trivial bound $|x^{G_I}\cap H| \leqs |H\cap G_{I}|$ is sufficient. 

Finally, if $H$ is of type ${\rm GL}^{\epsilon'}_3(q_0)$ or ${\rm O}_3(q)$, then we obtain the required result by computing exact fixed point ratios using the relevant tables in~\cite[Appendix B]{BG}. All cases are very similar, so we will only consider the case where $\epsilon = +$ and $H$ is of type ${\rm GU}_3(q^{1/2})$ here to avoid unnecessary repetition. 

First suppose that $x$ is unipotent. Then $x$ has Jordan form $[J_3]$ or $[J_2, J_1]$ on $V$. If $x$ has Jordan form $[J_3]$, then we find
\[
|x^{G_I}\cap H| \leqs \frac{|{\rm GU}_3(q_0)|}{q_0^2|{\rm GU}_1(q_0)|} = q_0(q_0^2-1)(q_0^3+1) 
\]
and
\[
|x^{G_0}| \geqs \frac{|{\rm GL}_3(q)|}{q^2|{\rm GL}_1(q)|} = q(q^2-1)(q^3-1)
\]
and one can verify that the claim holds. Similarly, if $x$ has Jordan form $[J_2, J_1]$, then we get
\[
|x^{G_I}\cap H| \leqs \frac{|{\rm GU}_3(q_0)|}{q_0^3|{\rm GU}_1(q_0)|^2} = (q_0-1)(q_0^3+1)
\]
and 
\[
|x^{G_0}| \geqs \frac{|{\rm GL}_3(q)|}{q^3|{\rm GL}_1(q)|^2} = (q+1)(q^3-1)
\]
and the required bound holds in this case as well. 

Now suppose that $x$ is semisimple, and let $i$ be minimal such that $r\vert q^i-1$ and $i_0$ be minimal such that $r\vert q_0^{2i_0}-1$. Note that $i = i_0/2$ if $i_0$ is even and $i = i_0$ otherwise. There are two possible centraliser types for $x$, namely ${\rm GL}_2(q)\times {\rm GL}_1(q)$ and ${\rm GL}_1(q)^3$. If $C_G(x)$ is of type ${\rm GL}_2(q)\times {\rm GL}_1(q)$, then we compute
\[
|x^{G_I}\cap H| \leqs q_0^2(q_0^2-q_0+1) \,\,\, \text{and} \,\,\, |x^{G_0}| \geqs q^2(q^2+q+1)
\]
and if $C_G(x)$ is of type ${\rm GL}_1(q)^3$, then 
\[
|x^{G_I}\cap H|\leqs q_0^3(q_0-1)(q_0^2-q_0+1) + q_0^3(q_0^3+1) \,\,\, \text{and} \,\,\, |x^{G_0}|\geqs q^3(q+1)(q^2+q+1)
\]
and one can check that those bounds are sufficient.

\end{proof}

\begin{prop}\label{prop:gl3-outer}
Suppose that $G_0 = {\rm PSL}^{\epsilon}_{3}(q)$ and $x\in G$ is a field, graph, or graph-field automorphism of prime order $r$. Then $\fpr(x, G/H) < |x^{G_0}|^{-3/10}$.
\end{prop}

\begin{proof}
The cases where $q\leqs 19$ can be checked using {\sc Magma}, so we may assume that $q\geqs 23$. Lemma~\ref{lem:class-bounds} implies that $|x^{G_0}| > \frac{1}{6}\left(\frac{q}{q+1}\right)q^4$, and so unless $H$ is of type ${\rm GL}_3^{\epsilon'}(q_0)$, ${\rm O}_3(q)$ or ${\rm GL}_1^{\epsilon}(q)\wr S_3$, then the bound $|x^G\cap H| < |H|$ is sufficient. 

If $H$ is of type ${\rm GL}_3^{\epsilon'}(q_0)$ or ${\rm O}_3(q)$ and $x$ is not an involutory graph-field automorphism then the bounds in the proofs of~\cite[5.1]{fprII} and~\cite[8.2]{fprII} are sufficient, so now assume that $\epsilon = +$ and $x$ is an involutory graph-field automorphism. In both cases, Lemma~\ref{lem:outer-cosets} implies that $x^G\cap H \subseteq H_Ix$, where $H_I = H\cap G_I$ and using~\cite[1.3]{LLS} we find
\[
|x^G\cap H| \leqs i_2(H_I) + 1 < 2(1+1q^{-1/2})q^{5/2}
\]
which is sufficient.

Finally assume that $H$ is of type ${\rm GL}_1^{\epsilon}(q)\wr S_3$. If $r > 2$, then the bound $|x^{G}\cap H| < |H|$ is sufficient, so now assume that $r = 2$. In this case, again appealing to Lemma~\ref{lem:outer-cosets} gives 
\[
|x^G\cap H| \subseteq H_Ix \leqs i_2(H_I) + 1 \leqs 16,
\]
and recall that $|x^{G_0}| > \frac{1}{6}\left(\frac{q}{q+1}\right)q^4$, so we are done.
\end{proof}

\begin{prop}\label{prop:sp4}
Suppose that $G_0 = {\rm PSp}_{4}(q)$ for $q\geqs 4$, and $x\in G_I$ is an element of prime order $r$. Then $\fpr(x, G/H) < |x^{G_0}|^{-9/20}$.
\end{prop}

\begin{proof}
Using~\cite[Tables 8.12-8.14]{BHR}, as well as Table~\ref{t:ns}, we can determine the possibilities of $H$. For most of those cases we use entirely similar methods as the ones we use to prove Propositions~\ref{prop:gl5}, \ref{prop:gl4}, and~\ref{prop:gl3}, so we only discuss the cases where we use different methods here, to avoid unnecessary repetition. In particular, we discuss the cases where $H$ is of type ${}^2B_2(q)$, $O_2^{\epsilon}(q)\wr S_2$, or $O^{-}_2(q^2).2$ for $q$ even, and ${\rm L}_2(q)$ for $q$ odd. 

If $H$ is of type ${}^2B_2(q)$, then we obtain the required bounds by inspecting the proof of~\cite[4.4]{B07}, so it remains to handle the other four cases.

Suppose that $H$ is of type $O_2^{\epsilon}(q)\wr S_2$. Note that if $|x^{G_0}| > \frac{1}{2}\left(\frac{q}{q+1}\right)q^6$, then the trivial bound $|x^G\cap H| < |H|$ is sufficient. Now suppose that $|x^{G_0}| \leqs \frac{1}{2}\left(\frac{q}{q+1}\right)q^6$. Lemma~\ref{lem:class-bounds} implies that this occurs when $x$ is an $a_2$ or $b_1$ involution. In these cases we note that $x\in G_0$, and so $x^G\cap H = x^G \cap H_0$, where $H_0 = H\cap G_0$. 
We now note that $H_0 = C_{q-\epsilon}^2{:}{\rm D}_8$, and in particular, $i_2(H_0) = (q+\epsilon)(q+3)$. We therefore compute
\[
|x^{G_I}\cap H| \leqs i_2(H_0) \leqs (q+1)(q+3) \leqs (q^4-1)^{11/20} = |x^{G_0}|^{11/20}
\]

when $q > 32$. If $q\leqs 32$, then we can construct $H_0$ using {\sc Magma} as the normaliser of an appropriate Sylow subgroup. We only give the details for the case where $H$ is of type $\O^{-}_2(q)\wr S_2$ when $q = 32$, as the other cases are similar. In this case we can obtain $H_0$ as the normaliser of a Sylow $3$-subgroup. Note that $3$ divides $q+1$, and in particular, the largest power of $3$ dividing $|G_0|$ is the same as the largest power of $3$ dividing $(q+1)^2$, so a Sylow $3$-subgroup $P$ of $G_0$ is contained in some maximal torus $T \cong C_{q+1}\times C_{q+1}$.  In particular, $P\cong C_3\times C_3$ is a characteristic subgroup of $T$, so $N_G(T) \leqs N_{G}(P)$. One then checks using {\sc Magma} that $|N_G(P)| = |H_0|$, so $N_G(P) = H_0$, and we can verify that the bound $|x^G\cap H| < i_2(H_0)$ is sufficient. 

The case where $H$ is of type  $O^{-}_2(q^2).2$ is very similar. If $q\leqs 32$, then we again use {\sc Magma}, so now assume that $q\geqs 64$. If $|x^{G_0}| > \frac{1}{2}\left(\frac{q}{q+1}\right)q^6$, then the trivial bound $|x^G\cap H| < |H|$ is sufficient, so it remains to prove the claim for $a_2$ and $b_1$-type involutions. Again we note that $x\in G_0$, and so $x^G\cap H = x^G\cap H_0$. We now note that $H_0 = C_{q^2+1}:4$, and in particular $i_2(H_0) = q^2+1$. We then have
\[
|x^{G_I}\cap H| \leqs i_2(H_0) = q^2+1 \leqs (q^4-1)^{11/20} = |x^{G_0}|^{11/20}
\]
as required.

Finally, suppose that $H\in \mathscr{S}$ is of type ${\rm L}_2(q)$. Let $W = \la v_1, v_2 \ra$ be the natural module for $L = {\rm SL}_2(q)$. Then the embedding of $L$ in ${\rm Sp}_4(q)$ is the image of the representation $\rho : L \to {\rm Sp}_4(q)$ corresponding to the $L$-module $S^3(W)$. Let $\hat{x}$ be the lift of $x$ in ${\rm GSp}_4(q)$ and first suppose that $x$ is unipotent. In this case $x$ has associated Jordan decomposition $[J_2]$ one $W$ and we note that 
\[
\rho\left(
\begin{pmatrix}
1 & 1\\
0 & 1
\end{pmatrix}\right) = 
\begin{pmatrix}
1&1&1&1\\
0&1&2&3\\
0&0&1&3\\
0&0&0&1
\end{pmatrix} = M
\]
with respect to the basis $\{v_1^3, v_1^2v_2, v_1v_2^2, v_2^3\}$.
Now, since $M - I_4$ has rank $3$, we deduce that $M$ is not a transvection, and therefore $|x^{G_0}| > \frac{1}{2}\left(\frac{q}{q+1}\right)q^6$. Moreover, we know that $|x^G\cap H| \leqs q^2-1$, and so the claim holds.

Next suppose that $x\in H$ is semisimple and $r$ is odd, and let $i$ be minimal such that $r\vert q^i - 1$. First suppose that $i = 1$. Here $\hat{x}$ is $L$-conjugate to $[\lambda, \lambda^{-1}]$ for some $\lambda \in \F^{\times}$ and we have
\[
\rho([\lambda, \lambda^{-1}]) = [\lambda^3, \lambda, \lambda^{-1}, \lambda^{-3}]
\] 
If $x$ has order $3$, then $C_G(x)$ is of type ${\rm GL}_1(q)\times {\rm Sp}_2(q)$. In this case we have 
\[
|x^{G_0}| \geqs \frac{|{\rm Sp}_4(q)|}{|{\rm GL}_1(q)||{\rm Sp}_2(q)|} = q^3(q+1)(q^2+1)
\]
and
\[
|x^{G_I}\cap H| \leqs \frac{|{\rm GL}_2(q)|}{|{\rm GL}_1(q)|^2} = q(q+1)
\]
and so the bound holds.

If $r \geqs 5$, then $x$ is a regular semisimple element, and we have 
\[
|x^{G_0}| \geqs \frac{|{\rm Sp}_4(q)|}{|{\rm GL}_1(q)|^2}  = q^4(q+1)^2(q^2+1),
\]
which is sufficient. The case where $i = 2$ is similar, so we leave the details to the reader.

We finally consider involutions. Let $K$ denote the algebraic closure of $\F$ and let $\bar{G} = {\rm PSp}_4(K)$ and $\bar{H} = {\rm PSL}_2(K)$. Now let $\bar{W}$ be the natural module for ${\rm SL}_2(K)$ and $\bar{\rho} : {\rm SL}_2(K) \to {\rm Sp}_4(K)$ denote the representation afforded by the $K{\rm SL}_2(q)$-module $S^3(\bar{W})$. Now note that $\bar{H}$ has a unique class of involutions $y^{\bar{H}}$, and $y$ lifts in ${\rm SL}_2(K)$ to a diagonal matrix $\hat{y} = [\lambda, \lambda^{-1}]$ for some $\lambda \in K^{\times}$ of order $4$. Then
\[
\bar{\rho}([\lambda, \lambda^{-1}]) = [\lambda^{-1}, \lambda, \lambda^{-1}, \lambda]
\]
with respect to the basis $\{w_1^3, w_1^2w_2, w_1w_2^2, w_2^3\}$.  Now there are two classes of involutions in $\bar{G}$, one of type $A_1A_1$ and one of type $A_1T_1$, and the former lift to involutions in ${\rm Sp}_4(K)$, whilst the latter lift to elements of order $4$. We therefore deduce that $y$ embeds in $\bar{G}$ as an $A_1T_1$ involution. 

Since both $\bar{H}$ and $\bar{G}$ are stable under the standard Frobenius morphism $\sigma$, we deduce that if $x\in \bar{H}_{\sigma}$ is an involution, then it embeds in $\bar{G}_{\sigma} = {\rm Inndiag}(G_0)$ as an involution with centraliser type ${\rm GL}_2^{\epsilon}(q)$. We therefore have
\[
|x^{\bar{G}_{\sigma}} \cap H| \leqs i_2({\rm PSL}_2(q)) = q^2-1
\]
and
\[
|x^{G_0}| = |x^{\bar{G}_{\sigma}}| \geqs \frac{|{\rm Sp}_4(q)|}{2|{\rm GU}_2(q)|} = \frac{1}{2}q^3(q-1)(q^2+1)
\]
so the claim holds and the proof is complete.
\end{proof}

\begin{prop}\label{prop:sp4-outer}
Suppose that $G_0 = {\rm PSp}_{4}(q)$ for $q\geqs 4$, and $x\in G$ is a field or graph-field automorphism of prime order $r$. Then $\fpr(x, G/H) < |x^{G_0}|^{-3/10}$.
\end{prop}

\begin{proof}
The methods we use here are similar to the ones in the proofs of Propositions~\ref{prop:gl4-outer} and~\ref{prop:gl3-outer}, so we leave the details for the reader.
\end{proof}

\begin{proof}[Proof of Theorem~\ref{thm:low-rank-fprs}]
This follows by combining Propositions~\ref{prop:sp8}, \ref{prop:gl8}, \ref{prop:o10}, \ref{prop:gl6}, \ref{prop:sp6}, \ref{prop:o7}, \ref{prop:gl5}, \ref{prop:gl4}, \ref{prop:gl4-outer}, \ref{prop:gl3}, \ref{prop:gl3-outer}, \ref{prop:sp4}, and \ref{prop:sp4-outer}.
\end{proof}

\section{Eta function bounds}\label{s:eta}
We will use Proposition~\ref{prop:regular} and variations as heavily as possible in the proof of Theorem~\ref{thm:main-classical}. In most cases, the bounds in Proposition~\ref{prop:TG} suffice, but in some cases we will need stronger bounds on $T_G$, where $T_G$ is as defined in Definition~\ref{def:eta}, or even stronger bounds on $\eta_G(t)$ for some $t\geqs T_G$. We establish those bounds in this section.

\begin{prop}\label{prop:eta-sp6}
If $G_0 = {\rm PSp}_6(q)$ and $q$ is even, then $\eta_G(1/5) < 0.9$.
\end{prop}

\begin{proof}
If $q\leqs 16$, then we check the claim using {\sc Magma}. In particular, we first call \texttt{ClassicalClasses("Sp", 6, q)}, and we compute $\alpha = \sum_{i = 1}^k |x_i^{G_0}|^{-1/5}$ directly, where $x_1, \ldots, x_k$ denote representatives of the prime order classes of $G_0$. It now suffices to manually add the contribution $\beta$ from outer automorphisms to $\alpha$ to obtain an upper bound for $\eta_G(1/5)$. We note that
\[
\beta = \sum_{r\in \pi(f)} (r-1) \cdot \left(\frac{|{\rm PSp}_6(q)|}{|{\rm PSp}_6(q^{1/r})|}\right)^{-1/5}
\]
where recall that $f = \log_p(q)$ and we confirm that $\alpha + \beta < 0.9$ for $q\leqs 16$.

If $q\geqs 32$, then we use the bound obtained by Burness in the proof of~\cite[2.3]{B07}. In particular, we have 
\begin{align}\label{eq:eta-sp6}
\eta_G(t) &< 2\left(\frac{1}{2}(q^6-1)\right)^{-t} + \left(2 + \frac{q}{2}\log(q^2-1)\right)\cdot \left( \frac{1}{4}(q+1)^{-1}q^{11}\right)^{-t}\\
&+ ((q^2+1)(q^6-1))^{-t} + \left(6 + \frac{q}{2}\log(q^2-1)\right) \cdot \left( \frac{1}{4}(q+1)^{-1}q^{13}\right)^{-t} \nonumber \\
&+ q^2\log(q^2+1)\left(\frac{1}{8}q^{14}\right)^{-t} + \frac{1}{3}q^3\log(q^6-1)\cdot \left(\frac{1}{4}q^{18}\right)^{-t} \nonumber \\
&+2(\log_{2}q-1)\cdot \log(\log_{2}q + 2)\cdot \left(\frac{1}{4}q^{21/2}\right)^{-t} \nonumber
\end{align}
and one can check that this bound is sufficient if $q\geqs 32$. 

Finally, for $16 < q < 32$ we use a more precise version of~\eqref{eq:eta-sp6}. In particular, $\log (q^2-1), \log(q^2+1)$ and $\log (q^6-1)$ appear in~\eqref{eq:eta-sp6} as approximations of the number of prime divisors of $q^2-1, q^2+1$, and $q^6-1$ respectively, and replacing those approximations with the exact number of prime divisors of those numbers for $16 < q < 32$ gives the desired bound.
\end{proof}

\begin{prop}\label{prop:eta-o7}
If $G_0 = {\rm \O}_7(q)$ and $q$ is odd, then $\eta_G(1/5) < 0.9$.
\end{prop}

\begin{proof}
If $q\leqs 13$, then again we compute directly using {\sc Magma} exactly as in the proof of Proposition~\ref{prop:eta-sp6}, so we may assume that $q\geqs 17$ for the remainder of the proof.

Our aim here is to derive a similar bound to the one used in Proposition~\ref{prop:eta-sp6}. To do this, we must bound the number of $G_0$-conjugacy classes of prime order elements of each type in $G$, as well as the size of each $G_0$-conjugacy class. We record the relevant bounds for each type of prime order element in Table~\ref{t:etao7}, but for the sake of brevity, we only discuss a few cases in detail. The cases not discussed are very similar. 

Table~\ref{t:etao7} is organised as follows: The first row records the type of each prime order element in $G_0$. For unipotent elements we record the Jordan form on $V$, for odd order semisimple elements we record the centraliser type, and if $x$ is a semisimple involution or a field automorphism, then we use the standard notation from~\cite[Sections 3.5.2 and 3.5.5]{BG}. The second row records an upper bound $f(x, q)$ on the number of $G_0$-classes of each type and the third row records a lower bound $g(x, q)$ on the $G_0$-conjugacy class size of each element.

We first consider unipotent classes. Note that the possible Jordan forms on $V$ are as follows: 
\[
[J_7], [J_5, J_1^2], [J_3, J_2^2], [J_3, J_1^4], [J_2^2, J_1^3]
\]
We can estimate how many classes correspond to each Jordan decomposition by counting the number of signed partitions corresponding to that form and consulting~\cite[3.5.14]{BG} to determine which $G_I$-classes split in $G_0$. The $G_0$-class bounds in Table~\ref{t:etao7} are obtained via~\cite[3.21]{fprII}.

We now look at semisimple classes of odd order. Let $i$ be minimal such that $r\vert q^i - 1$. Note that since $r$ divides $|G|$ we must have $i\in \{1, 2, 3, 4, 6\}$. The $G_0$-class bounds are determined via~\cite[Tables 3.6 and 3.8]{fprII}, so it remains to determine $f(x, q)$. We must treat each possibility in turn. We indicatively discuss two cases in detail here.

Suppose first that $i = 6$. In this case $C_G(x)$ is of type ${\rm O}_1(q) \times {\rm GU}_1(q^3)$, and $x$ is of the form $[\Lambda, 1]$, for some $\sigma$-orbit $\Lambda$. We now know that if $i = 6$, then $r\vert q^3+1$, and moreover that there are exactly $\frac{r-1}{6} \leqs \frac{q^3}{6}$ $\sigma$-orbits, so the number of $G_0$-classes of elements of this type in $G$ is at most $\frac{q^3}{6}\pi(q^3+1) \leqs \frac{q^3}{6}\log (q^3+1)$. 

Now suppose that $i = 1$ and $C_G(x)$ is of type ${\rm GL}_1(q)^3 \times {\rm O}_1(q)$. In this case $r \vert q-1$, and since there are exactly $\frac{q - 1}{2}$ distinct pairs $(\lambda, \lambda^{-1})\in (\F^{\times})^2$, we conclude that there are at most $\binom{(q-1)/2}{3}\pi(q-1)\leqs \frac{q^3}{48}\log(q-1)$ $G_0$-classes of this type. 

We next consider semisimple involutions. By inspecting~\cite[Table B.8]{BG} we find that there are at most six $G_0$-classes of semisimple involutions, namely $t_i$ for $i\in \{1, 2, 3\}$ and $t'_i$ for $i\in \{1, 2, 3\}$.

Finally, we can infer from~\cite[3.5.20]{BG} that for a fixed prime divisor $r$ of $\log_{p}q$, there are at most $2(r-1) \leqs 2\log_{p}q$ $G_0$-classes of field automorphisms of order $r$, and $\log_{p}q$ has at most $\log (\log_{p}q + 2)$ distinct prime divisors, so there are at most $2\log_{p}q\cdot \log (\log_{p}q + 2)$ $G_0$-classes of field automorphisms.

Now if $x_1, \ldots, x_t$ is a set of representatives for the $G_0$-classes of prime order elements in $G$, then we have
\[
\eta_G(1/5) \leqs \sum_{i = 1}^t f(x_i, q)\cdot g(x_i, q)^{-1/5}
\]
where $f(x_i, q)$ and $g(x_i, q)$ are recorded in Table~\ref{t:etao7} for every $i$ and one can check that this bound is sufficient for $q\geqs 17$.  
{\small
\begin{table}
\[
\begin{array}{llll} \hline
\text{Type of } x& f(x, q) & g(x, q) & \text{Notes} \\ \hline
(J_7)&  4&\frac{1}{4}\left(\frac{q}{q+1}\right)q^{18}& \\
(J_5, J_1^2) &4 &\frac{1}{8}\left(\frac{q}{q+1}\right)^2q^{16}& \\
(J_3, J_2^2)& 4&\frac{1}{4}\left(\frac{q}{q+1}\right)q^{12}& \\
(J_3^2, J_1)& 4&\frac{1}{8}\left(\frac{q}{q+1}\right)^2q^{10}&\\
 (J_3, J_1^4)&4&\frac{1}{8}\left(\frac{q}{q+1}\right)^2q^{10}&\\
 (J_2^2, J_1^3)&2&\frac{1}{4}\left(\frac{q}{q+1}\right)q^{8}&\\
{\rm GU}_1(q^3)\times {\rm O}_1(q)&\frac{1}{6}q^3\log(q^3+1)&\frac{1}{2}\left(\frac{q}{q+1}\right)q^{18}& i = 6\\
{\rm GU}_1(q^2)\times {\rm O}_3(q) &\frac{1}{4}q^2\log (q^2+1)&\frac{1}{2}\left(\frac{q}{q+1}\right)q^{14}&i=4\\
{\rm GL}_1(q^3)\times {\rm O}_1(q) &\frac{1}{6}(q^2+q)\log(q^2+q+1)&\frac{1}{2}q^{18}&i = 3\\
{\rm GU}_3(q)\times {\rm O}_1(q) &\frac{1}{2}q\log(q+1)&\frac{1}{2}\left(\frac{q}{q+1}\right)q^{12}&i=2\\
{\rm GU}_2(q)\times {\rm GU}_1(q) \times {\rm O}_1(q)&\frac{1}{4}q^2\log (q+1)&\frac{1}{2}\left(\frac{q}{q+1}\right)^2q^{16}&i=2\\
{\rm GU}_2(q)\times {\rm O}_3(q)&\frac{1}{2}q\log (q+1)&\frac{1}{2}\left(\frac{q}{q+1}\right)q^{14}&i=2\\
{\rm GU}_1(q)^2\times {\rm O}_3(q)&\frac{1}{8}q^2\log (q+1)&\frac{1}{2}\left(\frac{q}{q+1}\right)^2q^{16}&i=2\\
{\rm GU}_1(q)^3\times {\rm O}_1(q)&\frac{1}{48}q^3\log (q+1)&\frac{1}{2}\left(\frac{q}{q+1}\right)^3q^{18}&i=2\\
{\rm GL}_3(q)\times {\rm O}_1(q) &\frac{1}{2}q\log(q-1)&\frac{1}{2}q^{12}&i=1\\
{\rm GL}_2(q)\times {\rm GL}_1(q) \times {\rm O}_1(q)&\frac{1}{4}q^2\log(q-1)&\frac{1}{2}q^{16}&i=1\\
{\rm GL}_2(q)\times {\rm O}_3(q)&\frac{1}{2}q\log(q-1)&\frac{1}{2}q^{14}&i=1\\
{\rm GL}_1(q)^2\times {\rm O}_3(q)&\frac{1}{8}q^2\log(q-1)&\frac{1}{2}q^{16}&i=1\\
{\rm GL}_1(q)^3\times {\rm O}_1(q)&\frac{1}{48}q^3\log(q-1)&\frac{1}{2}q^{18}&i=1\\
t_1&1&\frac{1}{4}\left(\frac{q}{q+1}\right)q^{10}&\\
t_2&1&\frac{1}{4}\left(\frac{q}{q+1}\right)q^{12}&\\
t_3&1&\frac{1}{4}\left(\frac{q}{q+1}\right)q^{6}&\\
t'_1&1&\frac{1}{4}\left(\frac{q}{q+1}\right)q^{10}&\\
t'_2&1&\frac{1}{4}\left(\frac{q}{q+1}\right)q^{12}&\\
t'_3&1&\frac{1}{4}\left(\frac{q}{q+1}\right)q^6&\\
\phi^{fj/r}&2\log_{p}q\cdot \log (\log_{p}q + 2)&\frac{1}{2}q^{21/2}&q = p^f, \,\, 1\leqs j < r\\
\hline
\end{array}
\]
\caption{Bounds on conjugacy classes and number of prime order elements in $\O_7(q)$.}
\label{t:etao7}
\end{table}
}
\end{proof}



\begin{prop}\label{prop:eta-gl6}
If $G_0 = {\rm PSL}^{\epsilon}_6(q)$ and $q\geqs 3$, then $\eta_G(1/3) < 0.979$.
\end{prop}

\begin{proof}
We adopt essentially the same strategy as the one used in the proof of~\cite[2.3]{B07}. For an element $x\in {\rm PGL}^{\epsilon}_6(q)$, we write $\nu(x)$ for the codimension of the largest eigenspace of the lift $\hat{x}$ on $\overline{V} = V\otimes \overline{\mathbb{F}_q}$, where $\overline{\mathbb{F}_q}$ denotes the algebraic closure of $\F$. Note that~\cite[3.1.3, 3.2.2, 3.3.3]{BG} together with~\cite[Sections 3.4.2, 3.5.2]{BG} imply that $\nu(x)$ is uniquely determined. For $1\leqs s\leqs 5$ we denote the number of distinct $G_0$-classes of unipotent (respectively semisimple) prime order elements $x\in G$ with $\nu(x) =s$ by $k_{s, u}$ (respectively $k_{s, s}$).
Moreover, we will write $c_u(s)$ (respectively $c_s(s)$) for the size of the smallest unipotent (respectively semisimple) $G_0$-class $y^G$ with $\nu(y) = s$, $\pi$ for the distinct prime divisors of $\log_{p}q$ and for $r\in \pi$ we will write $\mathcal{F}_r$ for the set of $G$-classes of field automorphisms of order $r$. Finally, we will denote the set of $G$-classes of involutory graph-field automorphisms by $\mathcal{F}_{gf}$, and we will write $C_{1}, C_{2}, C_3$ for the $G_I$-classes of a graph automorphism of type ${\rm PGSp}_6(q), {\rm PGO}^{+}_6(q)$, and ${\rm PGO}^{-}_6(q)$ respectively. Noting that~\cite[3.2.14]{BG} tells us that $C_i$ splits into at most $3$ $G_0$-classes for all $i\in \{1, 2, 3\}$, we have 
\[
\eta_G(t) \leqs  \sum_{C\in {\mathcal{F}_{gf}}}|C|^{-t} + \sum_{r\in \pi}\sum_{C\in \mathcal{F}_r}|C|^{-t} + \sum_{i = 1}^{3}3\left(\frac{|C_i|}{3}\right)^{-t} + \sum_{s = 1}^{5}(k_{s, u} + k_{s, s})\cdot c(s)^{-t}
\]
and therefore it suffices to find appropriate bounds for $|\mathcal{F}_r|$, $|\mathcal{F}_{gf}|$, $|C|$, $|C_i|$, $|\pi|$, $k_{s, u}$, $k_{s, s}$, $c_u(s)$, and $c_s(s)$ for $s\in \{1, \ldots, 5\}$.

Firstly, we know that every $N\in \Nat$ has at most $\log(N+2)$ distinct prime divisors, and so $|\pi| \leqs \log(\log_{p}q + 2)$ if $p\neq q$ and $|\pi| = 0$ if $p = q$.
Moreover,~\cite[4.9.1]{CFSGIII} gives $|\mathcal{F}_r| \leqs (6, q-\epsilon)(r-1) \leqs 6(\log_{p}q + 1)$. Similarly, it follows straight from~\cite[3.2.15]{BG} that the $G_I$-class of graph-field automorphisms splits into $(6, q^{1/2}-1)$ $G_0$-classes, and so $|\mathcal{F}_{gf}| \leqs 6$. 

We now consider unipotent elements. Note that $(8)$ in~\cite{fprII} tells us that there are at most $2^{s}$ unipotent ${\rm Inndiag}(G_0)$-classes with $\nu(x) = s$ for $s\in \{1, \ldots, 5\}$ and each class can split to at most $6$ classes in $G_0$. Therefore, $k_{s, u} < 6\cdot2^s$. Moreover, we can infer from~\cite[3.22]{fprII} and~\cite[3.2.7]{BG} that
\[
c_u(s) > 
\begin{cases}
 \frac{1}{2}\left(\frac{q}{q+1}\right)q^{2s(6-s)} \, \text{if} \, s\in \{1, 2\}\\
\frac{1}{4}\left(\frac{q}{q+1}\right)q^{6s} \, \text{if} \, s = 3\\
\frac{1}{12}\left(\frac{q}{q+1}\right)q^{6s} \, \text{if} \, s\in \{4, 5\}\\
\end{cases}
\]

Similarly, for semisimple elements~\cite[3.40]{fprII} tells us that there are at most $q^{s + 1}$ prime order elements with $\nu(x) = s$ of order $r$, so now it suffices to bound the possibilities for $r$ to bound $k_{s, s}$. Note that $r$ must divide $\prod_{i = 1}^{3}(q^i - \epsilon)(q^{2i} - 1) < q^{22}$, and since every integer $N > 6$ has at most $\log N$ divisors, we conclude that $k_{s, s} < 22\log q \cdot q^{s+1}$.

Finally,~\cite[3.36-3.37]{fprII} gives
\[
c_s(s) > 
\begin{cases}
 \frac{1}{2}\left(\frac{q}{q+1}\right)q^{2s(6-s)} \, \text{if} \, s\in \{1, 2\}\\
\frac{1}{4}\left(\frac{q}{q+1}\right)q^{6s} \, \text{if} \, s = 3\\
\frac{1}{6}\left(\frac{q}{q+1}\right)^{\frac{s}{6-s}}q^{6s} \, \text{if} \, s\in \{4, 5\}\\
\end{cases}
\]
and one can check that those bounds are sufficient for $q\geqs 11$. If $3\leqs q\leqs 9$, then we can compute $k_{s, s}$, and $k_{s, u}$ precisely and obtain the result this way. For example, let us consider the case $(\epsilon, q, s) = (+, 8, 4)$. If $\nu(x) = 4$ for a unipotent element $x\in G$, then the Jordan decomposition of $x$ on $V$ must comprise of two blocks. But one cannot have an involution in $G$ whose Jordan decomposition only has two blocks, so $k_{4, u} = 0$.

Now suppose that $x$ is semisimple and let $i$ be minimal such that $r\vert q^i-1$, so $i\in \{1, \ldots, 6\}$. Now it is easy to see that if $i\geqs 4$, then the lift $\hat{x}$ of $x$ in ${\rm GL}(V)$ is of the form $[\Lambda, I_{6-i}]$ for some $\sigma$-orbit $\Lambda$, so $\nu(x) = 5$, a contradiction, and so $i\geqs 3$. If $i = 3$, then since $\nu(x) = 4$ we must have $\hat{x} = [\Lambda^2]$ for some $\sigma$-orbit $\lambda$. Moreover, by inspecting the prime divisors of $|G_I|$, we find that $r = 73$, and so, as explained in~\cite[Section 3.2.1]{BG}, there are exactly $\frac{73-1}{3} = 24$ distinct $\sigma$-orbits for $i = 3$, so we get precisely $24$ $G_0$-classes with $\nu(x) = 4$. 

Now assume that $i = 2$. Note that for all the eigenvalues of $\hat{x}$ to occur with multiplicity at least $2$, $\hat{x}$ must be of the form $[\Lambda^2, I_2]$. Moreover, the assumption $i = 2$ implies that $r = 3$, and there is a unique $\sigma$-orbit when $r = 3$. 

Finally assume that $i = 1$, so $r = 7$. Since $7$ does not divide $6$, we find that $\hat{x} = [\lambda I_2, \lambda^{\alpha}I_2, \lambda^{\beta}I_2]$ for some non-trivial $7$-th root of unity $\lambda$ and $1 < \alpha, \beta \leqs 6$, so there are $\binom{7}{3} = 35$ possibilities for $x$. Therefore $k_{4, s} = 24 + 1 + 35 = 90 \ll 22\log 8 \cdot 8^5$. The argument in all remaining cases is similar, so we omit the details.
\end{proof}

\section{Proof of Theorem~\ref{thm:main-classical}}\label{s:proof}

In this section we complete the proof of Theorem~\ref{thm:main-classical}. Recall that we are left to deal with the following collections:
\begin{align*}
\mathcal{A}_1 &= \{{\rm PSp}_8(q), {\rm PSL}_8^{\epsilon}(q), {\rm P\O}^{\epsilon}_{10}(q)\}\\
\mathcal{A}_2 &= \{{\rm PSL}^{\epsilon}_6(q)\}\\
\mathcal{A}_3 &= \{{\rm PSp}_6(q) \, : \, q \text{ even}\} \cup \{{\rm \O}_7(q) \, : \, q \text{ odd}\}\\
\mathcal{A}_4 &= \{{\rm PSL}_3^{\epsilon}(q), {\rm PSL}_4^{\epsilon}(q), {\rm PSL}_5^{\epsilon}(q), {\rm PSp}_4(q)\}\\
\mathcal{A}_5 &= \{{\rm PSL}_2(q)\}
\end{align*}
We will treat each collection in turn.

\begin{prop}\label{prop:a1}
The conclusion to Theorem~\ref{thm:main-classical} holds for the groups with $G_0\in \mathcal{A}_1$.
\end{prop}

\begin{proof}
If $G_0 = {\rm PSp}_8(q)$ and $q = 2$, then the claim can be verified using {\sc Magma}, using the function \texttt{RegTuplesPlus()}, so we may assume for the remainder of the proof that $q\geqs 3$ if $G_0 = {\rm PSp}_8(q)$. 

Propositions~\ref{prop:sp8}, ~\ref{prop:gl8}, and ~\ref{prop:o10}, together with Theorem~\ref{thm:timsbound}, tell us that ${\rm fpr}(x, G/H) < |x^{G_0}|^{-1/3}$ for every core-free subgroup $H$ of $G$ and all $x\in G$ of prime order. Moreover, Proposition~\ref{prop:TG} gives $T_G < 1/3$, and so if $\tau = (H_1, \ldots, H_4)$ is any non-standard tuple of $G$, then Proposition~\ref{prop:regular} implies that $\tau$ is regular. Finally, Proposition~\ref{prop:eta-asymptotic} implies that $\mathbb{P}(G, 4) \to 1$ as $|G| \to \infty$.
\end{proof}

\begin{prop}\label{prop:a2}
The conclusion to Theorem~\ref{thm:main-classical} holds for the groups with $G_0 \in \mathcal{A}_2$.
\end{prop}

\begin{proof}
If $q = 2$, then as described in Section~\ref{sss:comp-classical}, we can check that all non-standard $4$-tuples (or $5$-tuples if $G = U_6(2).2$) are regular using the {\sc Magma} function \texttt{RegTuplesPlus()}. If $G = U_6(2).2$, then we can also verify that the only non-regular $4$ tuple is $(H, H, H, H)$ for $H = {\rm U}_4(3).2^2$ as the same function outputs all the non-regular $4$-tuples. We may therefore assume that $q\geqs 3$ for the remainder. Now note that Proposition~\ref{prop:gl6}, together with Theorem~\ref{thm:timsbound} implies that 
\[
\widehat{Q}(G, \tau) \leqs \eta_G(1/3) + \alpha,
\]
where
\[
\alpha =  \frac{q-1}{2}\log(q-1)\frac{(q^6(q+1)(q^2+1)(q^3+1))^4}{(q^9(q-\epsilon)(q^3-\epsilon)(q^5-\epsilon))^{3}}
\]
for any non-standard $4$-tuple $\tau$ of $G$. In view of Proposition~\ref{prop:eta-gl6}, it suffices to show that $\alpha < 0.021$
for all $q\geqs 3$ to prove that $\tau$ is regular, and one can easily check that this is true. Moreover, we see that $\alpha \to 0$ as $q\to \infty$, and this combined with Proposition~\ref{prop:eta-asymptotic} implies that $\mathbb{P}(G, 4) \to 1$ as $|G| \to \infty$.
\end{proof}

\begin{prop}\label{prop:a3}
The conclusion to Theorem~\ref{thm:main-classical} holds for the groups with $G_0 \in \mathcal{A}_3$.
\end{prop}

\begin{proof}
We can check that all non-standard tuples of $G$ are regular if $q\leqs5$ using {\sc Magma}, so we may assume that $q\geqs 7$ for the remainder of the proof. Let $\tau = (H_1, \ldots, H_4)$ be a non-standard tuple of $G$. First assume that at least one of the $H_i$'s is not of type $G_2(q)$, say $H_1$ without loss of generality. 
Note that~\cite[2.13]{fprIV} implies that if $H$ is of type $G_2(q)$ and $x\in H$ has prime order, then unless $x$ is unipotent with Jordan form $[J_2^2, J_1^3]$ or $[J_3^2, J_1]$, we have ${\rm fpr}(x, G/H) < |x^{G_0}|^{-0.268}$. 

Now let $x_1, \ldots, x_t$ be a set of representatives of the prime order conjugacy classes of $G$, and without loss of generality assume that $x_{t-1}$ and $x_t$ are the representatives of the unipotent $G$-classes with Jordan form $[J_2^2, J_1^3]$ and $[J_3^2, J_1]$ respectively. Propositions~\ref{prop:sp6} and~\ref{prop:o7} tell us that ${\rm fpr}(x_j, G/H_1) < |x_j^G|^{-0.4}$ for all $j\in [t]$, whilst ${\rm fpr}(x_j, G/H_k) < |x_j|^{-0.268}$ for $k\neq 1$ and $j\in [t-2]$. Finally, for $j\in \{t-1, t-2\}$ ${\rm fpr}(x_j, G/H_k)$ can be recovered from~\cite[Tables 2.4 and 2.5]{fprIV}. Combining all of this we get
\[
\widehat{Q}(G, \tau)\leqs \sum_{i = 1}^{t} |x_i^{G_0}|^{-3\cdot  0.268 - 0.4 + 1} + \alpha + \beta
\]
where
\[
\alpha = |x_{t-1}^{G}|^{0.6}\prod_{i = 2}^4 \frac{|{x_{t-1}}^G\cap H_i|}{|{x_{t-1}}^G|} = \frac{(q^6-1)^3}{((q^2-1)(q^6-1))^{12/5}}
\]
and 
\[
\beta =  |x_{t}^{G}|^{0.6}\prod_{i = 2}^4 \frac{|{x_{t}}^G\cap H_i|}{|{x_{t}}^G|} =  \frac{2^{12/5}(q^2(q^2-1)(q^6-1))^3}{(q^6(q-1)(q^4-1)(q^6-1))^{12/5}}.
\]
Therefore,
\[
\widehat{Q}(G, \tau) \leqs \eta_G(1/5) + \alpha + \beta
\]
and in view of Propositions~\ref{prop:eta-sp6} and~\ref{prop:eta-o7}, it suffices to show that 
\[
\alpha + \beta < 0.1,
\]
to show that $\tau$ is regular. We can easily check that this is true. 

Moreover, noting that $\alpha, \beta \to 0$ as $q\to \infty$ and combining this with Proposition~\ref{prop:eta-asymptotic}, we find that $\mathbb{P}(G, \tau) \to 1$ as $|G|\to \infty$.

\vs

Now suppose that $H_i$ is of type $G_2(q)$ for all $i\in \{1, \ldots, 4\}$. If $G_0 = {\rm Sp}_6(q)$ with $p = 2$, then there is a unique $G$-class of subgroups of type $G_2(q)$, so the regularity of $\tau$ follows straight from~\cite[3.1]{B07}, so now let us turn to the case $G_0 = {\rm \O}_7(q)$ and $q$ odd. If $G = {\rm SO}_7(q)$ or ${\rm \Gamma O}_7(q)$, then again, there is a unique $G$-class of subgroups of type $G_2(q)$, and so the claim follows directly by~\cite[3.2]{B07}.

However, if $G = G_0$ or $G.\la \phi \ra$ for some field automorphism $\phi$, then there are two $G$-classes of subgroups of type $G_2(q)$, so these cases require special attention. We thank Tim Burness for pointing out the omission of this analysis in a previous draft. 

We can write $G = G_0.\la \phi \ra$, where $\phi$ is a (possibly trivial) field automorphism of $G_0$. Let $\rho$ be an irreducible Spin representation of $J = {\rm P\O}_8^{+}(q)$. If $L = G_2(q).\la \phi \ra$, then $\rho(L)$ is the stabiliser inside $\rho(G_0).\la \phi \ra$ of a 1-dimensional non-singular subspace of the natural module $\overline{V}$ of $J$. We claim that $\rho(G_0).\la\phi\ra$ has two orbits in this action, namely
\[
O_1 := \{\la v \ra \, : \, v\in \overline{V}, Q(v) \text{ a square in } \F \} \,\,\, \text{and} \,\,\, O_2 := \{\la v \ra \, : \, v\in \overline{V}, Q(v) \text{ a non-square in } \F \},
\]
where $Q$ is the standard quadratic form on $\overline{V}$. Now let $v\in \overline{V}$, let $g\in \rho(G_0)$ be an isometry, and let $w\in \overline{V}$ be such that $\la v\ra^g = \la w\ra$. Then there exists some $c\in \F$ such that $v^g = cw$. Since $g$ is an isometry, we have $Q(v^g) = Q(v)$, and so 
\[
Q(cw) = c^2Q(w) = Q(v).
\]
Noting that the product of two squares is a square and the product of a square and a non-square is a non-square, we deduce that $g$ preserves $O_1$ and $O_2$.

Now let $h \in \rho(G_0).\la \phi\ra$ be a field automorphism. Since $Q$ is the standard form on $\overline{V}$, it is defined over the base field $\mathbb{F}_q$, so we have $Q^{h}(v) = Q(v)$, and in particular $Q(v^h) = Q(v)^{h}$. Since field automorphisms preserve squares and non-squares, it follows that $Q(v^{h})$ is a square if and only if $Q(v)$ is a square. 

Therefore, since both isometries and field automorphisms preserve $O_1$ and $O_2$, and any element in $\rho(G_0).\la \phi\ra$ is a product of those, we cannot send an element from $O_1$ to $O_2$. Finally, it follows by the Orbit-Stabiliser Theorem that $O_1$ and $O_2$ are orbits and not unions of orbits.

\vs 

Let $H$ be stabiliser of some $\la v \ra \in O_1$ and $K$ be the stabiliser of some $\la w \ra \in O_2$. We will now show that any tuple $\sigma$ only containing $H$ and $K$ is regular. Since the distinction between the two orbits is not made explicit in~\cite{B07}, we will also deal with the conjugate tuples for completeness. 

The first step of the proof will be to show that for every possibility for $\sigma$, we can find vectors $v_1, \ldots, v_4\in \overline{V}$, such that the following hold:
\begin{itemize}
\item [\rm (i)] The $\la v_i\ra $ are not all simultaneously fixed by a nontrivial field automorphism;
\vs 
\item [\rm (ii)] $W := \la v_1, \ldots, v_4 \ra$ is a minus-type $4$-space of $\overline{V}$;
\vs 
\item [\rm (iii)] The stabiliser of $\la v_i\ra$ is conjugate to the $i$-th entry of $\sigma$ for every $i\in \{1, \ldots, 4\}$. 
\end{itemize}
Let $\{e_1, f_1, \ldots, e_4, f_4\}$ be a standard basis for $\overline{V}$ and let $\alpha\in \F^{\times}$ be a primitive element of $\F$, so $\alpha$ is a non-square and it is not contained in any proper subfield of $\F$. We treat each possibility for $\sigma$ in turn:

\vs

\noindent \textbf{Case 1.} $\sigma = (H, H, H, K)$: Let $v_i = e_i + f_i$ for $i\in \{1, \ldots, 3\}$ and $v_4 = e_4 + \alpha f_4$. We note that $Q(v_i) = 1$ for $i\in \{1, \ldots, 3\}$ and $Q(v_4) = \alpha$, and moreover the $\la v_i\ra$ are not all simultaneously fixed by a nontrivial field automorphism, since $v_4$ is not fixed by any nontrivial field automorphism. Let $W = \la v_1, \ldots, v_4\ra $. Since the $v_i$ are mutually orthogonal, it follows that the restriction $Q_W$ of $Q$ on $W$ has Gram matrix $A = [Q(v_1), \ldots, Q(v_4)]$, and in particular ${\rm det}(A) = \alpha$, a non-square. Then~\cite[2.5.10]{KL} implies that $W$ is of minus type, as required.

\vs

\noindent \textbf{Case 2:} $\sigma = (H, K, K, K)$: Take $v_1 = e_1 + f_1$ and $v_i = e_i + \alpha f_i$ for $i\in \{2, \ldots, 4\}$, and again note that the $\la v_i\ra $ are not simultaneously fixed by a nontrivial field automorphism. Take $W = \la v_1, \ldots, v_4\ra $. As above, we note that $Q_W$ has Gram matrix with determinant $\alpha^3 = \alpha^2\cdot \alpha$, a non-square, and~\cite[2.5.10]{KL} implies that $W$ is of minus type.

\vs

\noindent \textbf{Case 3:} $\sigma = (H, H, K, K)$: Take $v_1 = e_1 + f_1$, $v_2 = e_2 + \alpha f_2$, and set $U = \la v_1, v_2 \ra$. Note that~\cite[2.5.10]{KL} implies that $U$ is of plus type if and only if $q\equiv 3 \imod{4}$.

\vs

First suppose that $q\equiv 3 \imod{4}$. By appealing to~\cite[2.2.12]{BG}, we note that we can find some $2$-dimensional subspace $U'$ of minus type inside the $6$-dimensional non-degenerate subspace $U^{\perp}$. Now by appealing again to~\cite[2.2.12]{BG}, we note that we can find $v_3, v_4\in U'$, such that $Q(v_3)$ is a square and $Q(v_4)$ is a non-square. Since $v_3$ and $v_4$ are clearly linearly independent by construction, we have $U' = \la v_3, v_4 \ra$. Finally, we note that~\cite[2.2.11(ii)]{BG} implies that $W = \la v_1, \ldots, v_4 \ra = U\oplus U'$ is a minus-type subspace, as required.

\vs

Now suppose that $q\equiv 1 \imod{4}$. In this case, $U$ is of minus ype. It then follows by~\cite[2.2.12]{BG} that we can find some $2$-dimensional subspace $U' \subseteq U^{\perp}$ of plus type. As before, by appealing to~\cite[2.2.12]{BG} once again, we note that we can find $v_3, v_4\in U'$, such that $Q(v_3)$ is a square and $Q(v_4)$ is a non-square. Arguing exactly as above, we find that $U' = \la v_3, v_4 \ra$ and $W = \la v_1, \ldots, v_4 \ra = U\oplus U'$ is of minus type.

We finally note that in both cases, the choice of $v_2$ ensures the $\la v_i\ra$ are not all simultaneously fixed by some nontrivial field automorphism.

\vs

\noindent \textbf{Case 4:} $\sigma = (K, K, K, K)$: Set $v_i = e_i + \alpha f_i$ for $i\in \{1, 2\}$ and let $U = \la v_1, v_2\ra$. Here,~\cite[2.5.10]{KL} implies that $U$ is of plus type if and only if $q\equiv 1 \imod{4}$.

\vs

First suppose that $q\equiv 3 \imod{4}$, in which case $U$ is of minus type. As before,~\cite[2.2.12]{BG} ensures that we can choose a plus-type $2$-dimensional subspace $U' \subseteq U^{\perp}$. Let $e, f\in U'$ denote a standard basis for $U'$ and set $v_3 = e + \alpha f$, $v_4 = e + \alpha^{-1}f$, so that both $Q(v_3)$ and $Q(v_4)$ are non-squares. Then $U' = \la v_3, v_4 \ra$ and $W = \la v_1, \ldots, v_4 \ra = U\oplus U'$ is a minus-type $4$-space by~\cite[2.2.11(ii)]{BG}.

\vs

We may now suppose that $q\equiv 1 \imod{4}$, so that $U$ is of plus type. It again follows from~\cite[2.2.12]{BG} that we can choose some $2$-dimensional subspace $U' \subseteq U^{\perp}$ of minus type this time. Now note that all $q + 1$ $1$-spaces in $U'$ are non-singular, and in particular exactly $(q+1)/2$ of them are in $O_1$. Therefore, we can find two $1$-spaces $\la v_3\ra, \la v_4 \ra \subseteq U'$ so that $\la v_3 \ra, \la v_4\ra \in O_2$. Then as before, we have $U' = \la v_3, v_4\ra$ and $W = \la v_1, \ldots, v_4\ra = U \oplus U'$ is a minus-type $4$-space of $\overline{V}$.

Finally, the choice of $v_1$ ensures that not all $\la v_i \ra$ are simultaneously fixed by a nontrivial field automorphism.

\vs

\noindent \textbf{Case 5:} $\sigma = (H, H, H, H)$: The argument is entirely similar to the one in Case 4, so we leave the details to the reader.

\vs

We have now shown that for every choice of $\sigma$, we can find a minus-type $4$-space of $\overline{V}$ satisfying conditions (i)-(iii) above. Let $W = \la v_1, \ldots, v_4\ra$ be such a subspace. We now proceed as in the proof of~\cite[3.2]{B07}. Let us identify $G$ with $\rho(G_0).\la \phi \ra$ and suppose for a contradiction that $x\in \displaystyle \bigcap_{i = 1}^4 G_{\la v_i\ra}$ has prime order $r$. Then $x\in G_0$ and $x$ fixes a decomposition $W\oplus W^{\perp}$ of $\overline{V}$, whilst acting trivially on $W$, and in particular $r$ divides $|\Omega_4^{-}(q)| = q^2(q^4-1)$ by Lagrange's theorem.

First suppose that $r = 2$. Then~\cite[3.55]{fprII} implies that $C_J(x)$ must be of type ${\rm GL}_4^{\epsilon}(q), {\rm O}_4^{+}(q)^2$, or ${\rm O}_4^{+}(q^2)$, which contradicts the assumption that $x$ fixes the decomposition $W\oplus W^{\perp}$ and acts trivially on $W$.

Next suppose that $r$ is odd and let $i$ be minimal such that $r$ divides $q^i - 1$, so $i\in \{1, 2, 4\}$. We note that we cannot have $i = 4$, since this would imply that $r$ divides $q^2+1$. But we also have $x\in {G_0}_{\la v_1\ra}= G_2(q)$, and $q^2+1$ does not divide $|G_2(q)|$. On the other hand, if $i\in \{1, 2\}$, then~\cite[3.29]{fprII} implies that the codimension of the largest eigenspace of $\rho(x)$ on $\overline{V}$ is $2$. However, by inspecting the proof of~\cite[2.7]{fprIV} we deduce that there are no such elements in $\rho(G_0)$, and so $\displaystyle \bigcap_{i = 1}^4 G_{\la v_i\ra} = 1$, and $\sigma$ is regular.

Finally, suppose that $\mu = (H_1, \ldots, H_6)$ is a $6$-tuple of $G$ all of whose entries are of type $G_2(q)$. Then Theorem~\ref{thm:timsbound} implies that ${\rm fpr}(x, G/H_i) < |x^G|^{1/2-1/7-0.108} < |x^G|^{-0.24}$. Moreover, by inspecting the paragraph before~\cite[Theorem 1.1]{LSh1}, we note that if $T_G$ is as defined in Definition~\ref{def:ns}, then $T_G < 1/3$. We now note that $1/3 < 6\cdot 0.24 - 1$, and so~\ref{prop:eta-asymptotic} implies that $\mathbb{P}(G, \mu) \to 1$ as $|G|\to \infty$.
\end{proof}

\begin{prop}\label{prop:a4}
The conclusion to Theorem~\ref{thm:main-classical} holds for the groups with $G_0\in \mathcal{A}_4$.
\end{prop}

\begin{proof}
Let $\tau = (H_1, \ldots, H_4)$ be a non-standard tuple. If $G_0 = {\rm PSL}^{\epsilon}_5(q)$ or ${\rm PSL}^{\epsilon}_4(q)$, then in view of Propositions~\ref{prop:gl5}, \ref{prop:gl4}, and \ref{prop:gl4-outer}, we may write
\begin{equation}\label{eq:a4}
\widehat{Q}(G, \tau) \leqs \alpha_1 \cdot \beta_1^{-4/5} +  \alpha_2\cdot \beta_2^{-1/5} + \alpha_3\cdot \beta_3^{-1/5}
\end{equation}
where $\alpha_1$ denotes the number of prime order $G_0$-classes of unipotent and semisimple elements, $\alpha_2$ denotes the number of $G_0$-classes of field and graph-field automorphisms and $\alpha_3$ denotes the number of $G_0$-classes of graph automorphisms respectively, and
\begin{align*}
\beta_1 &= \min \{|x^{G_0}| \,: \, x\in {\rm Inndiag}(G_0), \, |x| \in \pi(|G|)\}\\
\beta_2 &= \min \{|x^{G_0}| \,: \, x\in \mathcal{F} \}\\
\beta_3 &= \min \{|x^{G_0}| \,: \, x\in \mathcal{G} \}
\end{align*}
where $\mathcal{F}$ denotes the set of prime order field and graph-field automorphisms in $G$ and $\mathcal{G}$ the set of involutory graph automorphisms in $G$. We will derive bounds for  each $\alpha_i$ and $\beta_i$.

\vs

First suppose that $G_0 = {\rm PSL}^{\epsilon}_5(q)$. If $q\leqs 4$, then the regularity of $\tau$ can be verified using {\sc Magma}, so we may assume that $q\geqs 5$. In this case, we can infer from~\cite[Table 3]{FG} that the total number of conjugacy classes in ${\rm PGL}_5^{\epsilon}(q)$ is bounded above by $q^4 + 8q^3$. We now note that the only $G$-class of unipotent elements that potentially splits in $G_0$ is the one corresponding to the Jordan form $[J_5]$, and~\cite[3.2.7, 3.3.10]{BG} imply that it splits into at most $5$ $G_0$-classes. Therefore, we may take $\alpha_1 \leqs q^4 + 8q^3+4$. Moreover,~\cite[4.9.1]{CFSGIII} implies that there are at most 
\[
\left(5, \frac{q-\epsilon}{q^{1/r}-\epsilon}\right)(r-1) \leqs 5\log_{p}q
\]
$G_0$-classes of field automorphisms of prime order $r$, and since $|\pi(\log_{p}(q))| \leqs \log(\log_p q + 2)$, there are at most 
\[
5\log_{p}q \cdot \log(\log_p q + 2)
\]
distinct $G_0$-classes of prime order field automorphisms in $G$. Similarly, there are at most $5$ distinct $G_0$-classes of involutory graph-field automorphisms by~\cite[3.2.15]{BG}. We also get at most $5$ distinct $G_0$-classes of involutory graph automorphisms by~\cite[3.2.14, 3.3.17]{BG}, so we have
\[
\alpha_2 \leqs 5 + 5\log_{p} q \cdot \log(\log_{p}q + 2) \,\,\, \text{and} \,\,\, \alpha_3 \leqs 5.
\]
Finally, it follows straight from Lemma~\ref{lem:class-bounds} that $\beta_1 > \frac{1}{2}\left(\frac{q}{q+1}\right)q^{8}$, $\beta_2 > \frac{1}{10}\left(\frac{q}{q+1}\right)q^{12}$, and $\beta_3 > \frac{1}{10}\left(\frac{q}{q+1}\right)q^{14}$, and by substituting the above bounds in~\eqref{eq:a4} we find that $\widehat{Q}(G, \tau) < 1$ for all $q\geqs 5$, as required.

\vs

Next assume that $G_0 = {\rm PSL}^{\epsilon}_4(q)$. If $q\leqs 9$, then the regularity of $\tau$ can again be verified using {\sc Magma}, so we may assume that $q \geqs 11$ for the remainder. Again, from~\cite[Table 3]{FG} we see that there are at most $\alpha_1\leqs q^3 + 8q^2$ distinct $G_I$-classes of unipotent and semisimple elements, and the only classes that possibly split in $G_0$ are the class of unipotent elements corresponding to the Jordan decompositions $[J_4]$ and $[J_2^2]$, so $\alpha_1 \leqs q^3 + 8q^2+4$. Moreover,~\cite[4.9.1]{CFSGIII} again tells us that the contribution to $\alpha_2$ from field automorphisms is at most  $4\log_{p} q \cdot \log(\log_{p}q + 2)$, and~\cite[3.2.15]{BG} tells us that there are at most $4$ distinct $G_0$-classes of involutory graph-field automorphisms. Finally, we can infer from~\cite[3.2.14, 3.3.17]{BG} that each $G_I$-class of graph automorphisms splits into at most $2$ $G_0$-classes, and so
\[
\alpha_2 \leqs 4\log_{p} q \cdot \log(\log_{p}q + 2) + 4 \,\,\, \text{and} \,\,\, \alpha_3 \leqs 6
\]
Now Lemma~\ref{lem:class-bounds} gives $\beta_1 > \frac{1}{2}\left(\frac{q}{q+1}\right)q^6$, $\beta_2 > \frac{1}{8}\left(\frac{q}{q+1}\right)q^{15/2}$, and $\beta_3 > \frac{1}{4}\left(\frac{q}{q+1}\right)q^4$. Substituting into~\eqref{eq:a4} we get $\widehat{Q}(G, \tau) < 1$ for $q\geqs 32$. If $q\leqs 32$, then we replace the upper bound on $\alpha_1$, by the tighter bound $(4, q-1)\log_{p} q \cdot |\pi(\log_{p}q)|$ and we compute the exact contribution from involutory graph automorphisms, and we again verify that $\widehat{Q}(G, \tau) < 1$.

\vs

Our strategy is similar for the cases where $G_0 = {\rm PSL}_3(q)$ or ${\rm PSp}_4(q)$. Let $\alpha_1, \beta_1, \alpha_3, \beta_3$ be defined as above, and this time let $\alpha_2$ be the number of $G_0$-classes of field automorphisms of odd order and 
\[
\beta_2 = \min \{|x^{G_0}| \,: \, x\in \mathcal{F}, |x| \neq 2 \}.
\]
Finally, let $\alpha_4$ denote the number of $G_0$-classes of involutory field and graph-field automorphisms and 
\[
\beta_4 =  \min \{|x^{G_0}| \,: \, x\in \mathcal{F}, |x| = 2 \}.
\]
Then Propositions~\ref{prop:gl3} and~\ref{prop:gl3-outer} give us
\begin{equation}\label{eq:gl3-sp4}
\widehat{Q}(G, \tau) \leqs \alpha_1\beta_1^{-4/5} + \alpha_2\beta_2^{-1/5} + \alpha_3\beta_3^{-1/5} + \alpha_4\beta_4^{-1/5}
\end{equation}
so it suffices to find sufficient bounds for $\alpha_i, \beta_i$ for $i\in [4]$. We can again infer from~\cite[Table 3]{FG} that there are at most $q^2 + 8q$ distinct prime order classes in $G_I$, none of which split in $G_0$, so $\alpha_1 \leqs q^2 + 8q$. Moreover, the same arguments we used when $G_0 = {\rm PSL}_5(q)$ gives $\alpha_2 \leqs 3\log_{p} q \cdot \log(\log_{p}q + 2)$, $\alpha_3\leqs 3$ and $\alpha_4\leqs 6$. Next, as usual we obtain 
\begin{align*}
\beta_1 &> \frac{1}{2}\left(\frac{q}{q+1}\right)q^{4}\\
\beta_2 &> \frac{1}{6}\left(\frac{q}{q+1}\right)q^{16/3}\\
\beta_3 &> \frac{1}{6}\left(\frac{q}{q+1}\right)q^{4}\\
\beta_4 &> \frac{1}{6}\left(\frac{q}{q+1}\right)q^{4}\\
\end{align*}
via Lemma~\ref{lem:class-bounds} and one can check that these bounds are sufficient for $q\geqs 67$. If $q\leqs 25$, then we verify the claim using {\sc Magma}, so it is left to verify that $\widehat{Q}(G, \tau) < 1$ for $27\leqs q \leqs 64$. In this case, the tighter bounds $\alpha_2 \leqs (3, q-\epsilon) \log_{p} q |\pi(\log_{p}q)|$, $\alpha_3 \leqs (3, q-\epsilon)$ and $\alpha_4 \leqs 2(3, q-\epsilon)$ are sufficient. 

\vs

Finally, suppose that $G_0 = {\rm PSp}_4(q)$. We verify the claim using {\sc Magma} if $q\leqs 11$ so we may assume that $q\geqs 13$ for the remainder. Let $\alpha_i, \beta_i$ be defined as in the case $G_0 = {\rm PSL}_3^{\epsilon}(q)$ for $i\in [4]$. Then Propositions~\ref{prop:sp4} and~\ref{prop:sp4-outer} imply that~\eqref{eq:gl3-sp4} still applies, so we will again derive sufficient bounds for $\alpha_i, \beta_i$ for $i\in [4]$.

By inspecting~\cite[Table 3]{FG} we deduce that there are at most $q^2 + 12q$ distinct $G_I$-classes of unipotent and semisimple elements and each unipotent class can split into at most $4$ distinct $G_0$-classes, so we get $\alpha_1 \leqs q^2+12q+9$. Moreover, Lemma~\ref{lem:class-bounds} implies that $\beta_1 >  \frac{1}{4}\left(\frac{q}{q+1}\right)q^4$. Next, using~\cite[3.4.15]{BG} we find
\[
\alpha_2 \leqs (2, q-1)(r-1)\pi(\log_{p}q) \leqs 2\log_{p} q \cdot \log(\log_{p}q + 2). 
\]
and once again via Lemma~\ref{lem:class-bounds}, we get $\beta_2 > \frac{1}{4}\left(\frac{q}{q+1}\right)q^{20/3}$. Similarly, we deduce that there are at most $2$ distinct $G_0$-classes of involutory field automorphisms by inspecting~\cite[3.4.15]{BG}, and~\cite[3.4.16]{BG} tells us that there is an additional $G_0$-class of involutory graph-field automorphisms when $p = 2$, so $\alpha_4 \leqs 3$. Finally, we get $\beta_4 > \frac{1}{4}q^5$ from Lemma~\ref{lem:class-bounds}, and trivially, $\alpha_3 = \beta_3 = 0$, since $G$ contains no graph automorphisms. One can check that those bounds are sufficient for $q\geqs 13$.

Finally, we observe that $\widehat{Q}(G, \tau) \to 0$ as $q\to \infty$ in all cases, so $\mathbb{P}(G, 4)\to 1$ as $|G|\to \infty$. 
\end{proof}

\begin{prop}\label{prop:a5}
The conclusion to Theorem~\ref{thm:main-classical} holds for the groups with $G_0 \in \mathcal{A}_5$.
\end{prop}

\begin{proof}
First we note that the claim can easily be verified using { \sc Magma} for $q\leqs 31$, so we may assume for the remainder of the proof that $q \geqs 32$. If $\tau = (H_1, \ldots, H_4)$ is a non-standard tuple of $G$, then Proposition~\ref{prop:liebeck-saxl} yields
\[
\hat{Q}(G, \tau) \leqs |G|\cdot \left(\frac{4}{3q}\right)^4 + \alpha \cdot \left(\frac{2 + q^{1/2}(q^{1/2} + 1)}{q^{1/2}(q+1)}\right)^4
\] 
where $\alpha$ denotes the number of involutory field automorphisms. We now note that 
\[
\alpha \leqs \frac{|{\rm PGL}_2(q)|}{|{\rm PGL}_2(q^{1/2})|} 
\]
and we can now check that the above bound gives $\widehat{Q}(G, \tau) < 1$ for all $q\geqs 32$. Moreover, we can easily see that $\widehat{Q}(G, \tau)\to 0$ as $q\to \infty$, and so $\mathbb{P}(G, 4)\to 1$ as $|G|\to \infty$, as required.
\end{proof}

\begin{proof}[Proof of Theorem~\ref{thm:main-classical}]
This follows by combining Propositions~\ref{prop:a6}, \ref{prop:a1}, \ref{prop:a2}, \ref{prop:a3}, \ref{prop:a4}, \ref{prop:a5}.
\end{proof}

\section{Exceptional group preliminaries}\label{s:exceptional}

For the remainder of the paper, we may assume that $G$ is an almost simple exceptional group of Lie type with socle $G_0$ over $\F$, where $q = p^f$ for some prime $p$ and $\tau$ is a $6$-tuple of core-free subgroups of $G$. As in the previous section, our main aim is to prove Theorem~\ref{thm:main}. In particular, we will show that the stronger statement below holds, which in combination with Theorem~\ref{thm:main-classical} for classical groups, completes the proof of Theorem~\ref{thm:main}.

\begin{thm}\label{thm:exceptional}
Let $G$ be a finite almost simple exceptional group of Lie type and let $\tau$ be a $6$-tuple of core-free subgroups of $G$. Then $\tau$ is regular, and moreover, $Q(G, \tau) \to 0$ as $q\to \infty$.
\end{thm}

We start by recalling some notation we have been using throughout and fixing some new notation. As before, we will write $\bar{G}$ for the ambient algebraic group (of the same type as $G$) over the algebraic closure of $\F$, so $G_0 = (G_{\sigma})'$ for some appropriate Steinberg endomorphism $\sigma$. Also, for a closed subgroup $\bar{H}$ of $\bar{G}$, we will write $\bar{H}^0$ to denote the connected component of $\bar{H}$ containing the identity. We will also be using the standard notation $E_6^{\epsilon}(q)$ for the finite exceptional groups of type $E_6$, where $E_6^{+}(q) = E_6(q)$ and $E_6^{-}(q) = {}^2E_6(q)$.

We will denote the number of elements of order $r$ in some group $H$ by $i_r(H)$. We will often denote the set of prime numbers at most $r$ by $\pi_r$ and the set of odd primes at most $r$ by $\pi'_r$. 

Finally, for the remainder of the paper we will write $\mathcal{M}_G$ for the set of core-free maximal subgroups of $G$ and $\mathcal{P}_G$ for the set of parabolic maximal subgroups of $G$.
\vs

\subsection{Subgroup structure} 
We will repeatedly refer to standard sources in the literature for information regarding maximal subgroups of finite exceptional groups. We recall some important results in this section. We write $\mathcal{M}_G$ as a union $\mathcal{C}\cup \mathcal{S}$ of two collections, where $\mathcal{C}$ is the collection of maximal subgroups $H$ of $G$ of the following types:

\begin{itemize}
\item [\rm (a)] $H = N_G(\bar{H}_{\sigma})$ for some maximal $\sigma$-stable positive dimensional closed subgroup $\bar{H}$ of $\bar{G}$;
\vs
\item [\rm (b)] $H$ is of the same type as $G$. That is, $H$ is either a subfield subgroup or a twisted version of $G$;
\vs
\item [\rm (c)] $H$ is the normaliser of some exotic $r$-local subgroup for some prime $r\neq p$. Those have been classified (up to conjugacy) in~\cite{Coh};
\vs
\item [\rm (d)] $G_0 = E_8(q)$, $p > 5$, and ${\rm soc}(H) = A_5\times A_6$.
\end{itemize}

A description of the subgroups in (a) and (b) can be found in~\cite{LSeitz3}. Finally, the group in case (d) was first described by Borovik in~\cite{Bor} and is known as the \emph{Borovik subgroup}. Note that the work of Borovik in~\cite{Bor} in combination with the work of Liebeck and Seitz in~\cite{LSeitz4} implies that all maximal subgroups in $\mathcal{S}$ are almost simple.

\vs

At the time of writing, all maximal subgroups of $G$ have been determined up to conjugacy unless $G_0 = E_7(q)$ or $E_8(q)$. We now give a few details for each case and the original sources where the information can be found:

\begin{itemize}
\item $\mathbf{G_0 = {}^2B_2(q):}$ In this case, Suzuki determines the maximal subgroups of $G$ up to conjugacy in~\cite{Sz}, and Bray, Holt and Roney-Dougal record them in~\cite[Table 8.16]{BHR}.

\vs

\item $\mathbf{G_0 = {}^2G_2(q)' \textbf{ or } {}^3D_4(q):}$ Here, the subgroups in $\mathcal{M}_G$ up to conjugacy are determined by Kleidman in~\cite{Kl} and~\cite{Kl2} respectively, and they are listed in~\cite[Tables 8.43, 8.51]{BHR}. Note that ${}^2G_2(3)' \cong {\rm PSL}_2(8)$, and so we may assume that $q > 3$ if $G_0 = {}^2G_2(q)'$.

\vs 

\item $\mathbf{G_0 = {}^2F_4(q)':}$ In the case $q\geqs 8$ the maximal subgroups of $G$ are determined by Malle in~\cite{Ma}, apart from the omission of three classes of maximal subgroups isomorphic to ${\rm PGL}_{2}(13)$ when $q = 8$. This is straightened out in~\cite[Remark 4.11]{Cr}. 

If $q = 2$, then the subgroups in $\mathcal{M}_G$ are determined up to conjugacy by Wilson in~\cite{Wi}, with the exception of the omission of a unique class of maximal subgroups isomorphic to ${\rm U}_3(2).2$.

\vs

\item $\mathbf{G_0 = G_2(q)':}$ Here, the members of $\mathcal{M}_G$ are determined up to conjugacy in~\cite{Co} for even $q$ and in~\cite{Kl} for odd $q$. For $q = 2^f$ and $f > 1$ one can find them listed~\cite[Table 8.30]{BHR} and for odd $q$ they are listed in~\cite[Tables 8.41, 8.42]{BHR}.

\vs

\item $\mathbf{G_0 = F_4(q):} \textbf{ or } \mathbf{{E}_6^{\epsilon}(q):}$ The definitive source for those groups is~\cite{Cr}, where Craven determines their maximal subgroups up to conjugacy. The members of $\mathcal{C}$ are listed in~\cite[Tables 7-10]{Cr} and the members of $\mathcal{S}$ are listed in~\cite[Tables 1-3]{Cr}.

\vs

\item $\mathbf{G_0 = E_7(q):}$ The members of $\mathcal{M}_G$ are discussed in~\cite{Cr2} in this case. In particular, the members of $\mathcal{C}$ are listed in~\cite[Table 4.1]{Cr2} up to conjugacy, and as stated in~\cite[1.1]{Cr2}, the collection $\mathcal{S}$ consists of the groups listed in~\cite[Table 1.1]{Cr2} up to conjugacy possibly together with some additional subgroups with socle ${\rm L}_2(q')$ for $q\in \{7, 8, 9, 13\}$. 

\vs

\item $\mathbf{G_0 = E_8(q):}$ The subgroups in $\mathcal{C}$ have been completely determined up to conjugacy. In particular, the parabolic and maximal rank subgroups of the form $N_G(\bar{H}_{\sigma})$ are listed in~\cite[Tables 5.1 and 5.2]{LSS}, the subfield subgroups and exotic local subgroups are listed in~\cite[Table 1]{LSeitz}, and finally for the non-maximal rank subgroups of the form $N_G(\bar{H}_{\sigma})$, one can find the possibilities for ${\rm soc}(\bar{H}_{\sigma})$ listed in~\cite[Table 3]{Coh}. If $p > 5$ then $\mathcal{C}$ also contains the Borovik subgroup. 

It is currently an open problem to determine the members of $\mathcal{S}$ even up to isomorphism. However, significant progress has been made towards a classification. The current state of the art regarding the groups in class $\mathcal{S}$ will not be relevant to us, but we refer the reader to~\cite[p.19]{BKor} for more information.


\end{itemize}
\subsection{Conjugacy classes}

We will repeatedly be using results from the literature about conjugacy classes of elements in finite exceptional groups. In this section we point the reader to the sources where the information is recorded.

\begin{itemize}
\item $\mathbf{G_0 = {}^2B_2(q):}$ The conjugacy classes of $G$ are determined in~\cite{Sz}.

\vs

\item $\mathbf{G_0 = {}^2G_2(q):}$ The conjugacy classes of $G$ are determined in~\cite{Wa}.

\vs

\item $\mathbf{G_0 = {}^3D_4(q):}$ The conjugacy classes of $G$ are determined in~\cite{DM} and~\cite{Sp}.

\vs

\item $\mathbf{G_0 = {}^2F_4(q)':}$ The $G_0$-classes are determined by Shinoda in~\cite{Sh}.

\vs

\item $\mathbf{G_0 = G_2(q)':}$ If $p\geqs 5$ then the conjugacy classes of $G$ are determined in~\cite{Ch} and if $p < 5$, they are determined by Enomoto in~\cite{E70}.

\vs

\item $\mathbf{G_0 = F_4(q):}$ Information on the semisimple classes of $G_0$ is recorded in~\cite[Table 23]{AHL}. For unipotent classes, the Foulkes functions for $G_0$ are given to us by Lübeck, and we discuss further in Section~\ref{s:parabolics} how we use those to compute fixed point ratios.

\vs

\item $\mathbf{G_ 0 = E_6^{\epsilon}(q), E_7(q)} \textbf{ or } \mathbf{E_8(q)}$: Our standard reference for semisimple classes is~\cite{FJ1} when $G_ 0 = E_6^{\epsilon}(q)$ and $E_7(q)$ and~\cite{FJ2} when $G_0 = E_8(q)$, where Fleischmann and Janiszczak describe the conjugacy classes of the simply connected groups of type $E_6, E_7$, and $E_8$ respectively. For unipotent elements we again use the Foulkes functions given to us by Lübeck in~\cite{Lu}.

\end{itemize}

Recall that $\tau$ is a $6$-tuple of core-free subgroups of $G$. As with classical groups, our central aim will be to show that $\widehat{Q}(G, \tau) < 1$, and in particular bound $\widehat{Q}(G, \tau)$ by a function of $q$ that tends to zero as $q$ tends to infinity. We adopt a rather different approach on how we derive an upper bound for $\widehat{Q}(G, \tau)$, depending on the socle type of $G$ and so we will divide the various possibilities for $G_0$ into the following three collections, according to the proof strategy we use, and we will handle each one in turn:

\begin{align}
\mathcal{E}_1 &= \{{}^2G_2(q)', {}^2B_2(q), {}^2F_4(q)', {}^3D_4(q), G_2(q)'\} \nonumber\\
\mathcal{E}_2 &= \{E_6^{\epsilon}(q), E_7(q), F_4(q)\}\\
\mathcal{E}_3 &= \{E_8(q)\}\nonumber
\end{align}

In all the above cases, our proof will rely on obtaining sufficient fixed point ratio estimates on all maximal subgroups of $G$, and we will be working with the bound
\[
\widehat{Q}(G, \tau) \leqs \sum_{x\in \pi} f(x, q),
\]
where $\pi$ denotes the set of prime order elements in $G$ and
\[
f(x, q) = \max \{{\rm fpr}(x, G/H) \, : \, H\in \mathcal{M}_G\}
\]
for each choice of prime order element $x\in G$. Note that $f(x, q)$ is independent of the choice of conjugacy class representative. That is, $f(x, q) = f(x^g, q)$ for all $g\in G$, so if $x_1, \ldots, x_t$ is a list of representatives of the prime order conjugacy classes of $G$, then we will be working with the expression:
\[
\sum_{i = 1}^t |x_i^G|\cdot f(x_i, q).
\]

Let $H$ be a core-free maximal subgroup of $G$. The definitive source on bounds on ${\rm fpr}(x, G/H)$ is~\cite{LLS}. However, the bounds in~\cite{LLS} will not always be sufficient for our purposes.

When $H$ is non-parabolic, our standard approach will be to obtain sufficient bounds by inspecting the proofs of~\cite{BLS} and~\cite{B18}. In certain rare cases, we will also need to improve existing bounds in the literature. For example, if $G_ 0 = F_4(2)$ and $|H| > 2^{22}$, then we will use the Character Table library in \textsf{GAP} to compute the precise value of ${\rm fpr}(x, G/H)$ for all $x\in G$. More details can be found in the proof of Proposition~\ref{prop:f4-fpr}. 

If $G$ has socle in $\mathcal{E}_2$, then obtaining fixed point ratio bounds that hold for all non-parabolic subgroups of $G$ is slightly more involved than in other cases, due to the larger number of possibilities for $H$. We carefully explain how to do this in Section~\ref{s:non-parabolics}. These bounds may be of independent interest, so we record them in Tables~\ref{t:e6:fprs}, \ref{t:e6:fprs-q2}, \ref{t:e7:fprs}, \ref{t:f4:fprs}, and~\ref{t:f4:fprs-q2} for the reader's convenience. Surprisingly, if $G_0 = E_8(q)$, then obtaining sufficient bounds is not as troublesome and can be done by careful inspection of~\cite{BLS}. 

Our approach when $H$ is parabolic is rather different and we give detailed information on how we obtain bounds on ${\rm fpr}(x, G/H)$ in Section~\ref{s:parabolics}. If $G_0 \in \mathcal{E}_1$, then the bounds in~\cite[Theorems 1--2]{LLS} are always sufficient for unipotent and semisimple elements, and better bounds can be obtained by inspecting the proof of~\cite[6.1]{LLS} for field and graph-field automorphisms, if necessary. This is however not true if $G_0 \in \mathcal{E}_2\cup \mathcal{E}_3$, where we will heavily use~\cite[3.2]{LLS} and~\cite{Lu} to derive better estimates for ${\rm fpr}(x, G/H)$ that are sensitive to ${\rm dim}\, x^{\bar{G}}$ for a given prime order element $x\in G$. For a unipotent element $x\in G$, it is always possible to compute precise values for ${\rm fpr}(x, G/H)$ using Lübeck's values of the Foulkes functions for $G_0$~\cite{Lu}. 

Now suppose that $x\in G$ is semisimple. If $G_0\in \mathcal{E}_3$, then an upper bound on the formula given by~\cite[3.2]{LLS} will be sufficient for our purposes but if $G_0\in \mathcal{E}_2$, then we will need to compute the precise value of ${\rm fpr}(x, G/H)$ to obtain our result. For field and graph-field automorphisms we will always appeal to a statement extracted from the proof of~\cite[6.1]{LLS} (see Proposition~\ref{prop:LLS-6.1}). Finally, for graph automorphisms when $G_0 = E_6^{\epsilon}(q)$, we derive appropriate bounds in Proposition~\ref{prop:e6-fpr-graph} when $q$ is odd using the approach described in ~\cite[6.4]{LLS}, which we record in Table~\ref{t:e6:graphs}, and for even $q$, we consult~\cite[Table 7]{BLS} and for the reader's convenience we record this information in Table~\ref{t:e6:graphs-p2}. We give more details in Section~\ref{s:parabolics}.

\vs

\noindent \textbf{Organisation.} The remainder of the paper is organised as follows: In Section~\ref{s:non-parabolics} we give fixed point ratio bounds for non-parabolic actions for groups with socle in $\mathcal{E}_2$. In Section~\ref{s:parabolics}, we explain how we derive fixed point ratio estimates on parabolic subgroups for groups with socle in $\mathcal{E}_2$ and $\mathcal{E}_3$. Finally, in Sections~\ref{s:e1}, \ref{s:e2}, and~\ref{s:e3}, we prove Theorem~\ref{thm:exceptional} for groups with socle in $\mathcal{E}_1, \mathcal{E}_2$ and $\mathcal{E}_3$, respectively.

\section{Fixed point ratios for non-parabolic actions}\label{s:non-parabolics}

In this section we establish fixed point ratio bounds for non-parabolic actions of groups with socle in $\mathcal{E}_2$. In particular, for every prime order element in $G$ we obtain a value $g(x, q)$ such that 
\[
{\rm fpr}(x, G/H) \leqs g(x, q)
\]
for every $H\in \mathcal{M}_G\setminus \mathcal{P}_G$ and we record those values in Tables~\ref{t:e6:fprs}, \ref{t:e6:fprs-q2}, \ref{t:e7:fprs}, \ref{t:f4:fprs}, and~\ref{t:f4:fprs-q2}. As mentioned above, $g(x, q)$ is readily obtained by inspecting the proofs of~\cite{BLS} in the majority of cases, and we carry out the necessary computations here for the cases where more work is required. 

For the remainder of the section, we fix $G$ with $G_0\in \mathcal{E}_2$ and $H\in \mathcal{M}_G\setminus \mathcal{P}_G$, and we proceed by establishing the bounds in question.

\begin{prop}\label{prop:e6-fpr}
Suppose that $G_0 = E^{\epsilon}_6(q)$ with $q\geqs 3$, and $x\in G$ has prime order $r$. Then ${\rm fpr}(x, G/H) \leqs g(x, q)$, where $g(x, q)$ is recorded in Table~\ref{t:e6:fprs}.
\end{prop}

\begin{proof}
\noindent \textbf{Case 1.} \emph{$x$ is semisimple:} Here, \cite[Theorem 2]{LLS} implies that ${\rm fpr}(x, G/H) \leqs 2q^{-12}$, and we claim that ${\rm fpr}(x, G/H) \leqs q^{-14}$ if ${\rm dim}\, x^{\bar{G}} \geqs 66$. Let $L$ denote the simply connected group of type $E_6^{\epsilon}(q)$ and recall that $L = G_0.Z$, where $Z = Z(L)$. Also recall that $Z \cong C_3$ if $q\equiv 1 \imod{3}$ and $Z = 1$ otherwise.

 By close inspection of~\cite{FJ1}, we find that 
\[
|C_L(x)| \leqs q^3(q^2-1)(q^3+1)(q^2-q+1)(q^2+q+1)
\]
when $\epsilon = +$ and ${\rm dim}\, x^{\bar{G}} \geqs 66$. Since $|x^L|$ can split into at most $3$ $G_0$-classes, we have 
\begin{align*}
|x^G| &\geqs \frac{|E_6(q)|}{3q^3(q^2-1)(q^3+1)(q^2+q+1)(q^2-q+1)}\\ &\geqs \frac{1}{3}q^{33}(q^2-1)(q^3-1)(q^5-1)(q^6+1)(q^8-1)(q^9-1)\\ &\geqs \frac{q^{39}(q^2-1)(q^3-1)(q^5-1)(q^8-1)(q^9-1)}{3(q^2-1)}
\end{align*}
so by Lemma~\ref{lem:number-theory} we have $|x^G| > \frac{1}{6}q^{64}$.

Similarly, if $\epsilon = -$, then we find
\[
|C_L(x)| \leqs q^3(q^2-1)(q^3+1)(q+1)^4
\]
and again using Lemma~\ref{lem:number-theory} we find $|x^G| > \frac{1}{6}q^{64}$. 

If $|H| \leqs q^{48}$, then we deduce that ${\rm fpr}(x, G/H) < 2q^{-15} < q^{-14}$, as required, so now suppose $|H| > q^{48}$. Here,~\cite{Cr} tells us that $H$ must be of type $F_4$, in which case we verify that ${\rm fpr}(x, G/H) < q^{-14}$ by inspecting the proof of~\cite[4.19]{BLS}.

\vs

\noindent \textbf{Case 2.} \emph{$x$ is unipotent:} Here, $x\in G_0$. If $p > 2$, then we check that the relevant bounds hold by inspecting~\cite[4.15--4.21]{BLS}, so let us assume that $p = 2$. If ${\rm dim}\, x^{\bar{G}} < 40$, then the claim follows directly by~\cite[Theorem 2]{LLS}, so we may now assume that ${\rm dim}\, x^{\bar{G}} \geqs 40$. That is, $x$ is in the $G$-class $A_1^3$. If $|H| \leqs q^{32}$, or  $H$ is of type $E_6^{\epsilon}(q^{1/2})$, then the claim is verified in the proofs of \cite[4.15, 4.21]{BLS} respectively, so we may assume that $H = N_G(\bar{M}_{\sigma})$ where $\bar{M}$ is a $\sigma$-stable subgroup of $\bar{G}$. 

First suppose that $\bar{M} = A_1A_5$. By the Bala-Carter Theorem (see~\cite{Bala-Carter-I, Bala-Carter-II}), which provides a description of unipotent elements in simple algebraic groups, the $\bar{M}$-class of $x$ corresponds to a pair $(L, P_{L'})$, where $L$ is a Levi subgroup of $\bar{M}$ and $P_{L'}$ is a distinguished parabolic subgroup of $L'$. If $L$ is also a Levi subgroup of $\bar{G}$, then the $\bar{G}$-class has the same label as the $\bar{M}$-class, whilst if $L = A_1^4$, then we deduce that $x$ lies in the $G$-class $A_1^3$ via \cite[Table 17]{La}. Now it is shown in \cite[4.16]{BLS} that ${\rm fpr}(x, G/H) < 4q^{-\delta(x)}$, where $\delta(x) = {\rm dim}\, x^{\bar{G}} - {\rm dim}\, (x^{\bar{G}}\cap \bar{M})$, and so we compute that ${\rm fpr}(x, G/H) < 4q^{-20}$. 

Next consider the case $\bar{M} = D_5T_1$. Here~\cite[Tables 9 and 10]{Cr} give $H_0 = H\cap G_0 = {\rm P\O}_{10}^{\epsilon}(q)\times \frac{q-\epsilon}{e}$, where $e = (3, q-\epsilon)$. Therefore, \cite[1.3]{LLS} implies that 
\[
|x^G\cap H| \leqs i_2(H_0) = i_2({\rm Aut}({\rm P\O}_{10}^{\epsilon}(q)) < 2(q+1)q^{24}
\]
so the claim is satisfied.

Now suppose that $\bar{M} = T_2D_4.S_3$. Here we deduce via~\cite[Tables 9 and 10]{Cr} that $x^G\cap H \subseteq \widetilde{H}$, where $\widetilde{H} = {\rm O}_8^{+}(q)$ or ${}^3D_4(q)$. Hence, we have
\[
|x^G\cap H| \leqs i_2(\widetilde{H}) \leqs 2q^{16}
\]
so certainly ${\rm fpr}(x, G/H) < 2q^{-8}$, as required.

Finally, if $\bar{M} = F_4$, then we deduce via \cite[Table A]{La95} that there are no involutions $x\in H$ with ${\rm dim}\, x^{\bar{G}} \geqs 40$.

\vs

\noindent \textbf{Case 3.} \emph{$x$ is a graph or field or graph-field automorphism:} If $r$ is odd and $|H| \leqs q^{32}$ we observe that ${\rm fpr}(x, G/H) \leqs 6q^{32-78(1-1/r)}$, as
\[
|x^G| \geqs |E_6^{\epsilon}(q) \, : \, E_6^{\epsilon}(q^{1/r})| > \frac{1}{6}q^{78(1-1/r)}.
\]
In all other cases the bounds derived in the proofs of~\cite[4.15--4.21]{BLS} are again sufficient, so the proof is complete.
\end{proof}

\begin{prop}\label{prop:e6-fpr-q2}
Suppose that $G_0 = E^{\epsilon}_6(2)$ and $x\in G$ has prime order $r$. Then ${\rm fpr}(x, G/H) \leqs g(x, q)$, where $g(x, q)$ is recorded in Table~\ref{t:e6:fprs-q2}.
\end{prop}

\begin{proof}
If $H$ is almost simple with socle ${\rm Fi}_{22}$, then the fusion of conjugacy classes of $H$ in $G$ is computable in \textsf{GAP} using the Character Table Library, so we can compute precise fixed point ratios for all $x\in G$, as noted in Section~\ref{sss:comp-classical}. The same is true if $G = {}^2E_6(2).2$ and $H = {\rm SO}_7(3)$. So, in view of~\cite[4.14]{BLS}, we may assume from now on that $|H|\leqs 2^{32}$ or $H = N_{\bar{G}}(\bar{M}_{\sigma})$ where $\bar{M}$ is some $\sigma$-stable subgroup of $\bar{G}$.

\vs

\noindent \textbf{Case 1.} \emph{$x$ semisimple:} If ${\rm dim}\, x^{\bar{G}} < 42$, then the claim follows directly by \cite[Theorem 2]{LLS}, so we may assume that ${\rm dim}\, x^{\bar{G}} \geqs 42$. If $H = N_{\bar{G}}(\bar{M}_{\sigma})$ with $\bar{M} = F_4$, then we note that $G \leqs G_0.2$, and the character tables of $G_0$ and $H_0 = H\cap G_0$ are both available in \textsf{GAP}. Since $x\in G_0$, we can compute the fusion of semisimple $H_0$-classes in $G_0$ and verify that the bounds in Table~\ref{t:e6:fprs-q2} are satisfied for all semisimple elements. 

Next, if $H = N_{\bar{G}}(\bar{M}_{\sigma})$ with $\bar{M} = D_5T_1$ then the bounds in the proof of \cite[4.17]{BLS} suffice.

In all remaining cases, we can infer from~\cite[Tables 9 and 10]{Cr} that $|x^G\cap H| \leqs 2^{38}$ and we immediately get ${\rm fpr}(x, G/H) \leqs 2^{38}/|x^G|$ when ${\rm dim}\, x^{\bar{G}} > 42$, so finally let us assume that ${\rm dim}\, x^{\bar{G}} = 42$, so $C_{\bar{G}}(x)^0 = T_1A_5$. 

First suppose that $|H| \leqs 2^{32}$. We can infer from \cite[Table 9]{BLS} that $|x^G| > 2^{41}$, so we have ${\rm fpr}(x, G/H) < 2^{-9}$. 
 
 Now assume that $H = N_{\bar{G}}(\bar{M}_{\sigma})$ where $\bar{M}$ is some $\sigma$-stable subgroup of $\bar{G}$. If $\bar{M} = A_1A_5$, then ${\rm fpr}(x, G/H) < 2^{-14}$ as shown in~\cite[4.16]{BLS}. 
 
 Finally, if $\bar{M} = T_2D_4.S_3$, then inspection of~\cite[Tables 9, 10]{Cr} gives 
 \[
 |x^G\cap H| < |H_0| \leqs 18\cdot |{\rm P\O}_8^{+}(2)| < 2^{33}
 \]
 where $H_0 = H\cap G_0$, so the claim follows.
 \vs

\noindent \textbf{Case 2.} \emph{$x$ is unipotent:} If $x$ is in class $A_1$, then~\cite[Theorem 2]{LLS} gives ${\rm fpr}(x, G/H) \leqs 2^{-5}$, so next assume that $x$ is in class $A_1^2$ or $A_1^3$. If $|H| \leqs 2^{32}$, then the claim is proved in~\cite[4.15]{BLS}, so we may now assume that $H = N_{\bar{G}}(\bar{M}_{\sigma})$ where $\bar{M}$ is some $\sigma$-stable subgroup of $\bar{G}$.

If $\bar{M} = A_1A_5$, then the bound ${\rm fpr}(x, G/H) \leqs 4\cdot2^{-\delta(x)}$ is still valid, as highlighted in~\cite[4.16]{BLS}, and we find that $\delta(x) = 16$ if $x$ is in $A_1^2$ and $\delta(x) = 20$ if $x$ is in $A_1^3$, so the claim holds. 

If $\bar{M} = D_5T_1$, then we note as above that $x^G\cap H \subseteq {\rm \O}_{10}^{\epsilon}(2)$, and so using {\sc Magma} we find
\[
|x^G\cap H| \leqs i_2({\rm \O}_{10}^{\epsilon}(2)) = 21999615.
\]
One can check that this bound is sufficient in both cases. 
Now suppose that $\bar{M} = T_2D_4.S_3$. As above we have
\[
|x^G\cap H| \leqs \max \{i_2({}^3D_4(2).3), i_2(\O_8^{+}(q).S_3)\} \leqs 2^{15}
\]
which is sufficent.

\vs

\noindent \textbf{Case 3.} \emph{x is an involutory graph automorphism:} Here, the bounds in Table~\ref{t:e6:fprs-q2} can be verified by inspecting~\cite[4.15--4.21]{BLS}, so the proof is complete.

\end{proof}

{\small
\begin{table}
\[
\begin{array}{lll} \hline
x &   g(x, q) & \text{Conditions} \\ \hline
x = s&  2q^{-12} &  {\rm dim}\, x^{\bar{G}} < 66\\
& q^{-14} & {\rm dim}\, x^{\bar{G}} \geqs 66\\
x = u&2q^{-6}& {\rm dim}\, x^{\bar{G}} < 32, \, p>2\\
& q^{-6}& 32\leqs {\rm dim}\, x^{\bar{G}} < 40, \, p>2\\
& q^{-8}& 40 \leqs {\rm dim}\, x^{\bar{G}} < 48, \, p>2\\
& q^{-14}& 48\leqs {\rm dim}\, x^{\bar{G}} < 54, \, p > 2\\
& q^{-13}& {\rm dim}\, x^{\bar{G}} \geqs 54, \, p>2\\
 & 2q^{-6}& {\rm dim}\, x^{\bar{G}} < 40, \, p = 2\\
& 2q^{-8}& {\rm dim}\, x^{\bar{G}} \geqs 40, \, p = 2\\
x = \phi  \text{ or } \phi\tau& q^{-12} & r = 2\\
& 12q^{-26(1-1/r)}& r\geqs 3\\
x = \tau& q^{-5}& C_{\bar{G}}(x) = F_4\\
& 12q^{-10}& C_{\bar{G}}(x) \neq F_4\\
\hline
\end{array}
\]
\caption{Bounds on ${\rm fpr}(x, G/H)$ for $G_0 = E^{\epsilon}_6(q)$, $H$ non-parabolic and $q\geqs 3$.}
\label{t:e6:fprs}
\end{table}
}

{\small
\begin{table}
\[
\begin{array}{lll} \hline
x &   g(x, q) & \text{Conditions} \\ \hline
x = s&  2^{-6} &  {\rm dim}\, x^{\bar{G}} < 42\\
&  2^{-8} &  {\rm dim}\, x^{\bar{G}} = 42\\
& 2^{38}/|x^G| &42 <{\rm dim}\,  x^{\bar{G}}\leqs 54 \\
&  2^{-17} &  {\rm dim}\, x^{\bar{G}} > 54\\
x = u & 2^{-5}& {\rm dim}\, x^{\bar{G}} < 32\\
& 2^{-6}& {\rm dim}\, x^{\bar{G}} < 40\\
& 2^{-8}& {\rm dim}\, x^{\bar{G}} \geqs 40\\
x = \tau& 2^{-5}& C_{\bar{G}}(x) = F_4\\
 & 3\cdot2^{-8}& C_{\bar{G}}(x) \neq F_4\\
\hline
\end{array}
\]
\caption{Bounds on ${\rm fpr}(x, G/H)$ for $G_0 = E^{\epsilon}_6(2)$ and $H$ non-parabolic.}
\label{t:e6:fprs-q2}
\end{table}
}

\begin{prop}\label{prop:e7-fpr}
Suppose that $G_0 = E_7(q)$ and $x\in G$ has prime order $r$. Then ${\rm fpr}(x, G/H) \leqs g(x, q)$, where $g(x, q)$ is recorded in Table~\ref{t:e7:fprs}.
\end{prop}

\begin{proof}
Appealing to~\cite[4.7]{BLS} we find that either $|H| \leqs q^{46}$, or $H$ is of type $E_7(q^{1/2})$, or $H = N_G(\bar{M}_{\sigma})$, where $\bar{M} = T_1E_6.2, A_1D_6, A_7.2$, or $A_1F_4$.

\vs

\noindent \textbf{Case 1.} \emph{$x$ is semisimple:} We first consider the case $q\geqs 3$, as $q = 2$ needs special attention. Here,~\cite[Theorem 2]{LLS} implies that ${\rm fpr}(x, G/H) < q^{-19}$, so the desired bound holds if ${\rm dim}\, x^{\bar{G}} \leqs 70$. We may now thus assume that ${\rm dim}\, x^{\bar{G}} > 70$. Recall that if $L$ denotes the simply connected group of type $E_7(q)$, then $L = Z.G_0$, where $Z = Z(L)$ and $Z \cong C_2$ if $q$ is odd and $Z = 1$ if $q$ is even. By carefully inspecting~\cite{FJ1}, we find that 
\[
C_L(x) \leqs |{\rm SU}_7(q)|(q+1) = q^{21}(q+1)(q^2-1)(q^3+1)(q^4-1)(q^5+1)(q^6-1)(q^7+1)
\]
when ${\rm dim}\, x^{\bar{G}} > 70$, so
\begin{align*}
|x^L| &> q^{42}(q^3-1)(q^4+1)(q^5-1)(q^6+1)(q^7-1)(q^{17} + q^{16} + \cdots + 1)\\
& > q^{69}(q^3-1)(q^5-1)(q^7-1)\\
\end{align*}
and using Lemma~\ref{lem:number-theory}, we find $|x^L| > \frac{1}{2}q^{82}$. Since $x^L$ is a union of at most $2$ distinct $G_0$-classes, we get $|x^G|> \frac{1}{4}q^{82}$.
 
 If $|H| \leqs q^{46}$, then the trivial bound ${\rm fpr}(x, G/H) \leqs |H|/|x^G|$ is sufficient. Next, assume that $H$ is of type $E_7(q^{1/2})$ and let $H_0 = G_0\cap H$. 
 
 First suppose $q$ is odd. Inspecting~\cite[Table 4.1]{Cr2} tells us that $H_0 = {\rm Inndiag}(E_7(q^{1/2})) = E_7(q^{1/2}).2$. Let $\xi : \alpha \mapsto \alpha^{q^{1/2}}$, so $H_0 = \bar{G}_{\xi}$. If $C_{\bar{G}}(x)$ is connected, then a well-known corollary to the Lang-Steinberg theorem~\cite[I, 2.7]{SS} gives $x^{\bar{G}} \cap \bar{G}_{\xi} = x^{\bar{G}_{\xi}}$ and similarly $x^{\bar{G}}\cap \bar{G}_{\xi^2} = x^{\bar{G}_{\xi^2}}$. Hence, in view of Proposition~\ref{thm:GLS-ss} we have $x^{G_0}\cap H = x^{H_0}$. On the other hand, if $C_{\bar{G}}(x)$ is disconnected, then the fact that $|C_{\bar{G}}(x) \, : \, C_{\bar{G}}(x)^0| = 2$ implies that $x^{\bar{G}}\cap H_0$ might be a union of two distinct $H_0$-classes. Similarly, $x^{\bar{G}}\cap G_0$ might be a union of two distinct $G_0$-classes. We therefore get
 \begin{equation}\label{eq:disconnected}
 {\rm fpr}(x, G/H) \leqs \frac{2\max \{|y^{H_0}| \, : \, y\in x^{\bar{G}}\}}{\min \{|y^{G_0}| \, : \, y\in x^{\bar{G}}\}}
 \end{equation}
We can now inspect Lübeck's online data on semisimple classes~\cite{Lu2} to deduce which elements have disconnected centraliser in $\bar{G}$ and compute ${\rm fpr}(x, G/H)$ using~\eqref{eq:disconnected}. 

For example, suppose that $C_{\bar{G}}(x)$ is of type $E_6T_1$. Then~\cite{Lu2} tells us that $C_{\bar{G}}(x) = E_6T_1.2$ is disconnected, and in particular we find that
\[
|x^{G_0}\cap H| \leqs 2\frac{|E_7(q^{1/2})|}{2|{}^2E_6(q^{1/2})|(q^{1/2}+1)} < q^{27}
\]
and 
\[
|x^{G_0}| = \frac{|E_7(q)|}{2|E_6(q)|(q-1)} > \frac{1}{2}q^{54}
\]
so the claim holds. The other cases are all similar, so we leave the details to the reader. 

If $q$ is even, then $G_0 = \bar{G}_{\sigma}$ and again inspecting~\cite[Table 4.1]{Cr2} we find that $H_0 = E_7(q^{1/2})$. Using~\cite[I, 2.7]{SS} once again, we deduce that $x^{G_0}\cap H_0 = x^{H_0}$, so the bounds derived in the case where $q$ is odd are again sufficient.

Finally, if $H = N_G(\bar{M}_{\sigma})$ for some $\sigma$-stable subgroup $\bar{M}$ of $\bar{G}$, then one can check that the bounds in Table~\ref{t:e7:fprs} hold by inspecting the proofs of~\cite[4.9, 4.10]{BLS} and~\cite[3.4]{B18}. 

Now suppose that $q = 2$. In this case the prime order $\bar{G}_{\sigma}$-classes and the possible centraliser types are conveniently recorded in~\cite[Table 2]{BBR}. If ${\rm dim}\, x^{\bar{G}} \leqs 66$, then the claim follows from~\cite[Theorem 2]{LLS}, so we may assume that ${\rm dim}\, x^{\bar{G}} > 66$.

If $|H| \leqs 2^{46}$, then one can can obtain much stronger bounds than the ones in Table~\ref{t:e7:fprs} by arguing as in the case $q\geqs 3$, so we may assume that $H = N_G(\bar{M}_{\sigma})$ for some $\sigma$-stable subgroup $\bar{M}$ of $\bar{G}$. If $66 < {\rm dim}\, x^{\bar{G}} \leqs 84$, then we can check that ${\rm fpr}(x, G/H) < 2^{-19}$ for all the remaining possibilities for $H$ by inspecting the proofs of ~\cite[4.9, 4.10]{BLS} and~\cite[3.4]{B18}. Now assume ${\rm dim}\, x^{\bar{G}} > 84$. By inspecting the same proofs, we again obtain sufficient fixed point ratio bounds that this time depend on the value of $z = {\rm dim} \, Z(\bar{D}^0)$, where $\bar{D} = C_{\bar{G}}(x)$. In particular, we find that ${\rm fpr}(x, G/H) < 2^{-21}$ if $z\leqs 3$ and ${\rm fpr}(x, G/H) < 2^{-24}$ if $z > 3$. The authors derive those bounds depending on $z$ by appealing to~\cite[4.5]{LLS}, which implies that
\[
{\rm fpr}(x, G/H) < \frac{|W(\bar{G}) \, : \, W(\bar{M})|\cdot |\bar{M}/\bar{M^0}|\cdot 2\cdot (q+1)^z \cdot 2}{q^{\delta(x) + z - l}(q-1)^l}
\]
where $W(\bar{G})$ and $W(\bar{M})$ denote the Weyl groups of $\bar{G}$ and $\bar{W}$ respectively, $\delta(x) = {\rm dim}\, x^{\bar{G}} - {\rm dim}(x^{\bar{G}}\cap \bar{M})$ and $l$ is the semisimple rank of $\bar{M}$.
\vs

\noindent \textbf{Case 2.} \emph{$x$ is unipotent:} We first consider the case $p = 2$. If ${\rm dim}\, x^{\bar{G}} \leqs 54$, then the claim follows directly by~\cite[Theorem 2]{LLS}, so we may assume that ${\rm dim}\, x^{\bar{G}} > 54$. In this case, $x$ does not lie in the class $A_1$ or $A_1^2$, and so inspecting~\cite[Table 22.2.2]{LSeitz2} we find that $|x^G| > q^{64}$.

 If $|H| \leqs q^{46}$, then we have ${\rm fpr}(x, G/H) < |H|/q^{64} < q^{-18}$. Next suppose that $H$ is of type $E_7(q^{1/2})$. Since $x$ is unipotent, we must have $x\in H_0 = H\cap G_0$, and note that both the $H_0$-class and the $G_0$-class of $x$ are determined by the labelling of its class in $\bar{G}$. Therefore, $|x^{G_0}\cap H_0| = |x^{H_0}|$, and in particular, ${\rm fpr}(x, G/H) = |x^{H_0}|/|x^{G_0}| < q^{-18}$. 
 
Finally, assume that $H = N_G(\bar{M}_{\sigma})$ for some $\sigma$-stable subgroup $\bar{M}$ of $\bar{G}$. If $\bar{M} = A_1D_6$, then the fusion of $H$-classes in $G$ is described in~\cite[Table 3]{BLS} and from this we can compute precise fixed point ratios and verify that the claim holds. For example, suppose that $x$ is in class $(A_1^{3})^{(2)}$, and write $x = uy$, where $u\in A_1$ and $y\in D_6$. Then ~\cite[Table 3]{BLS} tells us that either $u = 1$ and $y$ embeds in $D_6$ as an $a'_6$ involution (see~\cite[Section 3.5.4]{BG} for more information on unipotent involutions in orthogonal groups), or $u\neq 1$ (so $u$ has Jordan decomposition $[J_2]$ on the natural $A_1$-module), and $y$ embeds in $D_6$ as a $c_2$ involution. We therefore get
\[
|x^{G_0}\cap H| \leqs \frac{1}{2}q^6(q^6-1)(q^8-1)(q^{10}-1) + q^{4}(q^2-1)(q^{6}-1)(q^{10}-1) < q^{30}
\]
and so ${\rm fpr}(x, G/H) < q^{30-64} < q^{-18}$. The other cases are all similar, so we omit the details.

Finally, one can check that ${\rm fpr}(x, G/H) < q^{-18}$ for all other possibilities for $\bar{M}$ by inspecting the proofs of \cite[4.9, 4.10]{BLS} and \cite[3.4]{B18}.

Now assume that $p > 2$. If ${\rm dim}\, x^{\bar{G}} \leqs 66$, then the desired bound once again follows from~\cite[Theorem 2]{LLS}. Now suppose that ${\rm dim}\, x^{\bar{G}} > 66$. If $H$ is of type $A_1D_6$, then we verify the claim by inspecting the proof of~\cite[4.10]{BLS}, and in every other case the claim is proved by arguing exactly as in the case $p = 2$. 

\vs

\noindent \textbf{Case 3.} \emph{$x$ is a field automorphism:} If $r = 2$, then~\cite[Theorem 2]{LLS} gives ${\rm fpr}(x, G/H) \leqs q^{-22}$, so now assume that $r$ is odd. In this case, we can infer from~\cite[Table 4.1]{Cr2} that $C_{G_0}(x) = E_7(q^{1/r})$. Then Lemma~\ref{lem:number-theory} gives
\[
|x^{G_0}| = |E_7(q) \, : \, E_7(q^{1/r})| > \frac{1}{2}q^{-133(1-1/r)} \geqs \frac{1}{2}q^{-88}.
\]
If $|H| \leqs q^{46}$, then we get ${\rm fpr}(x, G/H) < \frac{1}{2}q^{-42}$, so the claim holds. If $H$ is of type $E_7(q^{1/2})$, then $x$ also acts as a field automorphism on $E_7(q^{1/2})$, and so Lemma~\ref{lem:inndiag-fpr} implies
\[
{\rm fpr}(x, G/H) \leqs \frac{|E_7(q^{1/2}) \, : \, E_7(q^{1/2r})|}{|E_7(q) \, : \, E_7(q^{1/r})|} < 4q^{-44}.
\]
In all other cases, one can verify that the bounds in Table~\ref{t:e7:fprs} are satisfied by inspecting the proofs of \cite[4.9, 4.10]{BLS} and \cite[3.4]{B18}, so the proof is complete.
\end{proof}

{\small
\begin{table}
\[
\begin{array}{lll} \hline
x &   g(x, q) & \text{Conditions} \\ \hline
x = s&  q^{-19} &  {\rm dim}\, x^{\bar{G}} \leqs 70, \, q\geqs3\\
& q^{-24} & {\rm dim}\, x^{\bar{G}} > 70,\, q\geqs 3\\
 & 2^{-12} & {\rm dim}\, x^{\bar{G}} \leqs 66,\, q = 2\\
 & 2^{-19} & 66 < {\rm dim}\, x^{\bar{G}} \leqs 84, \, q = 2\\
 & 2^{-21} & {\rm dim}\, x^{\bar{G}} > 84,\, q = 2, \, \text{and } z\leqs 3\\ 
& 2^{-24} &  {\rm dim}\, x^{\bar{G}} > 84, \, q = 2, \, \text{and } z > 3\\
x = u&2q^{-12}& {\rm dim}\, x^{\bar{G}} < 66, \, p>2\\
& q^{-22}& {\rm dim}\, x^{\bar{G}} \geqs 66, \, p>2\\
 & 2q^{-12}& {\rm dim}\, x^{\bar{G}} < 54, \, p = 2\\
& q^{-18}& {\rm dim}\, x^{\bar{G}} \geqs 54, \, p = 2\\
x = \phi & q^{-22} & r = 2\\
& 8q^{-32}& r\geqs 3\\
\hline
\end{array}
\]
\caption{Bounds on ${\rm fpr}(x, G/H)$ for $G_0 = E_7(q)$ and $H$ non-parabolic.}
\label{t:e7:fprs}
\end{table}
}

\begin{prop}\label{prop:f4-fpr}
Suppose that $G_0 = F_4(q)$ and $x\in G$ has prime order $r$. Then ${\rm fpr}(x, G/H) \leqs g(x, q)$, where $g(x, q)$ is recorded in Table~\ref{t:f4:fprs} if $q\geqs 3$ and Table~\ref{t:f4:fprs-q2} if $q = 2$.
\end{prop}

\begin{proof}
First note that either $|H| \leqs q^{22}$, or $H$ is one of the possibilities recorded in~\cite[4.23]{BLS}. The case $q = 2$ will be dealt with using different methods, so for now we will assume $q\geqs 3$.

\vs

\noindent \textbf{Case 1.} \emph{$x$ is semsimple:} If $C_{\bar{G}_{\sigma}}(x) = B_4$, then the bound immediately follows from~\cite[Theorem 2]{LLS}, so it is left to deal with the remaining possibilities for $x$. First assume that $|H| \leqs q^{22}$. If $C_{\bar{G}_{\sigma}}(x)$ has a $B_3$ or $C_3$ factor, then ${\rm dim}\, x^{\bar{G}} \geqs 28$, and inspecting~\cite{Lu2} gives
\[
|x^G| \geqs \frac{|F_4(q)|}{|{\rm Sp}_6(q)||{\rm L}_2(q)|} = q^{14}(q^8+q^4+1)(q^4+1)(q^2+1) > q^{28}. 
\]
 It follows that ${\rm fpr}(x, G/H) < q^{-6}$, as required. In all remaining cases we find that $|x^G| > q^{31}$, so we get ${\rm fpr}(x, G/H) < q^{-9}$.

If $H$ is of type $F_4(q^{1/2})$, then we refer to~\cite[Table 8]{Cr} and we find that $H_0 = H\cap G_0 = F_4(q^{1/2})$. Now we note that $\bar{G}$ is simply connected, and so $C_{\bar{G}}(x)$ is connected. Therefore, appealing to~\cite[I, 2.7]{SS} we find that $x^G\cap H_0 = x^{H_0}$. We can therefore compute precise fixed point ratios for all semisimple elements and verify that the bounds in Table~\ref{t:f4:fprs} are satisfied for this choice of $H$. In all other cases we appeal to the proofs of~\cite[4.22, 4.24--4.28]{BLS} and we verify that the claim holds.

\vs

\noindent \textbf{Case 2.} \emph{$x$ is unipotent, field or graph-field automorphism:} The bounds in Table~\ref{t:f4:fprs} can be directly verified by inspecting~\cite[4.22, 4.24--4.28]{BLS}, unless $H$ is of type $F_4(q^{1/2})$. If $x$ is unipotent, then we claim that $x^{G_0}\cap H = x^{H_0}$. This follows from the fact that both the $G_0$ and $H_0$-class of $x$ are determined by the labelling of its class in $\bar{G}$, and so one can verify that the bounds in Table~\ref{t:f4:fprs} are satisfied. Finally, if $x$ is a field (respectively, graph-field) automorphism, then $x$ also acts as a field (respectively, graph-field) automorphism on $H$, and so we get
\[
{\rm fpr}(x, G/H) = \frac{|F_4(q^{1/2}) \, : \, F_4^{\epsilon}(q^{1/2r})|}{|F_4(q) \, : \, F^{\epsilon}_4(q^{1/r})|}
\]
which is sufficient.

\vs 

Now assume that $q = 2$, so $G = G_0$, or $G_0.2$. If $x$ is an involutory graph-field automorphism, then~\cite[Theorem 2]{LLS} implies that ${\rm fpr}(x, G/H) < 2^{-6}$, so we may now assume that $x\in G_0$.

Recall that if $|H| > 2^{22}$, then the possibilities for $H$ are determined in~\cite[4.23]{BLS}, and so we can check that the character tables of both $G_0 = F_4(2)$ and $H_0 = H\cap G_0$ are computable in \textsf{GAP}. It follows that we can determine the fusion of $H_0$-classes in $G_0$, and we note that ${\rm fpr}(x, G/H) \leqs \frac{|x^{G_0}\cap H_0|}{|x^{G_0}|}$ by Lemma~\ref{lem:inndiag-fpr}. Now assume that $|H|\leqs 2^{22}$. If ${\rm dim}\, x^{\bar{G}} < 28$, then as noted in~\cite[4.22]{BLS}, $x$ is an involution belonging to one of the $\bar{G}$-classes labelled $A_1, \tilde{A}_1$, or $\tilde{A}_1^{(2)}$, and we have ${\rm fpr}(x, G/H) \leqs (2^4-2^2+1)^{-1} = 13^{-1}$ by~\cite[Theorem 1]{LLS} if $x\in A_1$ or $\tilde{A}_1$, and ${\rm fpr}(x, G/H) < 2^{-4}$ if $x\in \tilde{A}_1^{(2)}$. If ${\rm dim}\, x^{\bar{G}} \geqs 28$, then ${\rm fpr}(x, G/H) \leqs 2^{22}/|x^G|$. We now set
\[
g(x, q) = \max \{{\rm fpr}(x, G/H) \, : \, H\in \mathcal{M}_G\setminus \mathcal{P}_G\}
\]
 and we record this value in Table~\ref{t:f4:fprs-q2} for each prime order class. For semisimple elements we denote $x$ by its centraliser type, whereas for unipotent elements, we use the notation in~\cite{LSeitz2}.
\end{proof}

{\small
\begin{table}
\[
\begin{array}{lll} \hline
x &   g(x, q) & \text{Conditions} \\ \hline
x = s&  2q^{-5} &  C_{\bar{G}}(x) = B_4 \\
x = s& q^{-6} & C_{\bar{G}}(x) = B_3T_1, \text{ or } C_3T_1\\
x = s & 2q^{-9} &  C_{\bar{G}}(x) = A_2\tilde{A}_2, A_1\tilde{A}_2T_1, \tilde{A}_2T_2, \text{ or } B_2T_2\\
x = s & q^{-9} & C_{\bar{G}}(x) \text{ other}\\
x = u&(q^4-q^2+1)^{-1}& {\rm dim}\, x^{\bar{G}} \leqs 22\\
x = u & q^{-8}& 22 < {\rm dim}\, x^{\bar{G}} \leqs 28, \, p>2\\
x = u& 3q^{-8}& 28 < {\rm dim}\, x^{\bar{G}} \leqs 30,  \, p>2\\
x = u& 3q^{-10} & {\rm dim}\, x^{\bar{G}} > 30, \, p>2\\
x = u& q^{-6}& {\rm dim}\, x^{\bar{G}} > 22,\, p = 2 \\
x = \phi \text{ or } \phi\tau& q^{-6}& r = 2\\
x = \phi& 4q^{-16(1-1/r)} & r\geqs 3\\
\hline
\end{array}
\]
\caption{Bounds on ${\rm fpr}(x, G/H)$ for $G_0 = F_4(q)$, $H$ non-parabolic and $q\geqs 3$.}
\label{t:f4:fprs}
\end{table}
}

{\small
\begin{table}
\[
\begin{array}{lll} \hline
x & |x| & {\rm fpr}(x, G/H) \leqs\\
\hline
A_1 &2 & 13^{-1}\\
\tilde{A}_1&2 &  13^{-1}\\
\tilde{A}_1^{(2)}&2 &  2^{-4}\\
A_1\tilde{A}_1&2 & 262144/21928725\\
C_3T_1& 3& 128/23205\\
B_3T_1& 3&128/23205\\
A_2\tilde{A}_2&3 &1/5824\\
B_2T_2&5& 1/5376 \\
A_2T_2&7& 1/2496\\
\tilde{A}_2T_2&7& 1/113152\\
T_4&13& 1/5222400 \\
T_4& 17&1/69888\\
\phi\tau& 2& 2^{-6}\\
\hline
\end{array}
\]
\caption{bounds on ${\rm fpr}(x, G/H)$ for $G_0 = F_4(2)$ and $H$ non-parabolic.}
\label{t:f4:fprs-q2}
\end{table}
}

\section{Fixed point ratios for parabolic actions}\label{s:parabolics}
 In this section, we describe how we obtain fixed point ratio estimates for a group $G$ with socle $G_0\in \mathcal{E}_2\cup \mathcal{E}_3$ in the natural action of $G$ on $G/H$ for some maximal parabolic subgroup $H$ of $G$. We will be using the standard $P_m$ parabolic notation for maximal parabolic subgroups of exceptional groups. Moreover, we will denote the parabolic subgroup corresponding to removing nodes $m$ and $n$ from the Dynkin diagram of $G$ by $P_{m, n}$. 
 
 We will be using the same methods that are used and discussed in~\cite{LLS}. In particular, if $\O = G/H$, then we have
 \[
 {\rm fpr}(x, G/H) = \frac{\chi(x)}{|\O|},
 \]
 where $\chi(x) = 1_{\bar{H}_{\sigma}}^{\bar{G}_{\sigma}}$ is the corresponding permutation character. We divide this section into four subsections, where we discuss our approach for semisimple elements, unipotent elements, field and graph-field automorphisms, and graph automorphisms respectively. 

\subsection{Semisimple elements}\label{ss:semisimple}

Let $x\in \bar{G}_{\sigma}$ be a semisimple element and let $\O = G/H$. As described above, $|C_{\O}(x)| = \chi(x)$, where $C_{\O}(x) = \{\omega \in \O \, : \, \omega^x = \omega\}$ is the fixed point set of $x$ in $\O$. 

Let $W$ denote the Weyl group of $\bar{G}$ and let $W_{\bar{H}}$ denote the Weyl group of $\bar{H}$, so $W_{\bar{H}}$ is a standard parabolic subgroup of $W$. We write $\Phi$ for the root system of $\bar{G}$ with respect to a fixed maximal torus, we let $\Pi$ be a simple system of roots for $\bar{G}$ and $\alpha_0$ to be the highest root in $\Phi$ with respect to $\Pi$. The possible centraliser types of semisimple elements in $\bar{G}_{\sigma}$ are in one-to-one correspondence with the pairs $(J, [w])$, where $J\subsetneq \Pi \cup \{\alpha_0\}$ and is determined up to conjugacy in $W$ and $[w] = W_Jw$ is a conjugacy class representative of $N_W(W_J)/W_J$, where $W_J$ denotes the subgroup of $W$ generated by reflections in the roots in $J$. This is all described in~\cite{Derz, FJ1, FJ2}. 

Recall that $W$ is $\sigma$-stable, and so $\sigma$ induces permutation of the elements of $W$. We say that two elements $w, w'\in W$ are $\sigma$-conjugate if there exists $x\in W$ such that $w' = x^{-1}w\sigma(x)$. Now let $C_1, \ldots, C_l$ be the $\sigma$-conjugacy classes of $W$. As explained in~\cite[Section 3]{LLS}, the $\sigma$-stable maximal tori of $\bar{G}$ are parametrised by the $\sigma$-classes of $W$, so we will write $T_i$ for a representative of the $\bar{G}_{\sigma}$-class of maximal tori in $\bar{G}$ corresponding to $C_i$ for every $i\in \{1, \ldots, l\}$, so that $(T_1)_{\sigma}, \ldots, (T_l)_{\sigma}$ is a list of representatives of the $\bar{G}_{\sigma}$-classes of maximal tori of $\bar{G}_{\sigma}$. 

We will also write $\epsilon_{T_i} = (-1)^{r_{T_i}}$, where $r_{T_i}$ is the relative rank of $T_i$ (that is the multiplicity of $q - \epsilon$ as a divisor of $|(T_i)_{\sigma}|$) and similarly $\epsilon_{C_{\bar{G}}(x)^0} = (-1)^{r_{C_{\bar{G}}(x)^0}}$. Finally, $w^*$ will denote the longest word of $W$ if $G_0 = {}^2E_6(q)$ and $w^* = 1$ otherwise and $|(C_{\bar{G}}(x)^0)_{\sigma}|_{p'}$ will denote the largest factor of $|(C_{\bar{G}}(x)^0)_{\sigma}|$ not divisible by $p$.

Let $x$ be in the $\bar{G}_{\sigma}$-class corresponding to the pair $(J, [w])$. Using the above description of $\bar{G}_{\sigma}$-classes, Lawther, Liebeck, and Seitz provide an explicit formula for the number of fixed points of $x$ in $\O$ in~\cite[3.2]{LLS}, which we record below:

\begin{prop}\label{prop:LLS-3.2}
Let $x\in G$ be semisimple. With the notation above, we have
\[
|C_{\O}(x)| = \chi(x) = \sum_{i = 1}^l \frac{|W|}{|C_i|} \cdot \frac{|W_{\bar{H}} \cap C_i{w^*}^{-1}|}{|W_{\bar{H}}|} \cdot \frac{|W_Jw \cap C_i|}{|W_J|} \cdot \frac{\epsilon_{C_{\bar{G}}(x)^0} \cdot \epsilon_{T_i}|(C_{\bar{G}}(x)^0)_{\sigma}|_{p'}}{|(T_i)_{\sigma}|}
\]
\end{prop}

Now using Proposition~\ref{prop:LLS-3.2}, we will derive an upper bound on $|C_{\O}(x)|$ that will be of use to us when $G_0 = E_8(q)$. Let $\{[w_1], \ldots, [w_t]\}$ be a complete set of conjugacy class representatives in $N_W(W_J)/W_J$ and order the $\sigma$-classes $C_1, \ldots, C_l$ of $W$ so that $\epsilon_{T_i} = 1$ for $i \leqs s$ and $\epsilon_{T_i} = -1$ for $i > s$ for some $s\in \{0, \ldots, l\}$. Moreover, for every $i\in \{1, \ldots, l\}$, let $L_i = \max_{j\in [t]} |W_Jw_j \cap C_i|/|W_J|$ and $R_i =  \min_{j\in [t]} |W_Jw_j \cap C_i|/|W_J|$. We then obtain the following estimate:

\begin{cor}\label{cor:3.2}
With the notation above, we have
\[
C_{\O}(x) \leqs |A - B|,
\]
where 
\[
A = \sum_{i = 1}^s \frac{|W|}{|C_i|} \cdot \frac{|W_{\bar{H}} \cap C_i{w^*}^{-1}|}{|W_{\bar{H}}|} \cdot \frac{|(C_{\bar{G}}(x)^0)_{\sigma}|_{p'}}{|(T_i)_{\sigma}|} \cdot L_i 
\]
and
\[
B =  \sum_{i = s+1}^l \frac{|W|}{|C_i|} \cdot \frac{|W_{\bar{H}} \cap C_i{w^*}^{-1}|}{|W_{\bar{H}}|} \cdot \frac{|(C_{\bar{G}}(x)^0)_{\sigma}|_{p'}}{|(T_i)_{\sigma}|} \cdot R_i 
\]

\end{cor}

\begin{proof}
This follows almost immediately from Proposition~\ref{prop:LLS-3.2}. In particular, we note that $\epsilon_{C_{\bar{G}(x)^0}}$ is independent of $i$, so we can omit it from all the summands, and take an absolute value of the summation at the end. Moreover, for every $i\in \{1, \ldots, l\}$, we check which conjugacy class of $N_W(W_J)/W_J$ gives the largest intersection with $C_i$ and which class gives the smallest intersection in $C_i$, and substituting for $|W_Jw \cap C_i|/|W_J|$ in the expression gives the required bound.
\end{proof}

\begin{rem}
We can use the polynomial functionality in {\sc Magma} to estimate $\chi(x)$ using Proposition~\ref{prop:LLS-3.2} and Corollary~\ref{cor:3.2}.

The benefit of the upper bound in Corollary~\ref{cor:3.2} is that it is not sensitive to the choice of $w$. This is particularly useful when working with groups with socle $E_8(q)$, as the number of possible pairs $(J, [w])$, which are listed in~\cite{FJ2} is extremely large, and it would be an incredibly laborious task to consider them all separately. It turns out that the upper bound from Corollary~\ref{cor:3.2} is sufficient for our purposes, apart from the case $q=2$. When $q = 2$, the possibilities for the centraliser type for $x$ are very restricted and they are all given in~\cite[Section 3]{ABMNPRW}, so we can use Proposition~\ref{prop:LLS-3.2} to compute the precise value of $|C_{\O}(x)|$. 

However, the upper bound in Corollary~\ref{cor:3.2} is not sufficient when $G_0\in \mathcal{E}_2$, in which case we use Proposition~\ref{prop:LLS-3.2} to compute $|C_{\O}(x)|$ precisely in all cases. 
\end{rem}

We finish the discussion on semisimple elements by giving a description of how one can compute fixed point ratios of semisimple elements in {\sc Magma} using Proposition~\ref{prop:LLS-3.2} and Corollary~\ref{cor:3.2}. Suppose that we want to compute the fixed points for a semisimple element $x\in \bar{G}_{\sigma}$ corresponding to the pair $(J, [w])$. The key ingredient is the function \texttt{fprSemisimple} that takes as input a group of Lie type $G$, a set encoding a parabolic subgroup $P$ of $G$ in {\sc Magma} (which in our case will be $H$), a set $J$, where $J$ is as defined above, a polynomial \texttt{poly}, which is the $p'$-part of the centraliser of the element in question, and a word \texttt{wrd} in standard generators for $W$ that encodes $w$ in {\sc Magma}. Below we provide the relevant version of \texttt{fprSemisimple} for the case where $\bar{G}_{\sigma}$ is untwisted. If $\bar{G}_{\sigma}$ is twisted, then one can obtain a similar function with minor modifications. Similarly, if $G_0 = E_8(q)$ and we want to use the approximation of $|C_{\O}(x)|$ given by Corollary~\ref{cor:3.2}, it is possible to do so only with slight modifications of the function provided.

{\small
\begin{verbatim}
fprSemisimple := function(G, P, J, poly, wrd)
W := WeylGroup(G);
id := Identity(W);
R := Roots(G);
Q := RationalField();
L<q> := PolynomialRing(Q, 1);
S := RationalFunctionField(L);
C := Classes(W);
A := TwistedToriOrders(G);
T := [];
for i in [1 .. #A] do
  t := A[i];
  b := 1;
  for l in t[1] do
    b := b*Evaluate(l, q);
  end for;
  Append(~T, b);
end for;
 
E := [];
for i in [1 .. #T] do
  f := Factorization(T[i]);
  e := 0;
  if f[1][1] eq q - 1 then
    e := f[1][2];
  end if;
  Append(~E, e);
end for;
 
WJ := ReflectionSubgroup(W, J);
 
w := wrd;
 
WP := StandardParabolicSubgroup(W, P);
CP := Classes(WP);
 
k := L!poly;
f := Factorization(k);
F := 0;
if f[1][1] eq q - 1 then
  F := f[1][2];
end if;
k := S!k;
 
c := Classes(WP);
A := [0: i in [1 .. #C]];
for i in [1 .. #c] do
  x := c[i][3];
  a := exists(j){j : j in [1 .. #C] | IsConjugate(W, x, C[j][3])};
  A[j] := A[j] + c[i][2];
end for;
 
B := [0: i in [1 .. #C]];
U := [x*w : x in WJ];
for i in [1 .. #U] do
  x := U[i];
  a := exists(j){j : j in [1 .. #C] | IsConjugate(W, x, C[j][3])};
  B[j] := B[j] + 1;
end for;
z := S!0;
for i in [1 .. #C] do
  z := z + (#W/C[i][2]) * (A[i]/#WP) * (B[i]/#WJ) * (-1)^(F+E[i]) * k/T[i];
end for;
 
return z;
 
end function;
\end{verbatim}
}

On the right input, the function \texttt{fprSemisimple} will return $|C_{\O}(x)|$ and in this way one can compute ${\rm fpr}(x, G/H)$.

Let us see a concrete example when $G_0 = E_7(q)$ and $H$ is a $P_1$ parabolic subgroup of $G$. Here we write $\Pi = \{\alpha_1, \ldots, \alpha_7\}$ for the simple roots of $\bar{G}$ and the possible pairs $(J, [w])$ parametrising the $\bar{G}_{\sigma}$-classes are recorded in~\cite{FJ1}. Note that Fleischmann and Janiszczak adopt a labelling of the simple roots of $W$ that is different to the standard Bourbaki labelling~\cite{Bou}. However, the labelling of simple roots in {\sc Magma} agrees with the Bourbaki labelling, so one needs to translate the sets $J$ to Bourbaki labelling, when working with {\sc Magma}. In this paper, we will always adopt the Bourbaki labelling. 

Let $x\in G$ be a semisimple element whose $\bar{G}_{\sigma}$-class corresponds to the pair $(J, [w])$ for $J = \{\alpha_0, \alpha_1, \alpha_2\}$ and $w = a_1a_3$, as given in~\cite[p.123]{FJ1}. Translating the Fleischmann and Janiszczak to Bourbaki notation, we find that $a_1$ and $a_3$ are the reflections corresponding to $\alpha_7$ and $\alpha_5$ in $W$ respectively. We also see that $C_{\bar{G}_{\sigma}}(x)$ is of type $A_2A_1T_4$, and in particular $|C_{\bar{G}_{\sigma}}(x)| = q^4(q^2-1)^4(q^3-1)$. Finally, noting that the $P_1$ parabolic of $G$ is given by \texttt{StandardParabolicSubgroup(W, \{2, 3, 4, 5, 6, 7\})} in {\sc Magma}, we can run the following lines of code:

{\small
\begin{verbatim}
G := GroupOfLieType("E7", 3);
P := {2, 3, 4, 5, 6, 7};
W := WeylGroup(G);
id := Identity(W);
R := Roots(G);
Q := RationalField();
L<q> := PolynomialRing(Q, 1);
S := RationalFunctionField(L);
J := {1, 2, #R};
a1 := Reflection(W, 7);
a3 := Reflection(W, 5);
poly := (q^2-1)^4 * (q^3-1);
fprSemisimple(G, P, J, poly, a1*a3);
\end{verbatim}
}
\noindent and we find that $|C_{\O}(x)| = q^3 + 8q^2 + 9q + 8$. Then dividing by the index of $H$ in $G$ in {\sc Magma} we can also obtain an expression for ${\rm fpr}(x, G/H)$ in $q$. We can perform a similar computation to obtain the fixed point ratio of any semisimple element in $\bar{G}_{\sigma}$ and any parabolic action.

\vs

In the process of extracting results from~\cite{FJ1}, we spotted some minor errors in some places. We take the opportunity to correct them here in case this is useful to the reader.

Let $G$ denote the simply connected group of type $E_7$, and let $x\in G$ be a semisimple element corresponding to the pair $(J, [w])$. First consider the case $J = \{\alpha_0, \alpha_1, \alpha_2, \alpha_4, \alpha_6\}$, so that $C_G(x)$ is of type $A_2^2A_1$. If $w = w_2 = a_5$, as given in~\cite[p.127]{FJ1}, then we need to add an extra ${\rm SU}_3(q)$ factor in $C_G(x)$, so that $C_G(x) = {\rm SU}_3(q)^2\times {\rm SL}_2(q) \times (q+1)^2$.

Next suppose that $J = \{\alpha_0, \alpha_1, \alpha_2, \alpha_3, \alpha_5, \alpha_6, \alpha_7\}$, so that $C_G(x)$ is of type $A_3^2$. In this case there is a mismatch in~\cite{FJ1} between the $w_i$ and the right centraliser structure for $w = w_i$ for $i\in \{4, 6, 7, 8\}$, which can be corrected by an appropriate permutation of the $w_i$. We record the correct assignment of centraliser to each $w_i$ in Table~\ref{t:FJ-correction}.

{\small
\begin{table}
\[
\begin{array}{lll} \hline
i& w_i& C_G(x)\\
\hline
4& a_{49}a_{53}& {\rm SL}_4(q^2)\times (q-1)\\
6& a_5a_{49}a_{53} \\
7& a_{53} & {\rm SU}_4(q)^2\times (q+1)\\
8& a_5a_{53}&  {\rm SU}_4(q)^2\times (q-1)\\
\hline
\end{array}
\]
\caption{Correction to~\cite{FJ1} for $J =  \{\alpha_0, \alpha_1, \alpha_2, \alpha_3, \alpha_5, \alpha_6, \alpha_7\}$.}
\label{t:FJ-correction}
\end{table}
}

Finally, suppose that $C_G(x)$ is of type $(A_1^3)_1$. In this case, by inspecting the Dynkin diagram of $G$, we note that the choice of $J$ in~\cite{FJ1} is not correct. One could set $J = \{\alpha_0, \alpha_3, \alpha_7\}$. We then turn into {\sc Magma} and run the following code:

{\small
\begin{verbatim}
G := GroupOfLieType("E7", 3);
W := WeylGroup(G);
J := {3, 7, 126};
WJ:=ReflectionSubgroup(W, J);
Q, f := quo<Normalizer(W, WJ)| WJ>;
Z := [];
cl := Classes(Q);
for i in [1 .. #cl] do
  Append(~Z, cl[i][3]@@f);
end for;
\end{verbatim}}
\noindent to obtain a list of coset representatives in $N_W(W_J)$ of the $\sigma$-conjugacy classes of $N_W(W_J)/W_J$. The final step is to write those representatives as products of $a_1, \ldots, a_{62}$, as defined in~\cite[pp.118-119]{FJ1}. To do this, we construct $a_1, \ldots, a_{62}$ in {\sc Magma} as permutations and we store them in a list A. We then turn into \textsf{GAP} and we input the lists of permutations $A$ and $Z$ we have obtained. We then run the following code:
{\small
\begin{verbatim}
G := Group(A);
C := [];
for i in [1 .. Length(Z)] do
Add(C, Factorization(G, Z[i]));
od;
\end{verbatim}}
\noindent and we obtain the following list of $w_i$:
\begin{align*}
w_1 &= {\rm id},
w_2 = a_{38}a_{22}a_{12},
w_3 = a_{10}a_5,
w_4 = a_{61}a_{18},
w_5 = a_5,\\
w_6 &= a_{15}a_{10}a_5,
w_7 = a_{15}a_{11},
w_8 = a_{15},
w_9 = a_{21}a_{26}a_{21}a_6,
w_{10} = a_{38}a_{15},\\
w_{11} &= a_{49}a_{37}a_{15}a_5,
w_{12} = a_{33}a_{60}a_{16}a_{14},
w_{13} = a_{53}a_{37}a_{53},
w_{14} = a_{49}a_{37},
w_{15} = a_{53}a_{55}a_{61}a_3,\\
w_{16} &= a_{49}a_{37}a_{11},
w_{17} = a_{55}a_{30}a_{49},
w_{18} = a_{36}a_{37}a_{18},
w_{19} = a_{49}a_{36}a_{37}a_{15},
w_{20} = a_{16}a_{49}a_{26}a_6
\end{align*}
Taking $C_G(x)$ as in~\cite[p.122]{FJ1} for each $i\in \{1, \ldots, 20\}$, we obtain a complete list of the possible pairs $(J, [w])$ when $C_G(x)$ of type $(A_1^3)_1$.
\subsection{Unipotent elements}\label{ss:unipotent}

Let $x\in \bar{G}_{\sigma}$ be a unipotent element and let $W$, $W_{\bar{H}}$, and $\O$ be as in Section~\ref{ss:semisimple}. Note that we again have $|C_{\O}(x)| = \chi(x)$, where $\chi = 1^{\bar{G}_{\sigma}}_{\bar{H}_{\sigma}}$. Now if $\widehat{W}$ denotes the set of (ordinary) irreducible characters of $W$, then~\cite[2.4]{LLS} gives
\begin{equation}\label{eq:unipotent1}
\chi(x) = \sum_{\phi \in \widehat{W}}n_{\phi}R_{\phi}(x) 
\end{equation}
where
\begin{equation}
n_{\phi} = \la 1^W_{W_{\bar{H}}}, \phi\ra = \la 1_{W_{\bar{H}}}, \phi \vert_{W_{\bar{H}}}\ra_{W_{\bar{H}}} = \frac{1}{|W_{\bar{H}}|}\sum_{w\in W_{\bar{H}}} \phi(w)
\end{equation}
and $R_{\phi}(x)$ are the \emph{Foulkes functions} of $\bar{G}_{\sigma}$. The values $n_{\phi}$ are recorded in~\cite[pp. 413--415]{LLS} when $\bar{G}_{\sigma}$ is untwisted and $\bar{H}$ is a maximal parabolic subgroup of $\bar{G}$, and we can find them in~\cite[Table 8]{BT} if $G_0 = E_6(q)$ and $H = P_{1, 6}$, or $P_{3, 5}$, or $G_0 = F_4(q)$ and $H = P_{1, 4}$ or $P_{2, 3}$. Finally, if $G_0 = {}^2E_6(q)$, then the values $n_{\phi}$ are recorded in~\cite[p. 125]{BLS}.

The values $R_{\phi}$ of unipotent elements of order $p$ can be computed using the data on Green functions in~\cite{Lu3} as described in~\cite[Section 4]{Lu4}. Lübeck has conveniently converted Green functions to Foulkes functions in \textsf{GAP}-readable form for us~\cite{Lu}, so we can compute ${\rm fpr}(x, G/H)$ precisely for all groups with $G_0 \in \mathcal{E}_2\cup \mathcal{E}_3$, all unipotent elements and all maximal parabolic subgroups of $G$. 

\subsection{Field and graph-field automorphisms}\label{ss:fields}
Let $x\in G$ be a field or graph-field automorphism of prime order $r$. We have the following proposition:

\begin{prop}\label{prop:LLS-6.1}
Let $\bar{G}_{\sigma} = \bar{G}(q)$, $\bar{H}_{\sigma} = \bar{H}(q)$ and $C_{\bar{G}_{\sigma}}(x) = \bar{G}^{\epsilon}(q^{1/r})$. Then the following hold:
\begin{itemize}
\item [\rm (i)] $x^{\bar{G}_{\sigma}}\cap \bar{H}_{\sigma}x = x^{\bar{H}_{\sigma}}$;
\vs
\item [\rm (ii)] $C_{\bar{H}_{\sigma}} = \bar{H}^{\epsilon}(q^{1/r})$, where $\bar{H}^{\epsilon}(q^{1/r})$ is the corresponding parabolic subgroup of $C_{\bar{G}_{\sigma}}(x)$ and $\epsilon = -$ if $x$ is a graph-field automorphism and $\epsilon = +$ otherwise;
\vs
\item [\rm (ii)] We have
\[
{\rm fpr}(x, G/H) = \frac{|\bar{G}^{\epsilon}(q^{1/r}) \, : \, \bar{H}^{\epsilon}(q^{1/r})|}{|\bar{G}(q) \, : \, \bar{H}(q)|}.
\]
\end{itemize}
\end{prop}

\begin{proof}
(i) is~\cite[7.2]{GL} and (ii) and (iii) are shown directly at the start of the proof of~\cite[6.1]{LLS}.
\end{proof}
We will be using this result repeatedly.
\subsection{Graph automorphisms}

If $G_0 = {}^3D_4(q)$, then the bound in~\cite[Theorem 1]{LLS} will be sufficient for our purposes, so we now consider the case $G_0 = E_6^{\epsilon}(q)$, in which case there are two classes of graph automorphisms, namely those of type $F_4$, and those of type $C_4$. We record an upper bound $t(H, \tau, q)$ for every $H\in \mathcal{P}_G$ and every involutory graph automorphism $\tau\in G$. If $p = 2$, then~\cite[Section 19]{AS} implies that either $C_{\bar{G}}(x) = F_4$ or $C_{F_4}(t)$ for a long root element $t\in F_4$. The value $t(H, \tau, q)$ can be obtained via~\cite[2.6]{LLS} when $C_{\bar{G}}(x) = F_4$, and if $C_{\bar{G}}(x) = C_{F_4}(t)$, then a bound $t(H, \tau, q)$ is recorded in~\cite[Table 7]{BLS}. For the reader's convenience, we record these values in Table~\ref{t:e6:graphs-p2}.

{\small
\begin{table}
\[
\begin{array}{lll} \hline
H & C_{\bar{G}}(\tau)&   t(H, \tau, q)\\ \hline
P_2 & F_4 & (q^6-q^3-1)^{-1}\\
& C_{F_4}(t) & q^{-10}(q-1)^{-1}\\
P_4 & F_4 & q^{-6}(q^2-1)^{-1}(q-1)^{-1}\\
&C_{F_4}(t) & q^{-14}(q-1)^{-2}\\
P_{1,6}& F_4 & q^{-8}(q-1)^{-1}\\
& C_{F_4}(t)  & q^{-12}(q-1)^{-1}\\
P_{3,5}& F_4& q^{-10}(q-1)^{-1}\\
& C_{F_4}(t) & q^{-15}(q-1)^{-2}\\
\hline
\end{array}
\]
\caption{The value $t(H, \tau, q)$ when $p = 2$.}
\label{t:e6:graphs-p2}
\end{table}
}

If $p > 2$, then $C_{\bar{G}}(\tau) = F_4$ or $C_4$. Sufficient fixed point ratio bounds can be obtained by working as in the proof of~\cite[6.4]{LLS}. In the following proposition we carry out these computations to obtain $t(H, \tau, q)$. Note that we will only need an upper bound over all parabolic subgroups for our purposes, but those bounds might be of independent interest, and so we record separate upper bounds for each choice of parabolic subgroup. The value $t(H, \tau, q)$ for each distinct pair $(H, \tau)$ is recorded in Table~\ref{t:e6:graphs}.

\begin{prop}\label{prop:e6-fpr-graph}
Let $G_0 = E_6^{\epsilon}(q)$, where $q$ is odd, and $\tau \in G$ be an involutory graph automorphism. If $H\in \mathcal{P}_G$, then ${\rm fpr}(x, G/H) \leqs t(H, \tau, q)$, where $t(H, \tau, q)$ is recorded in Table~\ref{t:e6:graphs}.
\end{prop}

\begin{proof}
First observe that $H = N_G(P)$, where $P$ is a $\tau$-stable parabolic subgroup of $\bar{G} = E_6$. Moreover, as shown in~\cite[6.4]{LLS}, $C_P(\tau)$ is a parabolic subgroup of $C_{\bar{G}}(\tau)$. We write $P = QL$, where $Q$ is the unipotent radical of $P$ and $L$ is a Levi subgroup, and we treat each possibility for $P$ in turn.

First suppose that $P = P_2$, in which case $L = A_5T_1$. Here $\tau$ acts as a graph automorphism on $L' = A_5$, which implies that $C_{L'}(\tau)$ must either be of type $C_3$ or $A_3$. 

If $C_{\bar{G}}(\tau) = F_4$, then we note that no proper parabolic subgroup of $F_4$ has a Levi factor containing $A_3$, so $C_{L'}(\tau) = C_3$ and thus $C_P(\tau)$ is the $P_1$ parabolic subgroup of $F_4$. This implies that
\[
{\rm fpr}(\tau, G/H) = \frac{|H \, : \, C_H(\tau)|}{|G \, : \, C_G(\tau)|} \leqs \frac{1}{q^6 - q^3+1}.
\]

If $C_{\bar{G}}(\tau) = C_4$, then either $C_{L'}(\tau)$ is of type $C_3$, which corresponds to the $P_1$ parabolic of $C_4$, or $C_P(\tau)$ is of type $A_3$, in which case $C_{L'}(\tau)$ is the $P_4$ parabolic of $C_4$, and so we find ${\rm fpr}(\tau, G/H) < 2q^{-11}$.

Next assume that $P = P_4$, in which case $L = A_2^2A_1T_1$. Here, $\tau$ swaps the two $A_2$ factors of $L'$ and it either centralises the $A_1$ factor, or acts on it as an involution with centraliser $T_1$. It follows that $C_{L'}(\tau)$ is either $A_2A_1$ or $A_2T_1$.

First suppose that $C_{\bar{G}}(\tau) = F_4$. If $C_{L'}(\tau) = A_2A_1$, then $C_P(\tau)$ is the $P_2$ or $P_3$-parabolic subgroup of $F_4$. We therefore get ${\rm fpr}(\tau, G/H) < 2q^{-9}$. On the other hand, if $C_{L'}(\tau) = A_2T_1$, then the Levi factor of $C_P(\tau)$ must contain $A_2$, and we compute ${\rm fpr}(\tau, G/H) < 3q^{-8}$ for all choices of parabolic where this is possible.

Now let us assume that $C_{\bar{G}}(\tau) = C_4$. If $C_{L'}(\tau) = A_2A_1$, then $C_P(\tau)$ must be the $P_3$ parabolic subgroup of $C_4$ and we compute ${\rm fpr}(\tau, G/H) < 2q^{-17}$. On the other hand, if $C_{L'}(\tau) = A_2T_1$, then the Levi factor of $C_{P}(\tau)$ must contain $A_2$, and by checking all possibilities for $C_P(\tau)$, we deduce that ${\rm fpr}(\tau, G/H) < 2q^{-16}$.

We now consider the case $P = P_{1,6}$, in which case $L' = D_4$ and $\tau$ acts as an involutory graph automorphism on $D_4$, and so $C_{L'}(\tau) = B_3$ or $B_2A_1$. 

Since $F_4$ has no parabolic subgroups with Levi factor containing $B_2A_1$, it follows that if $C_{\bar{G}}(\tau) = F_4$, we must have $C_{L'}(\tau) = B_3$, and so $C_P(\tau)$ must be the $P_4$ parabolic of $F_4$. This gives ${\rm fpr}(\tau, G/H) < 2q^{-9}$.

On the other hand, if $C_{\bar{G}}(\tau) = C_4$, then we must have $C_{L'}(\tau) = B_2A_1$, and so $C_P(\tau)$ is the $P_2$ parabolic of $C_4$. From this we get ${\rm fpr}(\tau, G/H) < 2q^{-12}$.

Finally let us assume that $P = P_{3,5}$, so $L' = A_2A_1A_1$. Here $\tau$ swaps the two $A_1$ factors and it either centralises the $A_2$ factor, or acts on it as an involution with centraliser $A_1T_1$. Hence, we must have $C_{L'}(\tau) = A_2A_1$, or $A_1^2T_1$. 

If $C_{\bar{G}}(\tau) = F_4$ and $C_{L'}(\tau) = A_2A_1$, then as above we get ${\rm fpr}(\tau, G/H) < 2q^{-11}$. If $C_{L'}(\tau) = A_1^2T_1$, then the semisimple part of the Levi factor of $C_P(\tau)$ contains $A_1^2$, and by checking all possibilities for which this containment is satisfied, we find ${\rm fpr}(\tau, G/H) < 4q^{-9}$.

If $C_{\bar{G}}(\tau) = C_4$ and $C_{L'}(\tau) = A_2A_1$, then as above we get ${\rm fpr}(\tau, G/H) < 2q^{-19}$. On the other hand, if $C_{L'}(\tau) = A_1^2T_1$, then the semisimple part of the Levi factor of $C_P(\tau)$ must contain $A_1^2$, and by inspecting all possibilities for $C_P(\tau)$ we find that ${\rm fpr}(\tau, G/H) < 3q^{-17}$, as required.
\end{proof}

{\small
\begin{table}
\[
\begin{array}{lll} \hline
H & C_{\bar{G}}(\tau)& t(H, \tau, q)\\ \hline
P_2 & F_4 & (q^6-q^3+1)^{-1}\\
& C_4 & 2q^{-11}\\
P_4 & F_4 & 3q^{-8}\\
& C_4 & 2q^{-16}\\
P_{1,6}& F_4 & 2q^{-9}\\
& C_4 & 2q^{-12}\\
P_{3,5}& F_4& 4q^{-9}\\
& C_4& 3q^{-17}\\
\hline
\end{array}
\]
\caption{The value $t(H, \tau, q)$ when $p$ is odd.}
\label{t:e6:graphs}
\end{table}
}

\section{Proof of Theorem~\ref{thm:exceptional}}\label{s:proof-exceptional}

\subsection{Proof of Theorem~\ref{thm:exceptional} for groups with socle in $\mathcal{E}_1$}\label{s:e1}
In this section we prove Theorem~\ref{thm:exceptional} for the groups with socle in $\mathcal{E}_1$. For any prime order element $x\in G$ our aim is to derive an upper bound $f(x, q)$, where we recall 
\[
f(x, q) = \max \{{\rm fpr}(x, G/H) \, : \, H\in \mathcal{M}_G\}. 
\]
Then noting that 
\begin{equation}\label{eq:mx}
\widehat{Q}(G, \tau) \leqs \sum_{x \in \pi} f(x, q)^6
\end{equation}
where by $\pi$ we denote the set of prime order elements in $G$, we can show that the probability in question is strictly less than some function of $q$ that tends to zero as $q$ tends to infinity and prove the required claim.

In most cases we refer to~\cite{LLS} to obtain an upper bound on $f(x, q)$. In the cases where the bounds in ~\cite{LLS} are not sufficient, sometimes trivial bounds on the orders of core-free maximal subgroups of $G$ are good enough, but in most cases we use bounds derived in \cite{BLS}.
\begin{prop}
The conclusion to Theorem~\ref{thm:exceptional} holds when $G_0 = {}^2G_2(q)'$.
\end{prop}

\begin{proof}
Here $q = 3^{2m + 1}$ for some non-negative integer $m$. In view of the isomorphism ${}^2G_2(3)' \cong {\rm L}_2(8)$, we may assume that $q\geqs 27$. In this case we can infer from~\cite[Theorem 1]{LLS} that ${\rm fpr}(x, G/H) \leqs (q^2-q+1)^{-1}$ for all prime order elements $x\in G$ and $H\in \mathcal{M}_G$, and therefore
\[
\widehat{Q}(G, \tau) < |G|\cdot (q^2-q+1)^{-6} < q^{-1},
\]
since $|G| < q^{7}\log_3(q)$.
\end{proof}

\begin{prop}
The conclusion to Theorem~\ref{thm:exceptional} holds when $G_0 = {}^2B_2(q)$.
\end{prop}

\begin{proof}
Here $q = 2^{m+1}$ for some $m\geqs 1$. Observe that~\cite[Theorem 1]{LLS} asserts that ${\rm fpr}(x, G/H) \leqs (q^{2/3}+1)/(q^2+1)$ for all prime order elements $x\in G$ and $H\in \mathcal{M}_G$, and thus
\[
\widehat{Q}(G, \tau) < |G|\cdot  \left(\frac{q^{2/3}+1}{q^2+1}\right)^{6} < q^{-1} 
\]
since $|G| < q^{5}\log_2(q)$.
\end{proof}

\begin{prop}
The conclusion to Theorem~\ref{thm:exceptional} holds when $G_0 = {}^2F_4(q)'$.
\end{prop}

\begin{proof}
Here $q = 2^{2m + 1}$ for some non-negative integer $m$. The claim can easily be verified using {\sc Magma} for $q = 2$ using the function \texttt{RegOrbitsPlus} given in~\cite{ABcomp}, so we may assume that $q\geqs 8$ for the remainder. In particular, $G_0 = {}^2F_4(q)$.

As noted in~\cite[Section 4.6]{BLS}, we can infer from~\cite{Sh} that $G$ has two classes of involutions with respective representatives $t_2$ and $t'_2$, where
\[
|{t_2}^G| < q^{14} \,\,\, \text{and} \,\,\, |{{t'_2}^{G}}| < q^{11}.
\]
Moreover,~\cite[Theorem 1]{LLS} implies that ${\rm fpr}(x, G/H) \leqs q^{-4}$ for all involutions $x\in G$ and all $H\in \mathcal{M}_G$, so the contribution to $\widehat{Q}(G, \tau)$ from unipotent elements is less than $2q^{14}\cdot (q^{-4})^6 = 2q^{-10}$.

On the other hand, if $x\in G$ is semisimple or a field automorphism of prime order, then~\cite[Theorem 2]{LLS} gives ${\rm fpr}(x, G/H) \leqs 2q^{-6}$, so we have 
\[
\widehat{Q}(G, \tau) < 2q^{-10} + |G|\cdot (2q^{-6})^6 \leqs (2 + 64\log_2q) \cdot q^{-10} <  q^{-1}.
\]
\end{proof}

\begin{prop}
The conclusion to Theorem~\ref{thm:exceptional} holds when $G_0 = {}^3D_4(q)$.
\end{prop}

\begin{proof}
Let $x\in G$ be an element of prime order $r$. We claim that if ${\rm dim} \, x^{\bar{G}} > 20$ or $x$ is a field automorphism of prime order $r\geqs 5$, then ${\rm fpr}(x, G/H) < (q^2(q^3-2))^{-1}$ for any $H\in \mathcal{M}_G$. If $H$ is  parabolic, then the claim follows directly by~\cite[Theorem 2]{LLS}, so assume that $H$ is non-parabolic. If $H_0 = {}^3D_4(q^{1/2})$, then we verify that the bound holds by inspecting the proof of~\cite[4.43]{BLS}, and in all other cases $|H_0| < q^{14}$, so the bound is again satisfied.

Now assume that ${\rm dim} \, x^{\bar{G}} \leqs 20$, or $r\in \{2, 3\}$ and $x$ is either a field or triality graph automorphism. In this case~\cite[Theorem 1]{LLS} implies that ${\rm fpr}(x, G/H) < (q^4-q^2+1)^{-1}$ and we claim that there are at most $3q^{22}$ elements in $G$ with ${\rm dim}\, x^{\bar{G}} \leqs 20$. 

By inspecting~\cite[Table 4.4]{DM}, we deduce that there are at most $q^2 + q+1$ such semisimple classes, and we find that there are $4$ such unipotent classes by inspecting~\cite{Sp}. Moreover, if $x$ is a field automorphism, then either $r = 2$ and we have $|x^G| < 2q^{14}$, or $r = 3$ and $|x^G| < 2q^{56/3} < q^{20}$. Note that $G$ has at most two classes of order $3$ field automorphisms. Finally there are fewer than $4q^{20} + 2q^{14}$ triality graph automorphisms, so in total we get
\[
(q^2+q+11)\cdot q^{20} + 4q^{14} \leqs 3q^{22}
\]
elements with ${\rm dim}\, x^G \leqs 20$.

Hence, we obtain
\[
\widehat{Q}(G, \tau) < 3q^{22} \cdot (q^4-q^2+1)^{-6} + |G|\cdot (q^2(q^3-2))^{-6} < 3q^{-1}
\]
if $q\geqs 5$, since $|G| <  3q^{28}\cdot {\rm log}_2q$.
If $q = 3$, then we note that there are no field automorphisms and so we can replace $|G|$ with $|\bar{G}_{\sigma}|$ in the bound above and we obtain the result.
\end{proof}

\begin{prop}
The conclusion to Theorem~\ref{thm:exceptional} holds when $G_0 = G_2(q)'$.
\end{prop}

\begin{proof}
In view of the isomorphism $G_2(2)' \cong {\rm U}_3(3)$, we may assume that $q\geqs 3$. In particular, $G_0 = G_2(q)$. Let $x\in G$ be a prime order element. If $x$ is an involutory field or graph field automorphism, or a long root element, then \cite[Theorem 1]{LLS} asserts that ${\rm fpr}(x, G/H) \leqs (q^2-q+1)^{-1}$ for any $H\in \mathcal{M}_G$ and we can check that there are at most $q^8$ such elements in $G$. For all other elements, we claim that ${\rm fpr}(x, G/H) \leqs (q^2(q-1))^{-1}$. 

First suppose that $x$ is a field automorphism of order $r \geqs 3$. If $|H_0| \leqs q^6$, then the claim trivially holds, since $|x^G| > \frac{1}{2}q^{28/3}$ as noted in~\cite[4.30]{BLS}, and so 
\[
{\rm fpr}(x, G/H) < \frac{|H|}{|x^G|} < q^{-10/3}\log_2{q}. 
\]
Next, if $H$ is of type ${\rm SL}_3^{\epsilon}(q)$ or ${}^2G_2(q)$, then we confirm the required bound holds by inspecting the proofs of \cite[4.31, 4.32]{BLS} respectively, and if $H$ is of type $G_2(q^{1/2})$, then $x$ acts on $H_0$ as a field automorphism, so we obtain
\[
{\rm fpr}(x, G/H) \leqs \frac{|G_2(q^{1/r}) \, : \, G_2(q^{1/2r})|}{|G_2(q) \, : \, G_2(q^{1/2})|} < q^{-6(1-1/r)} \leqs q^{-4}
\]
and the claim holds.

It now remains to show that the bound holds if $H$ is parabolic. In this case Proposition~\ref{prop:LLS-6.1} implies that 
\[
{\rm fpr}(x, G/H) \leqs \frac{(q^5|{\rm GL_2}(q)|)/(q^{5/r}|{\rm GL}_2(q^{1/r})|)}{q^{14(1-1/r)}} < 2q^{-5(1-1/r)} \leqs 2q^{-10/3}
\]
which is sufficient. For all other possibilities for $x$, the claim follows immediately from~\cite[Theorem 2]{LLS}.

It now follows that 
\[
\widehat{Q}(G, \tau) < q^8 (q^2-q+1)^{-6} + \log_2q \cdot q^{14} \cdot (q^2(q-1))^{-6} < q^{-1}
\]
and the result follows.
\end{proof}

\subsection{Proof of Theorem~\ref{thm:exceptional} for groups with socle in $\mathcal{E}_2$}\label{s:e2}

In this section we prove Theorem~\ref{thm:exceptional} for the groups with $G_0 \in \mathcal{E}_2$. In view of~\eqref{eq:mx}, it suffices to show that 
\[
\sum_{x\in \mathcal{P}}f(x, q)^6 < R(q)
\]
where $\mathcal{P}$ denotes the set of prime order elements in $G$ and $F$ is a function with the properties that $R(q) < 1$ for all $q$, and $R(q) \to 0$ as $q\to \infty$. 

Let $\mathcal{P} = \mathcal{S}\cup \mathcal{U}\cup \mathcal{F}$, where $\mathcal{S}$ denotes the set of semisimple elements of prime order in $G$, $\mathcal{U}$ denotes the class of prime order unipotent elements in $G$, and $\mathcal{F}$ denotes the prime order graph, field, and graph-field automorphisms of $G$. We may then write
\[
\sum_{x\in \mathcal{P}}f(x, q)^6 = \alpha + \beta + \gamma,
\]
where 
\begin{equation}\label{eq:abc}
\alpha = \sum_{x\in \mathcal{S}}f(x, q)^6, \, \, \, \beta = \sum_{x\in \mathcal{U}} f(x, q)^6, \,\,\,\text{and} \,\,\, \gamma = \sum_{x\in \mathcal{F}}f(x, q)^6.
\end{equation}
Our aim is to derive suitable bounds on $\alpha$, $\beta$ and $\gamma$ to prove Theorem~\ref{thm:exceptional} for $G$.

\begin{rem}
Let us now record an important remark about semisimple elements. Note that Fleischmann and Janiszczak work with simply connected groups of type $E_6$ and $E_7$ in~\cite{FJ1}, whereas we are interested in the adjoint type groups. However, for a semisimple element $x\in \bar{G}_{\sigma}$, we are only using their results on centraliser orders to estimate ${\rm fpr}(x, G/H)$. As described in Section~\ref{s:parabolics}, we do this by appealing to Proposition~\ref{prop:LLS-3.2}, and so we are only interested in $(C_{\bar{G}}(x)^0)_{\sigma}$, which is the same both in the adjoint and in the simply connected case.
\end{rem}

\begin{prop}\label{prop:e6}
The conclusion to Theorem~\ref{thm:exceptional} holds when $G_0 = E_6^{\epsilon}(q)$.
\end{prop}

\begin{proof}
For now we will assume that $q\geqs 3$, as $q = 2$ needs separate consideration. We will derive sufficient bounds for $\alpha, \beta$ and $\gamma$, as defined in~\eqref{eq:abc}, and to do this, we need to compute upper bounds on $m_x$ for each prime order element $x\in G$. First note that 
\begin{equation}\label{eq:mx-nx}
f(x, q) \leqs \max \left(\{g(x, q)\}\cup \{{\rm fpr}(x, G/H) \, : \, H\in \mathcal{P}_G\}\right) =: F(x, q),
\end{equation}
where $g(x, q)$ is determined in Proposition~\ref{prop:e6-fpr}. It therefore remains to consider fixed point ratio bounds for parabolic subgroups. 

We first consider semisimple elements. Recall that the possible centralisers of semisimple elements in $\bar{G}_{\sigma}$ are parametrised by pairs $(J, [w])$, where $J$ and $w$ are defined in Section~\ref{ss:semisimple}. The possibilities for $(J, [w])$ are listed in~\cite{FJ1}, and the number of $\bar{G}_{\sigma}$-classes with centraliser type corresponding to a fixed pair $(J, [w])$ is recorded in~\cite{Lu2}. Therefore, if $\{x_1, \ldots, x_t\}$ is a complete set of representatives of $\bar{G}_{\sigma}$-classes of each centraliser type corresponding to a fixed pair $(J, [w])$ and $c_i$ denotes the number of $\bar{G}_{\sigma}$-classes corresponding to a pair $(J, [w])$, then we can write
\begin{equation}\label{eq:semisimple}
\alpha \leqs \sum_{i = 1}^t c_i \cdot |x_i^{\bar{G}_{\sigma}}|\cdot f(x_i, q)^6 \leqs \sum_{i = 1}^t c_i \cdot |x_i^{\bar{G}_{\sigma}}|\cdot F(x_i, q)^6.
\end{equation}
Now recall that using {\sc Magma}, we can compute the exact value of ${\rm fpr}(x, G/H)$ for every parabolic subgroup $H\leqs G$, which is given by the formula in Proposition~\ref{prop:LLS-3.2}, and so we can compute the value of $F(x_i, q)$ for every $i\in [t]$. Moreover, we can read off $c_i$ for every $i\in [t]$ from~\cite{Lu2}, and using {\sc Magma} we can compute the polynomial expression
\begin{equation}\label{eq:e6-semisimple}
\sum_{i = 1}^t c_i \cdot |x_i^{\bar{G}_{\sigma}}|\cdot F(x_i, q)^6
\end{equation}
and check that it is less than $q^{-1}$ for all $q\geqs 3$.

Let us see an example. Assume $\epsilon = +$ and consider the $\bar{G}_{\sigma}$-class corresponding to the pair $(J, [w])$ with $J = \{\alpha_0, \alpha_2,\alpha_5,\alpha_6\}$ and $w = a_1a_2$ as given in~\cite{FJ1}. Let $y$ be a representative of a $\bar{G}_{\sigma}$-class corresponding to this pair. Inspecting~\cite{FJ1} tells us that $C_{\bar{G}}(y)^0$ is of type $A_2^2T_2$, and in particular, $|C_{\bar{G}_{\sigma}}(y)| = \Phi_3(q)|{\rm SL}_3(q)|^2$, where $\Phi_3(q) = q^2-q+1$ denotes the third cyclotomic polynomial. We now turn to~\cite{Lu2} and we find that the entry corresponding to the class of $y$ is the one with label $(i, j, k) = (12, 1, 5)$, and from there we can read the number of classes.

We will now derive an upper bound for $\beta$.  As described in Section~\ref{ss:unipotent}, in this case we can compute precise fixed point ratios for parabolic subgroups of $G$ using the data in~\cite{Lu}, and so using \textsf{GAP}, we can compute the expression 
\[
\beta \leqs \sum_{i = 1}^s |x_i^{\bar{G}_{\sigma}}|\cdot F(x_i, q)^6
\]
where $x_1, \ldots, x_s$ represent the prime order unipotent classes of $G$. In this way, we find that 
\begin{align*}
\beta < 
\begin{cases}
0.43  \text{ if }  q = 3, 4\\
q^{-1} \text{ if } q\geqs 5
\end{cases}
\end{align*}

We finally compute a bound on $\gamma$. Let $\gamma_1$ denote the contribution from involutory field and graph-field automorphisms, $\gamma_2$ denote the contribution from the odd order field automorphisms and $\gamma_3$ denote the contribution from graph automorphisms. 

If $\bar{H}\leqs \bar{G}$ is parabolic and $x$ is a field or graph-field automorphism of prime order $r$, then Proposition~\ref{prop:LLS-6.1} gives
\[
{\rm fpr}(x, G/H) \leqs \frac{|E_6^{\epsilon}(q^{1/r})\, : \,\bar{H}^{\epsilon}(q^{1/r})|}{|E_6(q) \, : \,\bar{H}(q)|}. 
\]
where we recall that $\bar{H}(q) = \bar{H}_{\sigma}$ and $\bar{H}^{\epsilon}(q^{1/r}) = C_{\bar{H}_{\sigma}}(x)$.
Note that this implies that ${\rm fpr}(x, G/H)$ is maximised when $|\bar{H}(q)|$ (and consequently $|\bar{H}^{\epsilon}(q^{1/r})|$) is maximised, and we certainly have $|\bar{H}(q)| \leqs |P_1(q)|$, where $P_1$ is a $P_1$-parabolic subgroup of $E_6$. Moreover, we note that $|{}^2E_6(q)| > |E_6(q)|$ for all $q\geqs 3$. Therefore, if $q\geqs 5$ we find
\[
\gamma_1 \leqs  \frac{3+\epsilon}{2}\cdot \left(\frac{|E_6^{\epsilon}(q)|}{|{}^2E_6(q^{1/2})|}\right)^{-5}\left(\frac{|P_1(q)|}{|P_1(q^{1/2})|}\right)^6
\]
and
\[
\gamma_2 \leqs \sum_{r\in \pi'_{13}} (r-1) \left(\frac{|E^{\epsilon}_6(q)|}{|E^{\epsilon}_6(q^{1/r})|}\right)^{-5}\left(\frac{|P_1(q)|}{|P_1(q^{1/r})|}\right)^6 + q^{78}\log_2{q} \cdot  \left(\frac{|E^{\epsilon}_6(q^{1/17})\, : \,P_1(q^{1/17})|}{|E_6^{\epsilon}(q) \, : \,P_1(q)|}\right)^{6}
\]
where recall that $\pi'_{13}$ denotes the set of odd prime numbers at most $13$. 

Finally, we infer from Propositions~\ref{prop:e6-fpr}, \ref{prop:e6-fpr-q2}, and~\ref{prop:e6-fpr-graph} that if $x$ is an involutory graph automorphism, then $f(x, q) \leqs q^{-5}$ if $C_{\bar{G}}(x) = F_4$ and $f(x, q) \leqs 12q^{-10}$ if $C_{\bar{G}}(x)$ is not of type $F_4$. If $q \geqs 5$, then we find
\begin{align*}
\gamma_3 \leqs q^{26} (q^{-5})^{6} + q^{42}(12q^{-10})^{6}
\end{align*}
One then checks that 
\[
\gamma = \gamma_1 + \gamma_2 + \gamma_3 < q^{-1}.
\]

Finally, if $q = 3$ or $4$, then by only considering the contribution by involutory graph automorphisms (and field and graph-field automorphisms if $q = 4$), we compute $\gamma \leqs 0.025$ using {\sc Magma}.

It follows that $\widehat{Q}(G, \tau) < 1$ for all $q$, and moreover, $\widehat{Q}(G, \tau) < 3q^{-1}$ for all $q\geqs 5$, so $\widehat{Q}(G, \tau) \to 0$ as $q\to \infty$.

\vs

Now assume that $q = 2$. Here, the value $g(x, q)$ can be obtained by Proposition~\ref{prop:e6-fpr-q2}. For parabolic subgroups, the unipotent $\bar{G}_{\sigma}$-involutions and the tables of Foulkes functions are are determined by Lübeck in~\cite{Lu}, so we can compute precise fixed point ratios for involutions in parabolic actions using \textsf{GAP}. 

The odd order semisimple classes are listed in~\cite[Table 9]{BLS}. From this table we can recover the pair $(J, [w])$ corresponding to each semisimple $\bar{G}_{\sigma}$-class via~\cite{FJ1}. More specifically, the centraliser type for each class is listed in~\cite[Table 9]{BLS}, so we can immediately recover $J$ by consulting~\cite{FJ1}, and by inspecting the possible centraliser orders for this particular choice of $J$, we can then recover $w$. 

For example, assume that $\epsilon = +$ and let us look at the fourth row in~\cite[Table 9]{BLS}. In this case our element $y$ has centraliser type $A_3T_3$. We then turn to~\cite{FJ1} and we find that $J = \{\alpha_0, \alpha_2, \alpha_4\}$ and by computing all possible centraliser orders we deduce that $w = a_1a_{17}a_{18}$. In this way we can compute precise fixed point ratios for all prime order semisimple elements, as in the case $q \geqs 3$. 

Finally, sufficient fixed point ratio bounds for involutory graph automorphisms are given by Proposition~\ref{prop:e6-fpr-graph}.

If $\epsilon = -$, then we can show that $\widehat{Q}(G, \tau) < 1$ for any $6$-tuple $\tau$ of $G$ using almost the same strategy as in the case $q\geqs 3$. The only difference is that we are only considering prime order elements when finding upper bounds for $\alpha$ and $\beta$. In particular, \eqref{eq:mx-nx} still holds, where, as mentioned above, $g(x, q)$ is obtained via Proposition~\ref{prop:e6-fpr-q2}. Then using~\eqref{eq:e6-semisimple}, we find that $\alpha \leqs 0.1$ and by using Lübeck's data on unipotent elements we find that $\beta \leqs 0.055$. Finally, the contribution from involutory graph automorphisms is no more than $0.125$, and hence $\widehat{Q}(G, \tau) < 1$.

 The case $\epsilon = +$ requires special attention. As pointed out in~\cite{BLS}, if $G_0 = E_6(2)$ and $\tau$ is a conjugate $6$-tuple containing copies of $P_1$ or $P_6$, then $\widehat{Q}(G, \tau) > 1$. Therefore, it is not possible to prove the claim using the same methods as in the case $q\geqs 3$, so we adopt a slightly different approach. We find that it is possible to construct $P_1$ and $P_6$ in {\sc Magma}. More specifically, we can construct a $P_1$-parabolic using the following code:
 {\small
 \begin{verbatim}
G:=AutomorphismGroupSimpleGroup("E6",2);
T:=Socle(G);
H:=Stabilizer(T,1);
 \end{verbatim}
 }
\noindent and then noting that $P_1$ and $P_6$ are $G_0.2$-conjugate but not $G_0$-conjugate, we can construct $P_6$ using random search. Then again using random search, one can show that every $6$-tuple of $G$ only containing copies of $P_1$ and $P_6$ is regular. We may therefore assume that $\tau$ contains at least one component that is not isomorphic to $P_1$ or $P_6$. 

For $x\in G$, we define 
\[
l(x, q) = \{{\rm fpr}(x, G/H) \, : \, H\in \mathcal{M}_G\setminus \{P_1, P_6\} \}, 
\]
and we note that 
\[
l(x, q) \leqs \max \left(\{g(x, q)\}\cup \{{\rm fpr}(x, G/H) \, : \, H\in \mathcal{P}_G\setminus \{P_1, P_6\}\}\right) =: L(x, q)
\]
Further, observe that our assumption on $\tau$ gives
\[
\widehat{Q}(G, \tau) \leqs \sum_{x\in \mathcal{P}}f(x, q)^5l(x, q) \leqs \alpha' + \beta' + \gamma',
\]
where 
\[
\alpha' = \sum_{x\in \mathcal{S}}f(x, q)^5l(x, q), \, \, \, \beta' = \sum_{x\in \mathcal{U}} f(x, q)^5l(x, q), \,\,\,\text{and} \,\,\, \gamma' = \sum_{x\in \mathcal{F}}f(x, q)^5l(x, q).
\]
We can now estimate $\alpha', \beta', \gamma'$ from above using the methods described above, and we find that $\alpha' \leqs 0.053$, $\beta' \leqs 0.6$, and $\gamma' \leqs 0.125$, so $\widehat{Q}(G, \tau) < 1$ and the proof is complete.
\end{proof}

\begin{prop}\label{prop:e7}
The conclusion to Theorem~\ref{thm:exceptional} holds if $G_ 0 = E_7(q)$.
\end{prop}

\begin{proof}
We will again estimate $\alpha, \beta, \gamma$, as defined in~\eqref{eq:abc} to derive a bound on $\widehat{Q}(G, \tau)$. Our approach here is almost the same as in the case $G_0 = E_6^{\epsilon}(q)$ for $q\geqs 3$. In particular, \eqref{eq:mx-nx} applies here as well, where $n_x$ is obtained this time via Proposition~\ref{prop:e7-fpr} for each choice of prime order element $x\in G$. 

We first consider semisimple elements. The pairs $(J, [w])$ encoding the possible centraliser types of semisimple elements in $\bar{G}_{\sigma}$ are again determined in~\cite{FJ1} and so we can compute precise fixed point ratios for each centraliser type using {\sc Magma}, as described in Section~\ref{ss:semisimple}. Now let $x_1, \ldots, x_t$ be representatives of the semisimple $\bar{G}_{\sigma}$-classes of each centraliser type and $c_i$ be the number of $\bar{G}_{\sigma}$-classes of centraliser type $C_{\bar{G}_{\sigma}}(x_i)$ for each $i\in [t]$. Then we get
\[
\alpha \leqs \sum_{i = 1}^t c_i \cdot |x_i^G|\cdot f(x_i, q)^6 \leqs \sum_{i = 1}^t C_i \cdot |x_i^G|\cdot f(x_i, q)^6
\]
where $C_i$ is an upper bound on $c_i$ for all $i\in [t]$, so it is left to compute sufficient bounds $C_i$ for all $i\in [t]$ to estimate $\alpha$. Note that the exact number of classes corresponding to a given pair $(J, [w])$ is given to us by~\cite{Lu2}. However, crude bounds are good enough for most cases, and due to the large number of possible pairs, we will be working with upper bounds instead.

We find in~\cite{Lu2} that the total number of semisimple classes in $\bar{G}_{\sigma}$ is no more than $q^7 + q^4$ and unless $C_G(x)$ has an $E_6, D_6$, or $E_7$ factor, it suffices to take $C_i = q^7 + q^4$. In the final few cases we can estimate $c_i$ by inspecting~\cite{Lu2} and we give a bound $C_i$ in Table~\ref{t:e7-c_i}. In particular, we record the type of $C_{\bar{G}}(x_i)$, the set $J$ and $w$ and an upper bound on the number of classes of such type. Recall that we use the standard Bourbaki labelling~\cite{Bou} when recording the set $J$ in this paper, instead of the notation used in~\cite{FJ1}. Now using {\sc Magma}, we can compute the expression
\[
\alpha \leqs \sum_{i = 1}^t C_i \cdot |x_i^G|\cdot F(x_i, q)^6
\]
and we find that $\alpha < q^{-1}$, and in particular $\alpha \leqs 0.25$ if $q = 2$.

For unipotent elements, we can again use the \textsf{GAP}-readable data provided by Lübeck in~\cite{Lu} to compute precise fixed point ratios. If $\{x_1, \ldots, x_r\}$ is a complete set of representatives of the $\bar{G}_{\sigma}$-classes of elements of order $p$, then using \textsf{GAP} we compute
\[
\beta \leqs \sum_{i = 1}^r |x_i^G|\cdot F(x_i, q)^6 <
\begin{cases}
0.61 \, \text{ if } \, q = 2\\
1/q \, \text{ if } \, q\geqs 3
\end{cases}
\]

Finally, if $x$ is a field automorphism of prime order $r$, then sufficient fixed point ratio bounds on non-parabolic subgroups are obtained in Proposition~\ref{prop:e7-fpr}, and if $H\leqs G$ is parabolic, then Proposition~\ref{prop:LLS-6.1} gives 
\[
{\rm fpr}(x, G/H) \leqs \frac{|E_7(q^{1/r}) \, : \, \bar{H}(q^{1/r})|}{|E_7(q) \, : \, \bar{H}(q)|},
\]
where $\bar{H}(q)$ and $\bar{H}(q^{1/r})$ are defined as in Proposition~\ref{prop:LLS-6.1}. By computing a bound for each parabolic subgroup and comparing with $g(x, q)$, we find that $f(x, q)$ is achieved when $\bar{H}$ is the $P_7$ parabolic of $\bar{G}$ for all $r$, and so we get
\[
\gamma \leqs \sum_{r\in \pi_{11}} (r-1)\left(\frac{|E_7(q)|}{|E_7(q^{1/r})|}\right)^{-5}\left(\frac{|P_7(q)|}{|P_7(q^{1/r})|}\right)^6 + q^{133}\log_2{q} \cdot  \left(\frac{|E_7(q^{1/13})\, : \,P_7(q^{1/13})|}{|E_7(q) \, : \,P_7(q)|}\right)^{6} < q^{-1}
\]
for all $q\geqs 4$.

It follows that $\widehat{Q}(G, \tau) < 1$ for all $q$, with $\widehat{Q}(G, \tau) < 3q^{-3}$ for all $q\geqs 3$.
\end{proof}

{\small
\begin{table}
\[
\begin{array}{llll} \hline
\text{Type of $C_{\bar{G}_{\sigma}}(x_i)$} & J & w& C_i \\ \hline
E_6T_1&\{\alpha_1, \alpha_2, \alpha_3, \alpha_4, \alpha_5, \alpha_6\}&w_1 & q + 2\\
&& w_2& q+2\\
D_6T_1 & \{\alpha_0, \alpha_1, \alpha_2, \alpha_3, \alpha_4, \alpha_5\}& w_1& q/2\\
&& w_2 & q/2\\
D_6A_1&\{\alpha_0, \alpha_1, \alpha_2, \alpha_3, \alpha_4, \alpha_5, \alpha_7\}& w_1 & 1\\
\hline
\end{array}
\]
\caption{The value $C_i$ for $E_7(q)$-classes that contain an $E_6$ or $D_6$ factor, as defined in Proposition~\ref{prop:e7}.}
\label{t:e7-c_i}
\end{table}
}

\begin{prop}
The conclusion to Theorem~\ref{thm:exceptional} holds when $G_0 = F_4(q)$.
\end{prop}

\begin{proof}
We use the same approach we did in the proof of Proposition~\ref{prop:e6} when $q\geqs 3$. In particular, \eqref{eq:mx-nx} holds in this case as well, where $g(x, q)$ is obtained this time via Proposition~\ref{prop:f4-fpr} for each choice of prime order element $x\in G$, so we can estimate $\alpha, \beta$, and $\gamma$, as defined in~\eqref{eq:abc} to show that $\widehat{Q}(G, \tau) < 1$. 

First let us consider semisimple elements. In this case, the pairs $(J, [w])$ encoding the possible centraliser types of semisimple elements can be found in~\cite[Table 23]{AHL} and the number of classes of each centraliser type can be found in~\cite{Lu2} as usual, so we can compute a bound on $\alpha$ via~\eqref{eq:e6-semisimple}, and we find that $\alpha < q^{-1}$.

The contribution to $\widehat{Q}(G, \tau)$ from unipotent elements can be estimated in the usual way. Sufficient fixed point ratio bounds for non-parabolic subgroups can be obtained via Proposition~\ref{prop:f4-fpr} and \textsf{GAP}-readable data that allow us to compute precise fixed point ratios for parabolic subgroups are again given to us by Lübeck~\cite{Lu}. Therefore, if $\{x_1, \ldots, x_r\}$ is a complete list of representatives for the conjugacy classes of prime order unipotent elements in $G$, we can compute the expression
\[
\sum_{i = 1}^r |x_i^G|\cdot F(x_i, q)^6
\]
using \textsf{GAP} and we deduce that $\beta < q^{-1}$.

It is only left to obtain a sufficient bound for $\gamma$. Let $x$ be a field or graph-field automorphism of prime order $r$. We have already obtained a sufficient fixed point ratio bound for non-parabolic subgroups in Proposition~\ref{prop:f4-fpr}, and for a parabolic subgroup $H$ of $G$ we can once again appeal to Proposition~\ref{prop:LLS-6.1} to get
\[
{\rm fpr}(x, G/H) \leqs \frac{|F_4^{\epsilon}(q^{1/r}) \, : \, \bar{H}^{\epsilon}(q^{1/r})|}{|F_4(q) \, : \, \bar{H}(q)|},
\]
where $\bar{H}(q)$ and $\bar{H}(q^{1/r})$ are defined as in Section~\ref{prop:LLS-6.1}. In this way, we find that $f(x, q) \leqs q^{-6}$ if $r = 2$, whilst $f(x, q)$ is achieved when $\bar{H} = P_1$ if $r\geqs 3$. Noting that $|x^G| < 2q^{26}$ if $x$ is an involutory field or graph-field automorphism, we find
\[
\gamma \leqs 4q^{26}(q^{-6})^6 + \sum_{r\in \pi'_{17}} (r-1)\left(\frac{|F_4(q)|}{|F_4(q^{1/r})|}\right)^{-5}\left(\frac{|P_1(q)|}{|P_1(q^{1/r})|}\right)^6 + 2q^{52}\log_2{q} \cdot  \left(\frac{|F_4(q^{1/19})\, : \,P_1(q^{1/19})|}{|F_4(q) \, : \,P_1(q)|}\right)^{-6} < q^{-1},
\]
where we recall that $\pi'_{17}$ is the set of odd primes at most 17. Therefore $\widehat{Q}(G, \tau) < 1$ and if $q\geqs 3$, then $\widehat{Q}(G, \tau) < 3q^{-1}$.
\end{proof}

\subsection{Proof of Theorem~\ref{thm:exceptional} for groups with socle in $\mathcal{E}_3$}\label{s:e3}

In this final section we prove Theorem~\ref{thm:exceptional} for groups with socle $G_0 = E_8(q)$, which completes the proof of Theorem~\ref{thm:main}.

\begin{prop}\label{prop:e8}
The conclusion to Theorem~\ref{thm:exceptional} holds if $G_0 = E_8(q)$.
\end{prop}

\begin{proof}
We will divide the prime order elements of $G$ into five collections $\mathcal{L}_1, \ldots, \mathcal{L}_5$. In particular, we have $x\in \mathcal{L}_i$ if and only if $x\in \bar{G}_{\sigma}$ and $a_i \leqs {\rm dim} \, x^{\bar{G}} < b_i$, or $x$ is a field automorphism with $q^{a_i} <  |x^{G_0}| < q^{b_i}$, where $a_i$ and $b_i$ are listed in Table~\ref{t:a_i-d_i} for all $i\in \{1, \ldots, 5\}$. We claim that if $x\in \mathcal{L}_i$, then $f(x, q) \leqs c_i$, where $c_i$ is listed in Table~\ref{t:a_i-d_i} for all $i$. Moreover, we claim there are no more than $d_i$ prime order elements in $\mathcal{L}_i$ for all $i\in \{1, \ldots, 5\}$, where $d_i$ is again listed in Table~\ref{t:a_i-d_i} for all $i$. This would in turn imply that 
\[
\widehat{Q}(G, \tau) < \sum_{i = 1}^5 d_ic_i^6 < q^{-1}
\]
which proves the claim. 

First let us record some lower bounds on $|x^{G_0}|$ for a prime order element $x\in G$, which we will need below. In particular, we claim that if $x\in \mathcal{L}_i$ for some $i\in \{1, \ldots, 5\}$, then $q^{88}/c_i < |x^{G}| < q^{b_i}$. All cases are similar, so we look at the case $i = 2$, and we leave the remaining cases for the reader to check. If $x\in \mathcal{L}_2$ is a field automorphism, then the claim follows straight from the definition, so we may assume that $x\in \bar{G}_{\sigma}$. First suppose that $x$ is semisimple. In this case, inspecting~\cite{FJ2} tells us that the possible centraliser types for $x$ are 
\[
D_5, A_4^2, A_6, D_5A_1, A_6A_1, D_5A_1^2, D_5A_2, D_5A_3, A'_7, A''_7 \text{ and } A_7A_1
\]
and from this we deduce that 
\[
|x^{G_0}| \geqs \frac{|E_8(q)|}{|{\rm SU_8(q)}||{\rm SL}_2(q)|} > q^{139}(q^{28}+q^{27}+\dots +q^2+1)(q^3-1)(q^7-1)(q^5-1) > q^{125}.
\]
Similarly, we find that 
\[
|x^{G_0}| \leqs \frac{|E_8(q)|}{|{\rm SO}_{10}^{-}(q)|(q-1)^3} < 2q^{200}
\]
by Lemma~\ref{lem:number-theory}, so the claim follows.

Now suppose that $x$ is unipotent. In this case inspecting~\cite[Table 22.2.1]{LSeitz2} tells us that the smallest class of elements in $\mathcal{L}_2$ is $A_3A_2A_1$ and the largest is $D_5$. We therefore find
\[
|x^{G_0}| = \frac{|E_8(q)|}{q^{60}|{\rm SL}_2(q)|^2} > q^{125}
\]
Similarly, we get
\[
|x^{G_0}| \leqs \frac{|E_8(q)|}{q^{27}|{\rm Sp}_6(q)|} < q^{202}
\]
\vs

We now show that $f(x, q)\leqs c_i$ for all $i\in \{1, \ldots, 5\}$. Fix an $i\in \{1, \ldots, 5\}$, a prime order element $x\in G$ such that $a_i \leqs {\rm dim} \, x^{\bar{G}} < b_i$ and $H\in \mathcal{M}_G$. If $i = 5$, then~\cite[Theorem 1]{LLS} implies that ${\rm fpr}(x, G/H) \leqs 2q^{-24} = c_5$, so we may assume that $i \leqs 4$. First assume that $H$ is non-parabolic. If $|H| \leqs q^{88}$, then the claim follows from the bounds on $|x^G|$ above, as ${\rm fpr}(x, G/H) \leqs |H|/|x^G| \leqs c_i$. 

Now assume that $|H| > q^{88}$. The possibilities for $H$ are described in~\cite[4.2]{BLS} and one can check that ${\rm fpr}(x, G/H) \leqs c_i$ by inspecting the proofs of~\cite[4.4--4.6]{BLS}. 

Finally assume that $H$ is parabolic. If $x$ is unipotent, then using~\cite{Lu}, we can compute precise fixed point ratios in \textsf{GAP} for each conjugacy class and verify the claim. Now suppose that $x$ is semisimple. If $q\geqs 3$, then using {\sc Magma}, we compute the upper bound for $\chi(x)$ derived in Corollary~\ref{cor:3.2}, we verify that ${\rm fpr}(x, G/H) \leqs c_i$ in all cases. If $q = 2$, then the prime order semisimple classes of $G$ are listed in \cite[Tables 3--15]{ABMNPRW}, and by calculating precise fixed point ratios for each class, we again deduce that ${\rm fpr}(x, G/H) \leqs c_i$ in all cases.

Finally, if $x$ is a field automorphism, then Proposition~\ref{prop:LLS-6.1} gives us
\[
{\rm fpr}(x, G/H) \leqs \frac{|E_8(q^{1/r})\, : \,\bar{H}(q^{1/r})|}{|E_8(q) \, : \,\bar{H}(q)|} < c_i.
\]
where $\bar{H}(q)$ and $\bar{H}(q^{1/r})$ are defined as in Section~\ref{prop:LLS-6.1}. It now remains to show that the number of prime order elements $x\in G$ with $a_i\leqs {\rm dim}\, x^{\bar{G}} < b_i$ is bounded above by $d_i$ for all $i\in \{1, \ldots, 5\}$. We treat each possibility of $i$ in turn.
\vs

\noindent \textbf{Case 1.} $i = 1$:

\vs

In this case $d_1 \leqs |G|$, so the claim about $d_1$ is trivially true.

\vs

\noindent \textbf{Case 2.} $i = 2$:

\vs

Here, we either have $x\in \bar{G}_{\sigma}$  or $x$ is a field automorphism of order $5$.  It follows from~\cite{Lu2} that $G$ has fewer than $2q^8$ semisimple classes and we verify using~\cite{Lu} that there are exactly 24 unipotent classes with $182 \leqs {\rm dim}\, x^{\bar{G}} < 202$. Finally, there are no more than 4 classes of field automorphisms of order $5$, and $|x^{G_0}| < q^{202}$, as shown above, so we get 
\[
q^{202} (2q^8 + 28) < 3q^{210} = d_2
\]
as required.

\vs

\noindent \textbf{Case 3.} $i = 3$:

\vs 

Here, we either have $x\in \bar{G}_{\sigma}$  or $x$ is a field automorphism of order $3$. By inspecting \cite{Lu2}, we deduce that there are at most $2q^2 + 3q + 6$ semisimple classes with $154 \leqs {\rm dim} \, x^{\bar{G}} < 182$ and it follows from~\cite{Lu} that there are exactly 11 unipotent classes whose dimension is in this range. Finally, there are no more than $2$ classes of involutory field automorphisms. We've also shown above that $|x^G| < q^{182}$, and so we get
\[
q^{182}(2q^2 + 3q + 19) < q^{188} = d_3
\]
as claimed.

\vs

\noindent \textbf{Case 4.} $i = 4$:

\vs 

By inspecting~\cite{Lu} and~\cite{Lu2} one deduces that there are only $4$ unipotent classes with dimension in this range. Moreover, one can verify that $|x^G| < q^{154}$ for all such classes, and so we get $ 4q^{154} < q^{156} = d_4$, as claimed.

\vs

\noindent \textbf{Case 5.} $i = 5$:

\vs

Inspecting~\cite{Lu2} tells us that there are at most $q+2$ semisimple classes such that ${\rm dim}\, x^{\bar{G}} < 129$ and it follows from~\cite{Lu} that there are $6$ such unipotent classes. Finally, there is a unique class of involutory field automorphisms in $G$, and $|x^G| < q^{129}$ if ${\rm dim}\, x^{\bar{G}} < 129$, so we find
\[
q^{129}(q + 9) < q^{133} = d_5
\]
which completes the proof.

{\small
\begin{table}
\[
\begin{array}{llll} \hline
a_i & b_i & c_i & d_i \\ \hline
202 & 242& q^{-42}& q^{248}\log_2(q)\\
182& 202& q^{-37}& 2q^{211}\\
154 & 182& q^{-32}&q^{188}\\
129 &154& q^{-29}&q^{156}\\
58 &129& 2q^{-24}&q^{133}\\
\hline
\end{array}
\]
\caption{The values $a_i, b_i, c_i, d_i$ as defined in Proposition~\ref{prop:e8}.}
\label{t:a_i-d_i}
\end{table}
}
\end{proof}

\end{document}